\documentclass[a4paper]{article}
\usepackage[scale=0.85]{geometry}

\usepackage[longnamesfirst]{natbib}
\usepackage{lipsum}
\usepackage{amsfonts}
\usepackage{graphicx}
\usepackage{epstopdf}
\usepackage{algorithm}%
\usepackage{algorithmicx}%
\usepackage{algpseudocode}%
\usepackage{mathtools}
\usepackage{enumitem}
\usepackage{nicefrac}
\usepackage{booktabs}
\usepackage{multirow}
\usepackage{float}
\usepackage[rightcaption]{sidecap}
\usepackage{placeins}
\usepackage{amsopn}
\usepackage{caption}
\usepackage{amsmath}
\usepackage{amssymb}
\usepackage{amsthm}
\usepackage{authblk}
\usepackage{xcolor}
\usepackage{xspace}
\usepackage{etoolbox}

\usepackage{setspace}
\definecolor{darkblue}{rgb}{0,0,0.5}
\usepackage[colorlinks=true,linkcolor=darkblue,citecolor=darkblue]{hyperref}
\usepackage[capitalise,nameinlink,noabbrev]{cleveref}

\crefname{hypothesis}{Hypothesis}{Hypotheses}
\crefname{fact}{Fact}{Facts}

\providecommand{\orcidauthor}[2]{}
\providecommand{\emailauthor}[2]{}
\providecommand{\NewLabel}[2]{}

\newenvironment{keywords}{\par\small\noindent\textbf{Keywords: }\ignorespaces}{\par}
\newcommand{\sep}{; }

\algrenewcommand{\algorithmicensure}{\textbf{Return:}}

\newcommand{\REVISORONE}[1]{\textcolor{green!60!black}{#1}}
\renewcommand{\REVISORONE}[1]{#1}

\newcommand{\REVISORTWO}[1]{\textcolor{blue}{#1}}
\renewcommand{\REVISORTWO}[1]{#1}
\newcommand{\Error}[1]{\textcolor{red}{#1}}
\renewcommand{\Error}[1]{#1}
\definecolor{MODpurple}{RGB}{128,0,170}
\newcommand{\MOD}[1]{\textcolor{MODpurple}{#1}}
\renewcommand{\MOD}[1]{#1}

\newcommand\restr[2]{{% we make the whole thing an ordinary symbol
  \left.\kern-\nulldelimiterspace % automatically resize the bar with \right
  #1 % the function
  \vphantom{\big|} % pretend it's a little taller at normal size
  \right|_{#2} % this is the delimiter
  }}

\theoremstyle{plain}
\newtheorem{theorem}{Theorem}[section] % Se numera por sección (ej. Theorem 1.1)
\theoremstyle{definition}
\newtheorem{definition}{Definition}[section]

\theoremstyle{remark}
\newtheorem{remark}{Remark}[section]

\def\tsc#1{\csdef{#1}{\textsc{\lowercase{#1}}\xspace}}
\tsc{WGM}
\tsc{QE}
\tsc{EP}
\tsc{PMS}
\tsc{BEC}
\tsc{DE}
\numberwithin{equation}{section}

\begin{document}

\let\WriteBookmarks\relax
\def\floatpagepagefraction{1}
\def\textpagefraction{.001}
\title{Well-Balanced Schemes for Hyperbolic Kinetic Relaxation}

\renewcommand{\Affilfont}{\normalsize\itshape}
\setlength{\affilsep}{1em}

\author[1]{León Miguel Ávila León\thanks{ORCID: 0009-0003-8458-7020. Corresponding author: \href{mailto:leonavilaleon@uma.es}{leonavilaleon@uma.es}.}}
\author[1]{Manuel Jesús Castro Díaz\thanks{ORCID: 0000-0003-3164-7715. Email: \href{mailto:mjcastro@uma.es}{mjcastro@uma.es}.}}
\author[1]{Jose María Gallardo\thanks{ORCID: 0000-0002-6586-178X. Email: \href{mailto:jmgallardo@uma.es}{jmgallardo@uma.es}.}}
\author[2]{Victor Michel-Dansac\thanks{ORCID: 0000-0002-3859-8517. Email: \href{mailto:victor.michel-dansac@inria.fr}{victor.michel-dansac@inria.fr}.}}

\affil[1]{Departamento de Análisis Matemático, Estadística e I.O. y Matemática Aplicada, Facultad de Ciencias, Universidad de Málaga, 29010 Málaga, Spain}
\affil[2]{Université de Strasbourg, CNRS, Inria, IRMA, F-67000 Strasbourg, France}

\date{\today}

\maketitle

\begin{abstract}
    This paper introduces a novel class of well-balanced numerical schemes for hyperbolic systems of balance laws in 1D using the kinetic relaxation framework. Kinetic relaxation methods approximate nonlinear hyperbolic systems via a linear system of transport equations where nonlinearities are relegated to the source term. While these methods simplify the Riemann problem, they fail to preserve the discrete steady-state solutions of the original system because of the modified source term. We propose two distinct strategies to achieve the well-balanced property: a modified Finite Volume (FV) approach using well-balanced reconstruction operators and a well-balanced Semi-Lagrangian~(SL) approach based on characteristic backtracking. Both methods ensure that the steady solutions of the original system remain invariant during the transport and projection steps. We demonstrate the efficacy of these schemes with first-, second-, and third-order accuracy using palindromic splitting techniques. Numerical results for Burgers' equation, the Shallow Water Equations with bathymetry, and the Euler equations with gravitational potential confirm that our relaxation schemes maintain steady states to machine precision while capturing complex transient dynamics.
\end{abstract}

\begin{keywords}
    Kinetic Relaxation \sep Well-Balanced \sep Hyperbolic System of Balance Laws\sep Finite Volume\sep Semi-Lagrangian
\end{keywords}

\bigskip

\section{Introduction}
Let us consider a 1D hyperbolic system of balance laws of the form
\begin{equation}
    \label{balance law}
    \partial_t u + \partial_x F(u) = S(u,x).
\end{equation}

Here,
\begin{itemize}[nosep]
    \item $u: (x, t) \mapsto u(x, t)$ takes values on an open convex set $\Omega \subset \mathbb{R}^m$;
    \item $F : \Omega \to \mathbb{R}^m$ is the physical flux function, assumed to have sufficient regularity and to be such that the system is hyperbolic (i.e., the Jacobian matrix $\nabla_u F(u)$ is diagonalizable with real eigenvalues for all $u \in \Omega$);
    \item and the source term is written in the form $S(u,x)$, where $S : \Omega\times \mathbb{R} \to \mathbb{R}^m$.
\end{itemize}
\medskip
The previous problem poses several theoretical and computational challenges. A primary difficulty inherent in these systems is the loss of regularity of solutions; when the physical flux $F$ is nonlinear, discontinuities such as shock waves develop in finite time, even when starting from smooth initial conditions. Capturing these weak solutions without introducing spurious oscillations or excessive numerical diffusion requires robust high-resolution methods. Beyond the treatment of discontinuities, a highly desirable feature in the design of numerical schemes for balance laws is the Well-Balanced (WB) property. A scheme is considered well-balanced if it is capable of exactly preserving the steady-state solutions of the system, where the flux gradient and the source term reach an exact equilibrium, down to the order of machine precision. The preservation of such equilibria is fundamental across a wide range of engineering and geophysical applications.
For instance, in geophysical fluid dynamics, it is imperative to maintain the lake at rest state in the Shallow Water Equations (SWE); in astrophysics simulations, hydrostatic equilibria in atmospheric and astrophysical models under gravitational potentials have to be preserved.
Schemes lacking the WB property generate truncation errors of the same order as the mesh size. While these errors vanish as the mesh is refined, they are often orders of magnitude larger than the small physical perturbations being studied, such as seismic waves in the ocean or subtle pressure variations in a star, thereby rendering the numerical results invalid over long-time integration.
\medskip
\newline
On the one hand, the development of well-balanced schemes has been an active area of research in computational fluid dynamics over the past three decades.
Starting with the seminal work of Berm{\'u}dez \& V{\'a}zquez \cite{BerVaz1994} and Greenberg \& LeRoux \cite{GreLeR1996}, who introduced the concept of well-balanced schemes for the SWE, a variety of techniques have been developed to achieve the WB property. These include hydrostatic reconstruction methods~\cite{AudBouBriKlePer2004}, high-order finite volume methods~\cite{GalParCas2007}, or high-order discontinuous Galerkin methods~\cite{XinZhaShu2010}.
Beyond these initial works, mostly focused on preserving steady solutions at rest of the SWE, the literature on well-balanced schemes has expanded significantly, encompassing a wide range of balance laws and applications.
For instance, we mention a series of works dedicated to preserving moving steady solutions of the SWE, like~\cite{XinShuNoe2011,MicBerClaFou2016}, or of other systems, see e.g.~\cite{MicBerClaFou2017} for the SWE with friction,~\cite{DesMas2021} for the SWE with Coriolis force, or~\cite{BerMicTho2026} for the Euler equations with gravitational potential.
Another avenue of improvement concerned high-order accuracy, with the design of high-order well-balanced reconstruction operators, see e.g.~\cite{castro2008well,castro2020well}, and its extensions, see e.g.~\cite{gomez2021collocation,gomez2021high,gomez2021implicit}. We also mention high-order well-balanced methods based on other techniques or applied to other systems of equations, see e.g.~\cite{XinShu2012,FraMen2016,BriXin2020,BerBulFouMbaMic2022,AbgLiu2024,DumZanGabPes2024,FraMicNav2024,KazParRic2025}. Finally, regarding finite difference approximations, pioneering well-balanced schemes were introduced by Xing and Shu~\cite{XingShu2005}, a framework later extended to other balance laws, see e.g.~\cite{WangShuYeeSjogreen2009}.

\medskip

On the other hand, to help with the treatment of nonlinearities in the flux, we focus on the kinetic relaxation approach. It was pioneered in \cite{Bou1999} by Bouchut and in \cite{AuPe00} by Aregba-Driollet and Natalini, and it offers an elegant alternative to solving the nonlinear hyperbolic balance law. By introducing a set of linear transport equations with a stiff relaxation source term, the nonlinearity of the flux of the nonlinear system is effectively relaxed towards the source term of the relaxation system. This simplifies the characteristic structure of the new system (as it becomes made of linear transport equations with constant velocities). This makes the kinetic system computationally efficient and naturally suited for parallelization, at the cost of additional unknowns. However, a difficulty in such methods is the preservation of the steady states of the original system, as the stationary solutions of the kinetic system do not coincide with those of the balance law.
\REVISORONE{Let us also mention that relaxation and kinetic models are a classical tool to design approximate Riemann solvers: a linear transport step without collisions ($\tau=\infty$) followed by an instantaneous projection onto the equilibrium ($\tau=0$) defines a numerical flux for the original system, see e.g.~\cite{JinXin1995,AuPe00,Bouchut2004}. Within this transport-projection framework, well-balanced schemes have been obtained by embedding the source term into the relaxation or kinetic Riemann solver itself, see e.g.~\cite{GreLeR1996,Gosse2001,PerSim2001,Xu2002,Bouchut2004,BerCha2016,MicBerClaFou2016}. The strategy followed in this paper is different: the source term is embedded in the reconstruction operator and in the relaxation stage.}
\medskip
\newline
In this article, we propose a way to make kinetic relaxation techniques well-balanced. The novelty of our approach lies in the development of two distinct methodologies to ensure the WB property. In the {Finite Volume Relaxation approach}, we employ high-order reconstruction operators (up to third-order accurate) that are specifically designed to preserve the steady states of the initial hyperbolic system, even while evolving the relaxed linear variables. In the {Semi-Lagrangian Relaxation approach}, we use the method of characteristics to solve the transport stage. The well-balanced property is then achieved by carefully choosing the reconstruction at the foot of the characteristic curves and modifying the evolution of the original variable based on the physical source term. To validate these approaches, we address explicit (1st, 2nd, and 3rd order) and implicit (1st and 2nd order) formulations for FV, and a SL formulation (1st order). For higher-order temporal accuracy, we implement Strang splitting and multi-stage time discretizations. The robustness of these schemes is validated through classical test cases: the Burgers' equation with a nonlinear source term, the SWE with non-flat bathymetry (where moving steady states are preserved), and the Euler equations under a gravitational potential.
\medskip
\newline
The paper is organized as follows. \Cref{sec:preliminaries} first describes the required preliminaries on kinetic relaxation and on the well-balanced property. Then, the main contributions of the article are presented in \cref{sec:WB_kinetic_relaxation}. Namely, we describe the algorithm for the finite volume and semi-Lagrangian WB kinetic relaxation in \cref{sec:HKR-FV-WB,sec:HKR-SL-WB}, respectively. Finally, numerical experiments are reported in \cref{sec:numerical experiments}, and the conclusions follow in \cref{sec:conclusions}.

\section{Preliminaries}
\label{sec:preliminaries}

This section introduces the kinetic framework for systems of balance laws (in \cref{sec:kinetic_relaxation}) and the concept of exactly well-balanced reconstruction operators (in \cref{sec:well_balanced_procedure}).
These constitute the building blocks for the design of the well-balanced kinetic relaxation schemes presented in \cref{sec:WB_kinetic_relaxation}.

\subsection{Kinetic Relaxation for One-Dimensional Systems of Balance Laws}
\label{sec:kinetic_relaxation}

The kinetic relaxation framework is well-established for systems of balance laws in multiple space dimensions (which are treated in e.g.~\cite{coulette2019high,GerHelMic2022,GerHelMicWeb2024}).
Despite that, for simplicity and clarity, we restrict ourselves to the one-dimensional case in this paper, and we thus present this framework for the system \eqref{balance law}\REVISORONE{; the extension to several space dimensions is discussed in \cref{sec:conclusions}}.

The goal is to approximate the nonlinear system of balance laws with a linear system of kinetic equations with discrete velocities and a relaxation source term.
In dimension $d$, the system contains $(d+1)$ systems of equations;
in our one-dimensional setting, this reduces to the following system of equations:
\begin{equation}
    \begin{dcases}
        \partial_t f_+ + \lambda \partial_x f_+ = g_+(u\REVISORONE{,x}) + \frac{1}{\tau}\big(m_+(u) - f_+\big), \\
        \partial_t f_- - \lambda \partial_x f_- = g_-(u\REVISORONE{,x}) + \frac{1}{\tau}\big(m_-(u) - f_-\big).
    \end{dcases}
    \label{eq:kinetic_system}
\end{equation}
In \eqref{eq:kinetic_system}, we have defined:
\begin{itemize}[nosep]
    \item A (small but nonzero) relaxation time $\tau > 0$.
    \item A set of constant vectors, the kinetic velocities\footnote{This choice is commonly called  D1Q2 in \cite{QiaHumLal1992} in the lattice-Boltzmann literature; other choices include D1Q3, where $\mathcal{V}=(-\lambda, 0, \lambda) \in \mathbb{R}^3$.} $\mathcal{V}=(-\lambda, \lambda) \in \mathbb{R}^2$.
    \item The two kinetic unknowns $f_- \in \mathbb{R}^{m}$ and $f_+ \in \mathbb{R}^{m}$.
    \item The two kinetic equilibrium functions $m_{-}: \mathbb{R}^{m} \rightarrow \mathbb{R}^{m}$ and $m_{+}: \mathbb{R}^{m} \rightarrow \mathbb{R}^{m}$.
    \item The two kinetic source terms $g_{-}: \mathbb{R}^{m}\REVISORONE{\times \mathbb{R}} \rightarrow \mathbb{R}^{m}$ and $g_{+}: \mathbb{R}^{m}\REVISORONE{\times \mathbb{R}} \rightarrow \mathbb{R}^{m}$.
\end{itemize}

\medskip

We now seek expressions of $m_\pm$ and $g_\pm$ to recover the original system of balance laws in the limit $\tau \to 0$.
To that end, we make the consistency assumption
\begin{equation}
    f_+ + f_- = u. \label{eq:consistency_f}
\end{equation}
Then, we sum both kinetic equations and use \eqref{eq:consistency_f}, to get
\begin{equation}
    \partial_{t} u
    +
    \lambda \partial_x(f_+ - f_-)
    =
    g_+(u,x) + g_-(u,x) + \frac{1}{\tau}\left(m_+(u) + m_-(u) - u\right).
\end{equation}
For the relaxation source term to vanish and to recover the flux function of the original system, we obtain the following conditions on the kinetic equilibrium functions:
\begin{align}
    m_+(u) + m_-(u)                 & = u, \label{eq:consistency_u}       \\
    \lambda m_+(u) - \lambda m_-(u) & = F(u), \label{eq:consistency_flux}
\end{align}
These conditions lead to the following definition of the kinetic equilibrium functions:
\begin{equation*}
    m_\pm(u) = \frac{u}{2} \pm \frac{F(u)}{2\lambda}.
\end{equation*}
Moreover, the source term must be consistent with the original source term, which yields
\begin{equation}
    g_+(u,x) + g_-(u,x) = S(u,x). \label{eq:consistency_source}
\end{equation}
Further, a Chapman-Enskog expansion presented in \cite{coulette2019high} suggests defining the kinetic source terms as:
\begin{equation}
    g_\pm(u,x) = (\nabla_{u}m_\pm(u)) \, S(u,x),
\end{equation}
which satisfies the consistency condition since
\begin{equation*}
    g_-(u,x) + g_+(u,x) = (\nabla_u m_-(u) + \nabla_u m_+(u)) \, S(u,x) = (\nabla_u u) \, S(u,x) = S(u,x).
\end{equation*}

\medskip
The linear stability of the kinetic relaxation scheme is governed by the \textit{subcharacteristic condition}, see e.g.~\cite{CheLevLiu1994}.
This requirement ensures that the kinetic velocities transport information faster than the physical waves, and is akin in spirit to a CFL condition.
To satisfy this condition, the kinetic velocities should satisfy
\begin{equation}
    \lambda \geq \max_{i=1,\dots, \Error{m}} \{|\lambda_i|\},
    \label{eq:subchar_condition}
\end{equation}
where $\lambda_i$ are the characteristic velocities of the initial hyperbolic system (i.e., the eigenvalues of the Jacobian matrix $\nabla_u F(u)$).

To approximate the solutions of the vectorial kinetic representation for systems of balance laws, we employ a splitting method, see e.g. \cite{McLQui2002,coulette2019high}.
This approach decouples the transport phenomena from the relaxation and source term effects.
The problem is split into the following two distinct steps, solved sequentially.
We assume that the kinetic variables $f_\pm^n$ approximated at time $t^n$ are known\footnote{To compute them at the initial time, we set $f_\pm^0 = m_\pm(u^0)$ with $u^0$ the initial condition of \eqref{balance law}.}, and we seek to compute $f_\pm^{n+1}$ at time $t^{n+1}=t^n + \Delta t$.

\medskip

\begin{enumerate}[label=\textbf{\Alph*.}, leftmargin=*]
    \item \textbf{Transport Step:} We start by solving the linear homogeneous transport equation satisfied by both kinetic unknowns $f_+$ and $f_-$:
          \begin{equation}
              \partial_t f_\pm \pm \lambda \partial_x f_\pm = 0.
              \label{eq:transport_step}
          \end{equation}
          This step may be completed with different techniques. In this paper, we use, on the one hand, a FV scheme with an upwind numerical flux and, on the other hand, a Semi-Lagrangian approach based on the characteristic method. Solving \eqref{eq:transport_step} yields the transported kinetic variables $f_\pm^{n,*}$, and we recover the transported original variable $u^{n,*}$ by summing both kinetic variables.
          \medskip
    \item \textbf{Relaxation-Source Step:} We are left with solving a time-dependent ordinary differential equation (ODE), which accounts for the kinetic source term $\REVISORONE{g_\pm(u,x)}$ and the relaxation source term:
          \begin{equation}
              \partial_{t}f_\pm = g_\pm(u,x) + \frac{1}{\tau}(m_\pm(u)-f_\pm).
              \label{eq:relaxation_source_step}
          \end{equation}
          First, by summing both equations in \eqref{eq:relaxation_source_step}, and using the consistency properties \eqref{eq:consistency_f}, \eqref{eq:consistency_u} and \eqref{eq:consistency_source}, we recover the evolution equation related to the source term
          \begin{equation}
              \partial_{t}u = S(u,x).
              \label{eq:macro_evolution}
          \end{equation}
          This equation is then approximated by any ODE solver (e.g., a Crank-Nicolson scheme to ensure second-order accuracy in time).
          This leads to solving a potentially nonlinear system, starting from the transported state $u^{n,*}$ to find $u^{n+1,*}$.
          Equipped with $u^{n+1,*}$, we discretize \eqref{eq:relaxation_source_step} with an ODE solver (here, the Crank-Nicolson method again), to obtain
          \begin{equation*}
              f_\pm^{n+1} = (1-\omega)f_\pm^{n,*} + \omega \tau \frac{g_\pm(u^{n,*},x)+g_\pm(u^{n+1,*},x)}{2} + \omega \frac{m_\pm(u^{n,*})+m_\pm(u^{n+1,*})}{2},
              \label{eq:CN-relaxation 1}
          \end{equation*}
          where we have set
          \begin{equation}
              \omega = \frac{\nicefrac{\Delta t}{\tau}}{1 + \nicefrac{\Delta t}{2\tau}}.
          \end{equation}
          In the limit $\tau \to 0$, we observe that $\omega \to 2$, which is the well known over-relaxation regime.
          Choosing $\tau$ such that $\tau = \mathcal{O}(\Delta t^2)$ leads to \Error{$\omega = 2 - C\Delta t$ for some constant $C>0$}, which ensures that the term involving $\tau g_\pm$ can be neglected up to second order. The final update reads:
          \begin{equation}
              \Error{f_{\pm}^{n+1} = (1-\omega)f_{\pm}^{n,*} + \omega \frac{m_{\pm}(u^{n,*})+m_{\pm}(u^{n+1,*})}{2}.}\label{eq:CN-relaxation 2}
          \end{equation}
          This scheme is stable, second-order accurate, and avoids the explicit computation of the kinetic source terms $\REVISORONE{g_\pm(u,x)}$ (and thus the Jacobian of the flux).
\end{enumerate}

\begin{remark}[\REVISORONE{Stiff limit and asymptotic-preserving property}]
    \label{rem:AP}
    \REVISORONE{The relaxation time $\tau$ only appears through the parameter $\omega \in (0,2)$, and the coefficients $1-\omega$ and $\omega$ in \eqref{eq:CN-relaxation 2} are uniformly bounded with respect to $\tau$. Consequently, no time step restriction is associated with the stiffness of the relaxation source term: the only restriction is the CFL condition of the transport step in the explicit case, and there is none for the implicit and semi-Lagrangian transport steps of \cref{sec:implicit_WB_FV,sec:HKR-SL-WB}. In the stiff limit $\tau \to 0$ with $\Delta t$ fixed, $\omega \to 2$ and the scheme tends to the over-relaxation scheme, which is a consistent, second-order accurate discretization of the balance law \eqref{balance law}, see~\cite{Dellar2013,coulette2019high,Bellotti2023}. The scheme is thus asymptotic preserving in the sense of~\cite{Jin1999,Jin2012}.
        % More precisely, the accuracy with respect to \eqref{balance law} is governed by the size of $\tau$ relative to $\Delta t$, exactly as in the Chapman-Enskog expansion. For a fixed $\omega < 2$, not depending on $\Delta t$, one has $\tau = \Delta t \, (2-\omega)/(2\omega) = \mathcal{O}(\Delta t)$. Hence, the scheme approximates the kinetic system \eqref{eq:kinetic_system}, whose solution lies at a distance $\mathcal{O}(\tau)$ \MOD{from} that of \eqref{balance law}, so that only first-order accuracy in time is expected. The choice $\omega = 2 - C\Delta t$ corresponds to $\tau = \mathcal{O}(\Delta t^2)$ and yields second-order accuracy.
        Finally, $\omega = 1$ corresponds to a full projection onto the equilibrium, $f_\pm^{n+1} = m_\pm(u^{n+1,*})$, which is the limit $\tau \to 0$ of a backward Euler discretization of \eqref{eq:relaxation_source_step}, and gives back a standard finite volume scheme for \eqref{balance law}, see \cref{rem:riemann_solvers}. Let us finally stress that the well-balanced property established in \cref{sec:WB_kinetic_relaxation} holds for every $\omega \in (0,2]$, and is therefore uniform with respect to $\tau$; this is illustrated numerically in \cref{subsec:test9_swe}.}
\end{remark}

\medskip
While the transport and relaxation steps described above can be individually discretized with high-order methods, simply solving them sequentially restricts the global temporal accuracy to first order. This reduction in accuracy arises from the splitting error, which is proportional to the commutator of the convection and source operators \cite{holden2010splitting}. To overcome this limitation and achieve second-order accuracy in time, while remaining consistent with the spatial discretization, we employ a symmetric Strang splitting method \cite{strang1968construction}. This approach symmetrizes the sequence of operators, requiring a half-step of transport, a full step of relaxation, and a final half-step of transport. Thus, a second-order accurate splitting scheme is given by
\begin{equation}
    \mathcal{S}_{2}(\Delta t) = \mathcal{T}\left(\frac{\Delta t}{2}\right) \circ \mathcal{R}(\Delta t) \circ \mathcal{T}\left(\frac{\Delta t}{2}\right), \label{strang_splitting}
\end{equation}
where $\mathcal{T}$ and $\mathcal{R}$ denote the transport and relaxation operators, respectively.
Higher-order accuracy in time is then achieved, according to \cite{coulette2019high}, by considering suitable combinations of \eqref{strang_splitting}.

\subsection{Exactly Well-Balanced Procedure}
\label{sec:well_balanced_procedure}

In this paper, we tackle so-called exactly well-balanced methods.
They are designed to exactly preserve the stationary solutions
\begin{equation}
    \partial_x F(u^e)=S(u^e,x),\label{stationary solution}
\end{equation}
of the original system of balance laws \eqref{balance law}, up to machine precision.
One way of constructing explicit or implicit high-order well-balanced methods is through well-balanced reconstruction operators, see e.g.~ \cite{castro2020well}.

To construct such operators, consider a uniform space discretization made of cells $I_i = (x_{i-\nicefrac{1}{2}},x_{i+\nicefrac{1}{2}})_{i\in\mathbb{Z}}$ with $\Delta x = x_{i+\nicefrac{1}{2}} - x_{i-\nicefrac{1}{2}}$ constant for all $i$.
Moreover, we denote by $(\alpha_m, \omega_m)_{m\in\{1,\dots,M\}}$ the nodes and weights of a quadrature formula of order~$p \geq 1$ in the interval $[0, 1]$.
On the $i$\textsuperscript{th} cell, the quadrature nodes are thus given by $x_{i}^{m}=x_{i-\nicefrac{1}{2}}+\alpha_{m}\Delta x$.
Then, one defines the approximate cell averages $\bar{u}_{i}$ of a function $u$, using the quadrature formula, as
\begin{equation*}
    \bar{u}_{i} = \sum_{m=1}^{M}\omega_{m}u(x_{i}^{m})=\frac{1}{\Delta x}\int_{x_{i-\nicefrac{1}{2}}}^{x_{i+\nicefrac{1}{2}}}u(x)dx +O(\Delta x^p).
\end{equation*}
A reconstruction operator provides approximations of $u$ at any point $x$ of the $i$\textsuperscript{th} cell:
\begin{equation}
    \forall i \in \mathbb{Z}, \quad
    \forall x \in I_i, \qquad
    P_{i}(x;\{\bar{u}_{j}\}_{j\in\mathcal{S}_{i}}) = u(x) + O(\Delta x^s),
\end{equation}
where $\mathcal{S}_{i}$ is a stencil of cells around the $i$\textsuperscript{th} cell, and $\Error{s \geq 1}$ is the order of the approximation.

These reconstruction operators are obtained by interpolation from the cell averages of the stencil. Some high-order reconstruction operators include MUSCL~\cite{Lee1979}, CWENO~\cite{CraPupSemVis2017} or CWENOZ~\cite{SemTraPup2021}.
We can now define what would make such an operator exactly well-balanced.

\begin{definition}
    \label{def:well_balanced_reconstruction_operator}
    Given a stationary solution $u^{e}$ of \eqref{balance law} satisfying \eqref{stationary solution}, a reconstruction operator $P_{i}$ is said to be exactly well-balanced for $u^{e}$ if
    \begin{equation}
        \forall i \in \mathbb{Z}, \quad
        \forall x \in I_i, \qquad
        P_{i}(x;\{\bar{u}_{j}^{e}\}_{j\in\mathcal{S}_{i}}) = u^{e}(x),
    \end{equation}
    where $\bar{u}_{j}^{e}$ represent the cell-averages obtained from the stationary solution $u^{e}$ by a quadrature formula (or the exact cell-averages if available).
\end{definition}

In general, a standard reconstruction operator is not well-balanced, so we can use the strategy introduced in \cite{castro2008well} to obtain a well-balanced operator $P_{i}$ from a standard one $Q_i$.
It is described in \cref{alg:WB_reconstruction}.

\begin{algorithm}[!ht]
    \caption{Well-Balanced Reconstruction Operator}
    \label{alg:WB_reconstruction}
    \begin{algorithmic}[1]
        \Require Cell averages $\{\bar{u}_i\}$, stencils $\mathcal{S}_i$,
        quadrature nodes and weights $(x_j^m, \omega_m)_{m\in\{1,\dots,M\}}$,
        standard reconstruction operator $Q_i$
        \Ensure Well-balanced reconstruction operator $P_i$

        \For{each cell $i$}

        \State Attempt to find a stationary solution
        $u_i^e$ defined on
        $\mathcal{S}_i$
        such that
        \begin{equation}\label{alg2-steady}
            \sum_{m=1}^{M} \REVISORTWO{\omega_m} \, u_i^e(x_i^m) = \bar{u}_i.
        \end{equation}

        \If{no such solution exists}
        \State $u_i^e \gets 0$
        \EndIf

        \medskip

        \State Compute the fluctuations $\{v_{j}\}_{j\in\mathcal{S}_{i}}$, given by
        \begin{equation*}
            v_{j} = \bar{u}_{j} - \sum_{m=1}^{M}\REVISORTWO{\omega_m}u_{i}^{e}(x_{j}^{m}).
        \end{equation*}

        \medskip

        \State Finally, define the well-balanced reconstruction operator by
        \begin{equation}
            \label{wb rec op}
            P_i(x) = u_i^e(x) + Q_i\big(x;\{v_{j}\}_{j\in\mathcal{S}_{i}}\big).
        \end{equation}
        \EndFor
    \end{algorithmic}
\end{algorithm}

It can be easily proved that the reconstruction operator $P_{i}$ obtained from \cref{alg:WB_reconstruction} is exactly well-balanced provided that $Q_i$ is exact for the zero function, see~\cite{castro2008well}.
Note that, in every cell, a nonlinear problem has to be solved in \eqref{alg2-steady}.
It consists in finding a stationary solution in the stencil of the cell with the same average value in the cell. To that end, having a closed-form expression (or an implicit algebraic definition) of the stationary solutions is required.
If no such expression exists, reconstruction operators that are (not exactly) well-balanced can be designed, by numerically solving ODE system \eqref{stationary solution}, see e.g. \cite{gomez2021collocation,gomez2021high,gomez2021well}.
\REVISORTWO{This nonlinear problem may have a unique solution, multiple solutions, or no solution at all.
    In the case where the model implies that multiple solutions may exist, the selection of the correct solution is crucial. Indeed, the solution to \eqref{alg2-steady} should belong to the branch from which the cell averages originate whenever these are the averages of a stationary solution. An example for the SWE is given in \cref{sec:numerical experiments}. Another case occurs when there is no solution (\MOD{for instance, for moving equilibria of the SWE, when the specific energy of the cell is lower than the critical one at some point of the stencil}). If that happens, the fallback $u_i^e=0$ applies, and the scheme remains high-order accurate, without preserving the non-existent stationary states.}

\begin{remark}[\REVISORONE{Exactness and order of the reconstruction operator}]
    \label{rem:exactness}
    \REVISORONE{The exactness of $P_i$ does not depend on the order of $Q_i$: as soon as the cell averages are those of a stationary solution of the family being balanced, the fluctuations $v_j$ vanish on the whole stencil and \eqref{wb rec op} gives $P_i=u_i^e+Q_i(\cdot\,;0)=u_i^e$, whatever $Q_i$ may be. Hence a steady state is preserved up to round-off errors, whatever the order of the scheme. The only limitation \MOD{is} the accuracy with which the nonlinear equation \eqref{alg2-steady} is solved.}
    % \REVISORTWO{Note also that $u_i^e$ is an exact, continuous stationary solution on the stencil, only the selection of its parameters being discrete; and that, the cell averages being computed with a quadrature rule of order $p$, the discrete stationary state that is exactly preserved differs from the set of exact cell averages of $u^e$ by $\mathcal{O}(\Delta x^p)$, an offset unrelated to the order of the scheme, which vanishes when exact cell averages are available, see~\cite{castro2020well}.}
    \REVISORTWO{Note also that $u_i^e$ is an exact, continuous stationary solution evaluated on the stencil. Moreover, recall that the cell averages are computed with a quadrature rule of order $p$. Hence, the discrete stationary state that is exactly preserved differs from the set of exact cell averages of $u^e$ by $\mathcal{O}(\Delta x^p)$. This error is unrelated to the order of the scheme, and it is controllable by adjusting the quadrature rule, see~\cite{castro2020well}. Stationary solutions that do not belong to the balanced family are not preserved exactly.}
\end{remark}

% \medskip
% \newline
Equipped with the exactly well-balanced reconstruction operator, the general semi-discrete numerical method for solving \eqref{balance law} reads, for all $i \in \mathbb{Z}$,
\begin{equation} \label{eq:FV_WB}
    \begin{aligned}
        \frac{d\bar u_{i}}{dt} =
         & -\frac{1}{\Delta x}(F_{i+\nicefrac{1}{2}}(t)-F_{i-\nicefrac{1}{2}}(t))
        + \frac{1}{\Delta x}\left(F(u_{i}^{e,t}(x_{i+\nicefrac{1}{2}})) - F(u_{i}^{e,t}(x_{i-\nicefrac{1}{2}}))\right) \\
         & + \sum_{m=1}^{M}\omega_{m}\left(S(P_{i}^{t}(x_{i}^{m}),x_i^m) - S(u_{i}^{e,t}(x_{i}^{m}),x_i^m)\right),
    \end{aligned}
\end{equation}
where
\begin{itemize}[nosep]
    \item $P_{i}^{t}$ is the well-balanced reconstruction obtained from the cell averages $\{\bar{u}_{i}(t)\}$, following \cref{alg:WB_reconstruction};
    \item $u_{i}^{e,t}$ is the stationary solution in the $i$\textsuperscript{th} cell computed in \cref{alg:WB_reconstruction} from the cell averages  $\{\bar{u}_{i}(t)\}$;
    \item $F_{i+\nicefrac{1}{2}}(t)=\mathbb{F}(u_{i+\nicefrac{1}{2}}^{t,-},u_{i+\nicefrac{1}{2}}^{t,+})$, where $\mathbb{F}$ is any consistent numerical flux for the homogeneous system, and
          \begin{equation*}
              u_{i+\nicefrac{1}{2}}^{t,-}=P_{i}^{t}(x_{i+\nicefrac{1}{2}}), \quad u_{i+\nicefrac{1}{2}}^{t,+}=P_{i+1}^{t}(x_{i+\nicefrac{1}{2}}).
          \end{equation*}
\end{itemize}
\noindent
This numerical method is exactly well-balanced in the sense that, given any stationary solution~$u^{e}$, the set of cell values $\{\bar{u}_{i}^{e}\}$ is an equilibrium of the ODE system \eqref{eq:FV_WB}: see~\cite{castro2020well}.
The above procedure was written for explicit schemes.
To produce an exactly well-balanced implicit scheme, one must properly modify \eqref{wb rec op} and \eqref{eq:FV_WB}, see \cite{gomez2021implicit}.

\medskip

\section{Exactly Well-Balanced Kinetic Relaxation methods}
\label{sec:WB_kinetic_relaxation}
% \medskip
% \newline
Designing exactly well-balanced schemes within the kinetic relaxation framework is more complex than in the setting of the balance law \eqref{balance law}.
In standard finite volume formulations applied to \eqref{balance law}, the numerical flux is constructed to strictly balance the physical flux~$F(u)$ and the source term $S(u,x)$ that define the stationary solutions.
However, the kinetic relaxation approach replaces the nonlinear physical flux with a set of linear transport operators. The main difficulty is that a steady state $u^e$ satisfying \eqref{stationary solution} is no longer a stationary solution of the relaxed kinetic system \eqref{eq:kinetic_system}.
% Consequently, the standard initialization of kinetic variables using the kinetic equilibrium functions means that the linear transport step will not preserve this state; the solution to the kinetic equation tends to drift away from the macroscopic equilibrium due to the structural mismatch between the transport operator and the physical source term.
Therefore, this section's purpose is to devise a mechanism that enforces the invariance of $u^e$ during the transport and projection stages, without compromising the consistency or the transient dynamics of the scheme for non-stationary evolution.

\REVISORTWO{In particular, \cref{alg:WB_reconstruction} cannot be applied directly to the transport step \eqref{eq:transport_step}, whatever the choice of $Q_i$,
    as its stationary solutions do not coincide with those of the original system \eqref{balance law}.
    % its stationary solutions are the constant states, and the kinetic projections $m_\pm(u^e)=\nicefrac{u^e}{2}\pm\nicefrac{F(u^e)}{2\lambda}$ of a non-trivial steady state are not constant in space, so that they are not left invariant by the linear transport.
    The kinetic system thus has to be reformulated so that $u^e$ becomes an invariant of the transport step, and the well-balanced property has to be established for the composition of the transport and projection stages acting on $(f_+,f_-)$.}

% \medskip

The main idea relies on splitting the kinetic variables into two components: a stationary part that corresponds to the projection of the stationary solution onto the kinetic unknown space, and a fluctuation part that captures the deviation from this stationary state.
Let us first explain the strategy in a continuous setting.
We get
\begin{equation}
    \label{kinetic stationary proyection}
    f_\pm=f_\pm^e+(f_\pm-f_\pm^e)
    \text{, \qquad with \quad}
    f_\pm^e=m_\pm(u^e),
\end{equation}
with $u^e$ being the stationary solution that solves \eqref{stationary solution}.
Since $\partial_tf_\pm^e=0$, the transport step \eqref{eq:transport_step} applied to $f_\pm^e$ leaves~$f_\pm^e$ invariant only if $\partial_xf_\pm^e=0$, which is not the case as soon as $u^e$ is not constant.
\begin{equation*}
    \begin{dcases}
        \partial_t f_+ + \lambda \partial_x f_+ = g_+(u\REVISORONE{,x}) + \frac{1}{\tau}\big(m_+(u) - f_+\big), \\
        \partial_t f_- - \lambda \partial_x f_- = g_-(u\REVISORONE{,x}) + \frac{1}{\tau}\big(m_-(u) - f_-\big).
    \end{dcases}
    \label{eq:kinetic_system_new}
\end{equation*}
To circumvent this, we subtract $\pm\lambda\partial_xf_\pm^e$ to both sides of \eqref{eq:kinetic_system}, leading to
\begin{equation*}
    \partial_t f_\pm \pm \lambda \partial_x f_\pm
    \mp \lambda\partial_xf_\pm^e
    =
    g_\pm(u\REVISORONE{,x}) + \frac{1}{\tau}\big(m_\pm(u) - f_\pm\big)
    \mp \lambda\partial_xf_\pm^e,
\end{equation*}
which we also split into a transport step and a relaxation-source step.
The transport step becomes
\begin{equation*}
    \partial_t f_\pm \pm \lambda \partial_x (f_\pm - f_\pm^e) = 0.
\end{equation*}
Since $\partial_tf_\pm^e=0$, this is nothing but
the transport step \eqref{eq:transport_step} applied to the fluctuations:
\begin{equation}\label{eq: transport step fluctuations}
    \partial_t(f_\pm-f_\pm^e) \pm \lambda \partial_x(f_\pm-f_\pm^e)=0,
\end{equation}
which leaves $f_\pm^e$ invariant.
The relaxation-source step becomes
\begin{equation*}
    \partial_tf_\pm=g_\pm(u,x)\mp\lambda\partial_xf_\pm^e+\frac{1}{\tau}\Big(m_\pm(u)-f_\pm\Big).
\end{equation*}
Summing both equations as in \eqref{eq:macro_evolution}, and using that $\lambda\partial_x(f_+^e-f_-^e)=\partial_xF(u^e)=S(u^e,x)$ by \eqref{eq:consistency_flux} and \eqref{stationary solution}, we obtain the following ODE on $u$:
\begin{equation}\label{source evolution with stationary}
    \partial_tu=S(u,x)-S(u^e,x).
\end{equation}
After that, the projection step reads just like in \eqref{eq:CN-relaxation 2}, because the kinetic source terms, here $g_\pm(u,x)\mp\lambda\partial_xf_\pm^e$, are neglected up to second order.
In practice, $\lambda$ is taken as the largest velocity corresponding to either the stationary solution or the transient solution, so that \eqref{eq:subchar_condition} holds for both.

% \medskip
% \newline
The remainder of this section is dedicated to the discretization of the above strategy.
More specifically, we propose a finite volume scheme in \cref{sec:HKR-FV-WB} and a semi-Lagrangian scheme in \cref{sec:HKR-SL-WB}.
There, we address the numerical resolution of the system \eqref{eq: transport step fluctuations}--\eqref{source evolution with stationary}--\eqref{eq:CN-relaxation 2} on the space-time domain $[a,b] \times [0,T]$. The spatial domain is discretized using a uniform mesh consisting of $N_x$ equally spaced cells. A particular consideration must be made regarding the boundary treatment: although the physical flux of the hyperbolic system may possess a single characteristic direction, the underlying kinetic relaxation schemes rely on a set of discrete kinetic velocities that inherently propagate information in both directions. Consequently, the numerical evaluation of fluxes at the boundaries requires information from ghost cells on both sides of the domain. \REVISORONE{In this paper, we restrict ourselves to periodic and free-flow boundary conditions. For each scheme, we describe below how the ghost cells and the incoming kinetic populations are defined. A detailed analysis of more complex boundary procedures for kinetic schemes lies beyond the scope of this paper.}

\subsection{Finite Volume Exactly Well-Balanced Kinetic Relaxation}
\label{sec:HKR-FV-WB}

In this section, we focus on the finite volume framework, and we investigate both explicit schemes (in \cref{sec:explicit_WB_FV}) and implicit schemes (in \cref{sec:implicit_WB_FV}).

\subsubsection{Explicit schemes}
\label{sec:explicit_WB_FV}

% ~ \\
% \medskip
% \newline
Let $\{u_i^n\}_{i=1}^{N_x}$ be the set of cell averages\footnote{It was previously denoted by $\{\overline{u}_i^n\}_{i=1}^{N_x}$, but we omit the bar to simplify notation.} at time $t^n$. The goal is to compute the appropriate reconstruction operators. At this stage, two different strategies can be adopted: one may either construct the reconstruction operators directly in the space of kinetic variables or, alternatively, perform the reconstruction on the physical variables and subsequently recover the kinetic distributions by projecting onto the equilibrium functions. We select the second strategy, as it offers a significant reduction in computational cost while maintaining the efficiency.

\medskip
We start by presenting the explicit scheme. Let
\begin{equation}
    \label{eq:kinetic_reconstruction}
    P_{\pm,i}^n(x)=m_{\pm}(P_i^n(x)),
\end{equation}
be the reconstruction operator in the $i$\textsuperscript{th} cell at time $t^n$ of the kinetic variable $f_\pm$, where~$P_i^n(x)$ is the well-balanced reconstruction operator of the physical variable.
Now, we evolve the transport stage for the fluctuations \eqref{eq: transport step fluctuations} with a finite volume method, using an upwind numerical flux.
To that end, we introduce the notation
\begin{equation*}
    f_{\pm,i}^n\approx\frac{1}{\Delta x}\int_{x_{i-\frac{1}{2}}}^{x_{i+\frac{1}{2}}}f_\pm(x,t^n)dx,
    % \quad f_{\pm,i}^{e,n}(x)=m_\pm(u_i^{e,n}(x)).
    % f_{\pm,i}^n = \sum_{m=1}^{M} \alpha_m \, f_\pm(x_i^m, t^n)
    \text{\qquad and \qquad}
    f_{\pm,i}^{e,n}(x)=m_\pm(u_i^{e,n}(x)).
\end{equation*}
Taking into account that $\partial_tf_\pm^e=0$ in \eqref{eq: transport step fluctuations},
and integrating this equation over a space-time
control volume $[x_{i-\frac{1}{2}},x_{i+\frac{1}{2}}]\times[t^n,t^{n+1}]$,
we obtain
\begin{equation*}
    \frac{f_{\pm,i}^{n,*} - f_{\pm,i}^n}{\Delta t}
    \pm
    \frac{\lambda}{\Delta x}
    \left[
        \left(f_\pm - f_\pm^e\right)(x_{i+\frac{1}{2}},t^n) - \left(f_\pm - f_\pm^e\right)(x_{i-\frac{1}{2}},t^n)
        \right]
    = 0.
\end{equation*}
On the one hand, we can directly evaluate the stationary solution at each cell interface:
\begin{equation*}
    (f_\pm^e)(x_{i+\frac{1}{2}},t^n) = f_{\pm,i}^{e,n}(x_{i+\frac{1}{2}})
    \text{\qquad and \qquad}
    (f_\pm^e)(x_{i-\frac{1}{2}},t^n) = f_{\pm,i}^{e,n}(x_{i-\frac{1}{2}}).
\end{equation*}
On the other hand, we approximate the flux of the fluctuations
with an upwind numerical flux. This means that
\begin{equation*}
    \begin{dcases}
        f_+(x_{i+\frac{1}{2}},t^n) \approx P^n_{+,i}(x_{i+\frac{1}{2}})
        \text{\qquad and \qquad}
        f_+(x_{i-\frac{1}{2}},t^n) \approx P^n_{\REVISORTWO{+},i-1}(x_{i-\frac{1}{2}}), \\
        f_-(x_{i+\frac{1}{2}},t^n) \approx P^n_{-,i+1}(x_{i+\frac{1}{2}})
        \text{\qquad and \qquad}
        f_-(x_{i-\frac{1}{2}},t^n) \approx P^n_{-,i}(x_{i-\frac{1}{2}}).
    \end{dcases}
\end{equation*}
Putting everything together, we obtain
\begin{equation}\label{WB HKR-VF 2}
    f_{\pm,i}^{n,*}
    =
    f_{\pm,i}^n
    \mp
    \lambda \frac{\Delta t}{\Delta x} \Big(F_{\pm,i+\frac{1}{2}}^n-F_{\pm,i-\frac{1}{2}}^n\Big)
    \pm
    \lambda \frac{\Delta t}{\Delta x} \Big(f_{\pm,i}^{e,n}(x_{i+\frac{1}{2}})-f_{\pm,i}^{e,n}(x_{i-\frac{1}{2}})\Big),
\end{equation}
with the upwind numerical flux defined by
\begin{equation*}
    \begin{dcases}
        F_{+,i+\frac{1}{2}}^n = P^n_{+,i}(x_{i+\frac{1}{2}})
        \text{\qquad and \qquad}
        F_{+,i-\frac{1}{2}}^n = P^n_{\REVISORTWO{+},i-1}(x_{i-\frac{1}{2}}), \\
        F_{-,i+\frac{1}{2}}^n = P^n_{-,i+1}(x_{i+\frac{1}{2}})
        \text{\qquad and \qquad}
        F_{-,i-\frac{1}{2}}^n = P^n_{-,i}(x_{i-\frac{1}{2}}),
    \end{dcases}
\end{equation*}
or, equivalently with a more compact notation,
\[
    F_{\pm,i+\frac{1}{2}}^n = P^n_{\pm,i + \frac 1 2 \mp \frac 1 2}(x_{i+\frac{1}{2}})
    \text{\qquad and \qquad}
    F_{\pm,i-\frac{1}{2}}^n = P^n_{\pm,i - \frac 1 2 \mp \frac 1 2}(x_{i-\frac{1}{2}}).
\]
Once the kinetic variables $\Error{f_{\pm,i}^{n,*}}$ are transported,
% which are the result of $\mathcal{T}(\Delta t)$ according to the notation in \eqref{strang_splitting},
we recover the transported physical variables by taking
\begin{equation*}
    u_i^{n,*}=f_{-,i}^{n,*}+f_{+,i}^{n,*}.
\end{equation*}
These physical variables are then evolved following the ODE \eqref{source evolution with stationary}, considering the original source term applied to both physical and equilibrium variables. If the ODE is discretized in time with the Crank-Nicolson method, we have to solve, for each $i \in \{1, \dots, N_x\}$, the potentially nonlinear system
\begin{equation}\label{evolution source FV}
    u_i^{n+1,*}=u_i^{n,*}+\frac{\Delta t}{2}\Big(S(u_i^{n,*},x_i)+S(u_i^{n+1,*},x_i)-2S(u_i^{e,n}(x_i),x_i)\Big).
\end{equation}
Finally, after computing $u_i^{n+1,*}$, we just have to complete the projection step by evolving the kinetic variable like in \eqref{eq:CN-relaxation 2}, by setting
\begin{equation}\label{projection FV}
    f_{\pm,i}^{n+1} = (1-\omega)f_{\pm,i}^{n,*} + \omega \frac{m_{\pm}(u_i^{n,*})+m_{\pm}(u_i^{n+1,*})}{2},
\end{equation}
with $\omega=2-C\Delta t$.
Finally, we recover the updated physical variables at the next time step by computing
\begin{equation*}
    u_i^{n+1}=f_{-,i}^{n+1}+f_{+,i}^{n+1}.
\end{equation*}
This scheme is first-order accurate in space and time, due to the splitting and the use of an upwind flux in the transport step.
Like in \eqref{strang_splitting}, we denote by $\mathcal{T}_1^{\text{exp,FV}}(\Delta t)$ the transport step \eqref{WB HKR-VF 2} and by $\mathcal{R}_2(\Delta t)$ the projection-source step \eqref{evolution source FV}--\eqref{projection FV}; the combination of the two is denoted by $\mathcal{S}_1^{\text{exp,FV}}(\Delta t)$.
Due to the explicit treatment of the transport step, the stability of the overall scheme is conditioned to the classical finite volume time step restriction
\begin{equation*}
    \Delta t \leq \frac{\Delta x}{\lambda}.
\end{equation*}
Note that, thanks to the Crank-Nicolson discretization of the projection-source step,
this step is unconditionally stable,
and thus does not impose any additional time step restriction.

\paragraph{\REVISORONE{Boundary conditions}}
\REVISORONE{The transport step \eqref{WB HKR-VF 2} requires building the reconstruction operators in the boundary cells. For instance, at the left boundary, the incoming flux $F_{+,\frac{1}{2}}^n$ is required, and\MOD{, in a standard upwind scheme, it would be computed from} the ghost cell average $u_0^n$ (or more averages for wider stencils).
    For periodic boundary conditions, the ghost cells are the cells at the opposite end of the domain, and the well-balanced property follows from the periodicity of the stationary solution.
    For free-flow boundary conditions, the situation is a bit more complex, and we proceed as follows (focusing on the left boundary, as the right one is treated symmetrically).
    First, the ghost cell averages are defined as the quadrature averages of the stationary extension of the first cell,
    \begin{equation}
        \label{eq:ghost_cells}
        u_j^n = \sum_{m=1}^{M} \omega_m \, u_1^{e,n}(x_j^m), \qquad j \leq 0,
    \end{equation}
    so that the associated fluctuations in \cref{alg:WB_reconstruction} vanish, i.e., $v_j = 0$ for $j \leq 0$.
    Second, the incoming kinetic variable at the boundary interface is taken equal to the reconstruction of the first cell evaluated at this interface,
    \begin{equation}
        \label{eq:ghost_flux}
        F_{+,\frac{1}{2}}^n = P_{+,1}^n(x_{\frac{1}{2}}),
    \end{equation}
    so that no fluctuation with respect to the interior state enters the domain, whereas the outgoing kinetic variable is upwinded from the interior as in \eqref{WB HKR-VF 2}, i.e., $F_{-,\frac{1}{2}}^n = P_{-,1}^n(x_{\frac{1}{2}})$.
    Let us check that this treatment preserves the well-balanced property. Starting from the cell averages of a stationary solution $u^e$, the fluctuations vanish in the interior cells and, by \eqref{eq:ghost_cells}, in the ghost cells, so that $P_1^n = u_1^{e,n} = u^e$ on $I_1$. Using \eqref{eq:ghost_flux}, the transport step \eqref{WB HKR-VF 2} for $f_+$ in the first cell then reads
    \begin{equation*}
        f_{+,1}^{n,*}
        =
        f_{+,1}^n
        - \lambda \frac{\Delta t}{\Delta x}
        \Big(
        m_+\big(u_1^{e,n}(x_{\frac{3}{2}})\big) - m_+\big(u_1^{e,n}(x_{\frac{1}{2}})\big)
        \Big)
        + \lambda \frac{\Delta t}{\Delta x}
        \Big(
        f_{+,1}^{e,n}(x_{\frac{3}{2}}) - f_{+,1}^{e,n}(x_{\frac{1}{2}})
        \Big)
        =
        f_{+,1}^n.
    \end{equation*}
    For $f_-$, we note that $F_{-,\frac{3}{2}}^n = m_-(u_2^{e,n}(x_{\frac{3}{2}})) = m_-(u^e(x_{\frac{3}{2}})) = f_{-,1}^{e,n}(x_{\frac{3}{2}})$ and $F_{-,\frac{1}{2}}^n = f_{-,1}^{e,n}(x_{\frac{1}{2}})$, so that $f_{-,1}^{n,*} = f_{-,1}^n$ as well. The boundary cells thus behave exactly like the interior ones, and the proof of \cref{teorema WB FV-HKR} below is valid to them.
}

% Finally, in the numerical experiments with free-flow boundaries, the computational domain is extended by a sponge layer of a few cells on each side, in which the solution is relaxed at each time step towards the reference stationary solution $u^e$ through the implicit damping $u \leftarrow u - \frac{\sigma(x)\Delta t}{1+\sigma(x)\Delta t}\,(u - u^e)$, with a polynomial profile $\sigma(x)$ vanishing at the interface with the physical domain. This absorbs the outgoing waves and prevents spurious reflections; since the damping leaves $u^e$ invariant, it does not affect the well-balanced property.

\medskip
We now prove that the scheme described above is indeed exactly well-balanced.
The following result shows that a set of cell averages of a stationary solution is an equilibrium of the splitting operator.

\begin{theorem}\label{teorema WB FV-HKR}
    Consider $\{u_i^e\}_{i=1}^{N_x}$ a set of cell averages of a stationary solution~$u^e$ of the balance law \eqref{balance law}, which satisfies \eqref{stationary solution}. Let $\mathcal{S}_1^{\text{exp,FV}}$ be the scheme constructed above. Then
    \begin{equation*}
        \forall i \in \{1, \dots, N_x\}, \quad
        \mathcal{S}_1^{\text{exp,FV}}(\Delta t)
        (u_i^e)=u_i^e.
    \end{equation*}
\end{theorem}
\begin{proof}
    The goal of this proof is to show that the scheme $\mathcal{S}_1^{\text{exp,FV}}$ is exactly well-balanced, i.e., that if $u_i^n = u_i^e$ for all $i$, then $u_i^{n+1} = u_i^e$ for all $i$ as well. To that end, we show that the transport step \eqref{WB HKR-VF 2} and the projection-source step \eqref{evolution source FV}--\eqref{projection FV} preserve the stationary solution. Let $i \in \{1, \dots, N_x\}$ be fixed, and let $x \in I_i$. We assume that an exactly well-balanced reconstruction operator $P_i^n$ has been constructed.

    Considering \eqref{WB HKR-VF 2} and focusing on $f_{-,i}^n$, we have that
    \begin{equation*}
        F_{-,i-\frac{1}{2}}^n
        =
        P^n_{-,i}(x_{i-\frac{1}{2}})
        =
        m_-\left(P_i^n(x_{i-\frac{1}{2}})\right)
        =
        m_-\left(u^{e,n}(x_{i-\frac{1}{2}})\right),
    \end{equation*}
    where we have used the definition \eqref{eq:kinetic_reconstruction} of $P^n_{-,i}$ as well as the fact that $P_i^n$ is exactly well-balanced.
    Similar identities hold for \smash{$F_{-,i+\nicefrac{1}{2}}^n$}.
    Furthermore, recall that
    \begin{equation*}
        f_{-,i}^{e,n}(x_{i-\frac{1}{2}}) = m_-\left(u^{e,n}(x_{i-\frac{1}{2}})\right).
    \end{equation*}
    Similar identities hold for $f_{+,i}$.
    Therefore, plugging everything into \eqref{WB HKR-VF 2}, we get that
    \begin{equation}
        \label{eq:transport_step_WB_FV}
        f_{\pm,i}^{n,*}=f_{\pm,i}^n=m_{\pm}(u_i^{e,n})
        \text{, \qquad and so \qquad}
        u_i^{n,*}=u_i^{e,n}.
    \end{equation}

    Moving on to the projection-source step, plugging \eqref{eq:transport_step_WB_FV} into \eqref{evolution source FV}, we see that $u_i^{n+1,*}=u_i^{n,*}=u_i^{e,n}$ solves \eqref{evolution source FV}.
    Finally, applying \eqref{projection FV} with the values found above, we get that
    \begin{equation*}
        f_{\pm,i}^{n+1}=(1-\omega)\ m_{\pm}(u_i^{e,n})+\frac{\omega}{2}\Big(m_{\pm}(u_i^{e,n})+m_{\pm}(u_i^{e,n})\Big)=m_{\pm}(u_i^{e,n})=f_{\pm,i}^{n},
    \end{equation*}
    which implies that $u_i^{n+1}=f_{+,i}^{n}+f_{-,i}^{n}=u_i^{e,n}$, as we intended to demonstrate.
\end{proof}

\begin{remark}[\REVISORONE{Relation with approximate Riemann solvers}]
    \label{rem:riemann_solvers}
    \REVISORONE{Kinetic and relaxation models are classically used to build approximate Riemann solvers, through a linear transport step ($\tau=\infty$) followed by a projection onto the equilibrium ($\tau=0$), see e.g.~\cite{JinXin1995,AuPe00,Bouchut2004}. In our setting, the full projection corresponds to $\omega=1$, the only value for which $f_{\pm,i}^{n+1}=m_\pm(u_i^{n+1,*})$.
        % , so that the scheme is a map of $u$ alone.
        Summing the two equations of \eqref{WB HKR-VF 2} then gives a conservative update of $u$ with numerical flux
        \begin{equation*}
            \mathbb{F}_{i+\frac12}^n
            =
            \lambda\Big(F_{+,i+\frac12}^n-F_{-,i+\frac12}^n\Big)
            =
            \frac{F(u_{i+\frac12}^{n,-})+F(u_{i+\frac12}^{n,+})}{2}
            -\frac{\lambda}{2}\Big(u_{i+\frac12}^{n,+}-u_{i+\frac12}^{n,-}\Big),
        \end{equation*}
        with $u_{i+\frac12}^{n,\mp}$ the reconstructed states.
        This is nothing but the Rusanov flux with viscosity coefficient $\lambda$: the first-order scheme with $\omega=1$ is the well-balanced scheme \eqref{eq:FV_WB} of~\cite{castro2008well,castro2020well} with a two-wave approximate Riemann solver of speeds $\pm\lambda$, \eqref{eq:subchar_condition} being the usual bound on those speeds, and with a splitting of the source term.
    }
    % Contrary to the strategies of e.g.~\cite{GreLeR1996,Gosse2001,PerSim2001,BerCha2016}, the equilibrium is embedded in the reconstruction operator rather than in the Riemann solver. The over-relaxation regime $\omega=2-C\Delta t$, for which the kinetic variables are not projected back onto the equilibrium~\cite{Dellar2013}, the palindromic splittings, and the implicit and semi-Lagrangian variants all fall outside this identification; this is why the well-balanced property is established for the composition of the transport and projection stages, for any $\omega\in(0,2]$.}
\end{remark}

\medskip
Equipped with this first-order accurate explicit scheme, its extension to higher orders of accuracy is straightforward. To achieve second-order accuracy, we must ensure second-order precision in the transport step by employing, for example, a second-order TVD Runge-Kutta method coupled with a MUSCL \cite{Lee1979} spatial reconstruction operator for the space fluctuations. Additionally, this must be combined with a Strang splitting technique, as detailed in \eqref{strang_splitting}.
For the third-order scheme, it is sufficient to apply a third-order spatial reconstruction operator; in this work, we use the CWENOZ3 method from~\cite{SemTraPup2021}. Interestingly, to achieve third-order accuracy in time, an explicit third-order transport stage is not strictly required. As shown by \cite{coulette2019high}, it is enough to appropriately combine second-order splittings. For this purpose, we employ the Suzuki splitting method, which is in fact fourth-order accurate in time. It is defined by
\begin{equation}\label{suzuki splitting}
    \mathcal{S}_4(\Delta t)=\mathcal{S}_2(\gamma_0\Delta t)\circ \mathcal{S}_2(\gamma_1\Delta t)\circ \mathcal{S}_2(\gamma_2\Delta t)\circ \mathcal{S}_2(\gamma_3\Delta t)\circ \mathcal{S}_2(\gamma_4\Delta t)
\end{equation}
with $\gamma_0=\gamma_1=\gamma_3=\gamma_4=\frac{1}{4-4^{1/3}}$ and $\gamma_2=\frac{4^{1/3}}{4-4^{1/3}}$, and where the second-order splitting is
\begin{equation}\label{splitting second-order in suzuki}
    \mathcal{S}_2(\Delta t)
    =
    \mathcal{T}\left(\frac{\Delta t}{4}\right)\circ
    \mathcal{R}\left(\frac{\Delta t}{2}\right)\circ
    \mathcal{T}\left(\frac{\Delta t}{2}\right)\circ
    \mathcal{R}\left(\frac{\Delta t}{2}\right)\circ
    \mathcal{T}\left(\frac{\Delta t}{4}\right)
\end{equation}
Obviously, the well-balanced property of the high-order methods follows immediately from the proof of \Cref{teorema WB FV-HKR}, since each intermediate transport stage involving the second-order TVD Runge-Kutta scheme leaves the stationary solution of the original system invariant. Meanwhile, the source evolution and projection steps remain identical to those previously employed.

\subsubsection{Implicit schemes}
\label{sec:implicit_WB_FV}
% ~ \\
% \medskip
% \newline

To derive an implicit finite volume scheme, the transport stage must be modified by adopting an implicit scheme, alongside an adaptation of the well-balanced reconstruction operator following the approach proposed in~\cite{gomez2021implicit}.
As shown in~\cite{gomez2021implicit}, both first and second-order spatial accuracy can be attained by employing a piecewise constant reconstruction operator for the temporal fluctuations.
% \medskip
% \newline

The main difference with respect to the explicit case is that we now perform the reconstructions directly on the kinetic variables rather than the physical ones. The primary reason for this approach is that it allows us to solve linear systems during the finite-volume transport stage. In contrast, reconstructing the physical variables would introduce nonlinearities through the kinetic equilibrium variables, which would mean having to solve nonlinear systems (using e.g. Newton's method).
Then, the reconstruction operator now reads:
\begin{equation}\label{rec op kinetic variables implicit}
    \begin{aligned}
        P_{\pm,i}^t(x)
         & =f_{\pm,i}^{e,n}(x)+Q_{\pm,i}(x;\{v_{\pm,j}\}_{j\in\mathcal{S}_i})+\tilde{Q}_{\pm,i}(x;\{f_{\pm,j}^\text{fluc}(t)\}_{j\in\mathcal{\tilde{S}}_i}) \\
         & =P_{\pm,i}^n(x)+\tilde{Q}_{\pm,i}(x;\{f_{\pm,j}^\text{fluc}(t)\}_{j\in\mathcal{\tilde{S}}_i}),
    \end{aligned}
\end{equation}
where \smash{$f_{\pm,i}(t)=f_{\pm,i}^n+f_{\pm,i}^\text{fluc}(t)$}, with \smash{$f_{\pm,i}^\text{fluc}(t^n)=0$}.
In our particular case, it is sufficient to consider a piecewise constant $\tilde{Q}_i$.
Hence, the well-balanced reconstruction operator at time $t^{n+1}$ for the transport, denoted by $P_{\pm,i}^{n,*}$, reads
\begin{equation*}
    P_{\pm,i}^{n,*}(x)=P_{\pm,i}^{n}(x)+f_{\pm,i}^{n,*}-f_{\pm,i}^n.
\end{equation*}
For the transport stage of the first-order implicit scheme, we employ the backward Euler method and upwind numerical flux
\begin{equation}\label{FV-HKR transport implicit o1}
    % f_{k,i}^{n,*}=f_{k,i}^n-\frac{\Delta t}{\Delta x}v_k\Big(F_{k,i+\frac{1}{2}}^{n+1}-F_{k,i-\frac{1}{2}}^{n+1}\Big)+\frac{\Delta t}{\Delta x}v_k\Big(f_{k,i}^{e,n}(x_{i+\frac{1}{2}})-f_{k,i}^{e,n}(x_{i-\frac{1}{2}})\Big),
    f_{\pm,i}^{n,*}
    =
    f_{\pm,i}^n
    \mp
    \lambda \frac{\Delta t}{\Delta x} \Big(F_{\pm,i+\frac{1}{2}}^{n,*}-F_{\pm,i-\frac{1}{2}}^{n,*}\Big)
    \pm
    \lambda \frac{\Delta t}{\Delta x} \Big(f_{\pm,i}^{e,n}(x_{i+\frac{1}{2}})-f_{\pm,i}^{e,n}(x_{i-\frac{1}{2}})\Big),
\end{equation}
whereas for the second-order implicit scheme, the trapezoidal rule is used in conjunction with a second-order Strang splitting for the combination of the transport and the source-relaxation steps.
In any case, the source evolution and projection steps remain identical to those described previously.
However, thanks to the implicit treatment of the transport equation, the full scheme no longer has a stability condition on the time step, giving the opportunity to take larger time steps if desired.

\REVISORONE{Regarding free-flow boundary conditions, the reconstruction $P_{\pm,i}^n$ at time $t^n$ uses the ghost cells \eqref{eq:ghost_cells} and the incoming kinetic variables \eqref{eq:ghost_flux} of the explicit case. The time fluctuation of the incoming population in the ghost cell is actually an additional unknown of the linear system \eqref{FV-HKR transport implicit o1}, and it is set equal to that of the adjacent boundary cell. That is \MOD{to} say, we set $f_{+,0}^{n,*} - f_{+,0}^{n} = f_{+,1}^{n,*} - f_{+,1}^{n}$ at the left boundary (and symmetrically for $f_-$ at the right boundary). This zero-order extrapolation completes the linear system. Starting from the cell averages of a stationary solution, $f_{\pm,i}^{n,*} = f_{\pm,i}^n$ for all $i$, ghost cells included, satisfies both the interior equations (by the computation of the explicit case) and the boundary closure, and is therefore the solution of the linear system.}

We denote by $\mathcal{S}_1^\text{imp,FV}(\Delta t)=\mathcal{R}_2(\Delta t) \circ \mathcal{T}_1^\text{imp,FV}(\Delta t)$ the first-order implicit scheme; the following result shows that it is exactly well-balanced.

\begin{theorem}\label{teorema WB FV-HKR imp}
    % set of $\{u_i^e\}_{i=1}^{N_x}$ averages of a stationary solution $u^e$ of the balance law \eqref{balance law}, which satisfies \eqref{stationary solution}.
    % Let $P_{\pm,i}^n$ be an exactly well-balanced reconstruction operator for $m_\pm(u^e)$ in the $i$\textsuperscript{th} cell. Then
    % \begin{equation*}
    %     \forall i \in \{1, \dots, N_x\}, \quad
    %     \mathcal{S}_1^{\text{imp,FV}}(\Delta t)
    %     (u_i^e)=u_i^e.
    % \end{equation*}
    Consider $\{u_i^e\}_{i=1}^{N_x}$ a set of cell averages of a stationary solution~$u^e$ of the balance law \eqref{balance law}, which satisfies \eqref{stationary solution}. Let $\mathcal{S}_1^{\text{imp,FV}}$ be the scheme constructed above. Then
    \begin{equation*}
        \forall i \in \{1, \dots, N_x\}, \quad
        \mathcal{S}_1^{\text{imp,FV}}(\Delta t)
        (u_i^e)=u_i^e.
    \end{equation*}
\end{theorem}
\begin{proof}
    To prove this result, it suffices to observe that $f_{\pm,i}^{n,*} = f_{\pm,i}^{n}$ is solution of the system \eqref{FV-HKR transport implicit o1}, arguing similar arguments as in the proof of \cref{teorema WB FV-HKR}. Then, the projection-source step being the same as for the explicit case, the remainder of the proof is found in \cref{teorema WB FV-HKR}.
    % It is enough to prove that the vector $(u_i^e)_{i=1}^{N_x}$ is solution of the system \eqref{FV-HKR transport implicit o1}, because the projection-source step is the same as for the explicit case, treated in \cref{teorema WB FV-HKR}.
    % Evaluating in \eqref{FV-HKR transport implicit o1} $f_{\pm,i}^{n,*}=f_{\pm,i}^{n}=m_{\pm}(u_i^e)$. Then, for the $f_{-}$ equations (equivalently for $f_{+}$ case).
    % \[
    % \begin{aligned}
    %  & F_{-,i+\frac{1}{2}}^{n,*}-F_{-,i-\frac{1}{2}}^{n,*}-f_{-,i}^{e,n}(x_{i+\frac{1}{2}})+f_{-,i}^{e,n}(x_{i-\frac{1}{2}})= \\&P^{n,*}_{-,i+1}(x_{i+\frac{1}{2}})-P^{n,*}_{-,i}(x_{i-\frac{1}{2}})-f_{-,i}^{e,n}(x_{i+\frac{1}{2}})+f_{-,i}^{e,n}(x_{i-\frac{1}{2}})=\\&P^{n}_{-,i+1}(x_{i+\frac{1}{2}})-P^{n}_{-,i}(x_{i-\frac{1}{2}})-f_{-,i}^{e,n}(x_{i+\frac{1}{2}})+f_{-,i}^{e,n}(x_{i-\frac{1}{2}})=\\&m_{-}(u^{e,n}(x_{i+\frac{1}{2}}))-m_{-}(u^{e,n}(x_{i-\frac{1}{2}}))-m_{-}(u^{e,n}(x_{i+\frac{1}{2}}))+m_{-}(u^{e,n}(x_{i-\frac{1}{2}}))=0.
    % \end{aligned}
    % \]
\end{proof}

The proof for the second-order implicit scheme is analogous, applying the trapezoidal rule within the finite volume framework.

\subsection{Semi-Lagrangian Exactly Well-Balanced Kinetic Relaxation}
\label{sec:HKR-SL-WB}

We now extend this strategy by applying a Semi-Lagrangian (SL) scheme to the transport stage, instead of the finite volume framework. The SL approach evolves the solution using the method of characteristics, which are exactly known in the relaxed kinetic system since it is linear; specifically, they are straight, parallel lines for $f_+$ and $f_-$, respectively. A major advantage of this method is that it allows us to bypass the time-step restriction inherent to the explicit case. Because the characteristics traced backward from the target point can land at any arbitrary spatial position at the previous time level, a variable reconstruction procedure is necessary. For this purpose, we adopt the strategy of reconstructing the physical variables. Furthermore, we opt for a globally continuous reconstruction across the entire computational domain.

Let us begin by detailing the reconstruction procedure. Let $P_i^n$ be a well-balanced reconstruction operator for the stationary solution of the original hyperbolic system~\eqref{stationary solution}, as defined in \cref{def:well_balanced_reconstruction_operator}. We construct a globally continuous reconstruction operator, denoted by $P^n$, by defining it on the dual mesh as a convex linear combination of the adjacent local reconstruction operators. That is, for all $j \in \{\MOD{1}, \dots, N_x-1\}$ and $x \in (x_{j},x_{j+1})$, we define
\begin{equation*}
    \alpha_j(x)=\frac{x-x_{j}}{\Delta x},
\end{equation*}
and we set
\begin{equation}\label{globally continuous rec op}
    \restr{P^n}{(x_{j},x_{j+1})}(x)=\big(1-\alpha_j(x)\big)\ P_j^n(x)+\alpha_j(x) \ P_{j+1}^n(x).
\end{equation}
Depending on the prescribed boundary conditions, we must adapt the reconstruction procedure for any spatial points that fall outside the primal mesh. \REVISORONE{For periodic boundary conditions, $P^n$ is extended by periodicity. For free-flow boundary conditions, $P^n$ is extended outside $[x_1, x_{N_x}]$ by the well-balanced reconstruction of the boundary cell, that is,
    \begin{equation}
        \label{eq:SL_boundary}
        P^n(x) = P_1^n(x) \ \text{ for } x < x_1,
        \qquad
        P^n(x) = P_{N_x}^n(x) \ \text{ for } x > x_{N_x},
    \end{equation}
    which is the semi-Lagrangian counterpart of \eqref{eq:ghost_cells} -- \eqref{eq:ghost_flux}. Hence, no fluctuation with respect to the reconstruction of the boundary cell enters the domain. Note that \eqref{eq:SL_boundary} also applies for CFL numbers larger than one, where the foot of the characteristic may lie several cells outside the domain.}
Once this reconstruction of the physical variables is known at time $t^n$, to solve the transport stage~\eqref{eq: transport step fluctuations}, we evolve the kinetic variables through the characteristics.

As a first step, we define the kinetic variables between times $t^n$ and $t^{n+1}$ by evaluating the kinetic equilibrium functions along the characteristics, as follows:
\begin{align}
    \notag
    \forall x \in \mathbb{R}, \, \forall t\in[t^n,t^{n+1}], \quad
     & f_{\pm}(x,t)=m_{\pm}\Big(P^n(x\mp\lambda(t-t^n))\Big),
    \\
    \forall x \in \mathbb{R}, \, \forall t\in[t^n,t^{n+1}], \quad
     & f_{\pm,i}^{e,n}(x,t)=m_{\pm}\Big(u_i^{e,n}(x\mp\lambda(t-t^n))\Big). \label{SL-steady}
\end{align}
In \eqref{SL-steady}, $i$ corresponds to cell $I_i$, such that $x \in I_i$.

Let us recall that the transport equation \eqref{eq: transport step fluctuations} reads, with $\delta f_{\pm}=f_{\pm}-f_{\pm}^{e,n}$,
\begin{equation*}
    \partial_t (\delta f_{\pm}) \pm \lambda \partial_x (\delta f_{\pm})=0.
\end{equation*}
Solving the above equation using the SL framework, we find
\begin{equation*}
    \forall x \in \mathbb{R}, \, \forall t\in[t^n,t^{n+1}], \quad
    \delta f_{\pm}(x,t)=\delta f_{\pm}\big(x\mp\lambda(t-t^n),t^n\big).
\end{equation*}
For simplicity, for low order methods (1 and 2), we evolve the center $x_i$ of the cells in each time-step, as we could identify the point value of the center of the cells as the averages. Denoting $f_{\pm,i}=f_{\pm}(x_i)$, the above equation evaluated at $t=t^{n+1}$ and $x=x_i$ reads
\begin{align}
    \notag
    f_{\pm,i}^{n,*}
     & =
    f_\pm(x_i, t^n + \Delta t)
    =
    f_\pm\big(x_i \mp \lambda \Delta t, t^n\big)
    - f_{\pm, i}^{e, n}\big(x_i \mp \lambda \Delta t\big)
    + f_{\pm, i}^{e, n}\big(x_i\big) \\
     & =
    \label{eq:SL-transport-step}
    m_{\pm}\Big(P^n(x_i\mp\lambda\Delta t)\Big)
    - m_{\pm}\Big(u_i^{e,n}(x_i\mp\lambda\Delta t)\Big)
    + m_{\pm}\Big(u_i^{e,n}(x_i)\Big).
\end{align}
It is important to emphasize that, during the transport step of the kinetic variable fluctuations, we evaluate the projection of the stationary extension relative to the $i$\textsuperscript{th} cell, rather than the projection of the stationary state associated with the cell where the foot of the characteristic lies.

The evolution of the physical variables due to the source term and the subsequent projection step is performed exactly as in the finite volume case presented in \cref{sec:HKR-FV-WB}. Likewise, achieving higher-order schemes requires the use of high-order spatial reconstruction operators coupled with appropriate time-splitting techniques. Also, in third-order or higher, we have to evolve the values at quadrature nodes properly to compute the averages to solve the projection step.

We denote by $\mathcal{S}_1^\text{SL}(\Delta t)=\mathcal{R}(\Delta t) \circ \mathcal{T}_1^\text{SL}(\Delta t)$ the first-order SL scheme; the following result shows that it is exactly well-balanced.

\begin{theorem}\label{teorema WB SL-HKR}
    % Consider a set of $\{u_i^e\}_{i=1}^{N_x}$ averages of a stationary solution $u^e$ of the balance law \eqref{balance law}, which satisfies \eqref{stationary solution}. Let $P_i^n$ be an exactly well-balanced reconstruction operator for $u^e$ in the $i$\textsuperscript{th} cell, and $P^n$ the globally continuous reconstruction operator defined in \eqref{globally continuous rec op}. Then
    % \begin{equation*}
    %     \forall i \in \{1, \dots, N_x\}, \quad
    %     \mathcal{S}_1^{\text{SL}}(\Delta t)
    %     (u_i^e)=u_i^e.
    % \end{equation*}
    Consider $\{u_i^e\}_{i=1}^{N_x}$ a set of cell averages of a stationary solution~$u^e$ of the balance law \eqref{balance law}, which satisfies \eqref{stationary solution}. Let $\mathcal{S}_1^{\text{SL}}$ be the scheme constructed above. Then
    \begin{equation*}
        \forall i \in \{1, \dots, N_x\}, \quad
        \mathcal{S}_1^{\text{SL}}(\Delta t)
        (u_i^e)=u_i^e.
    \end{equation*}
\end{theorem}
\begin{proof}
    Like in the proof of \cref{teorema WB FV-HKR imp}, it is sufficient to take care of the transport step, as the relaxation-source step does not change with respect to the case of \cref{teorema WB FV-HKR}.
    To that end, we note that, starting from equilibrium data, \eqref{eq:SL-transport-step} rewrites as
    \begin{equation*}
        f_{\pm,i}^{n,*} = f_{\pm,i}^{n}
        + m_{\pm}\Big(P^n(x_i\mp\lambda\Delta t)\Big)
        - m_{\pm}\Big(u_i^{e,n}(x_i\mp\lambda\Delta t)\Big).
    \end{equation*}
    Therefore, to prove that $f_{\pm,i}^{n,*} = f_{\pm,i}^{n}$, it is enough to show that
    \begin{equation*}
        m_{\pm}\Big(P^n(x_i\mp\lambda\Delta t)\Big) = m_{\pm}\Big(u_i^{e,n}(x_i\mp\lambda\Delta t)\Big).
    \end{equation*}
    We prove this identity for the $f_{-}$ case ($f_{+}$ being analogous). Given $i \in \{1, \dots, N_x\}$, consider $j \in \{\MOD{1}, \dots, N_x\MOD{-1}\}$ as the cell in the dual mesh such that $x_i+\lambda\Delta t \in[x_j,x_{j+1}]$. Then, using the definition \eqref{globally continuous rec op} of $P^n$ and recalling that $P_i^n$ is well-balanced, we obtain the following chain of equalities:
    \[
        \begin{aligned}
            P^n(x_i+\lambda\Delta t)
             & =
            \Big(1-\alpha_j(x_i+\lambda\Delta t)\Big)\ P_j^n(x_i+\lambda \Delta t)
            +
            \alpha_j(x_i+\lambda\Delta t)\ P_{j+1}^n(x_i+\lambda \Delta t) \\
             & =
            \Big(1-\alpha_j(x_i+\lambda\Delta t)\Big)\ u_j^e(x_i+\lambda \Delta t)
            +
            \alpha_j(x_i+\lambda\Delta t)\ u_{j+1}^e(x_i+\lambda \Delta t) \\
             & =
            \Big(1-\alpha_j(x_i+\lambda\Delta t)\Big)\ u^e(x_i+\lambda \Delta t)
            +
            \alpha_j(x_i+\lambda\Delta t)\ u^e(x_i+\lambda \Delta t)       \\
             & =
            u^e(x_i+\lambda \Delta t)
            =
            u_i^{e,n}(x_i+\lambda \Delta t),
        \end{aligned}
    \]
    which proves that the transport step, and thus the scheme, is exactly well-balanced.
    \REVISORONE{If the foot of the characteristic lies outside $[x_1,x_{N_x}]$, the extension \eqref{eq:SL_boundary} gives $P^n(x_i+\lambda\Delta t)=P_{N_x}^n(x_i+\lambda\Delta t)=u^e(x_i+\lambda\Delta t)$, since $P_{N_x}^n$ is well-balanced, and the same conclusion holds.}
\end{proof}

\section{Numerical Experiments}
\label{sec:numerical experiments}

In this section, we present a series of numerical tests to validate the robustness, accuracy, and well-balanced properties of the proposed schemes. We consider three different physical models: the Burgers' equation with a source term, the shallow water equations with bottom topography, and the Euler equations with a gravitational potential. The different schemes that will be used are described and labelled in \cref{tab:schemes_summary}.

\begin{table}[!ht]
    \centering
    \begin{tabular}{ll}
        \toprule
        \textbf{Label} & \textbf{Scheme description}                           \\
        \midrule
        FV-HKR-O1-Exp  & WB, 1st-order explicit Finite Volume scheme with HKR. \\
        FV-HKR-O2-Exp  & WB, 2nd-order explicit Finite Volume scheme with HKR. \\
        FV-HKR-O3-Exp  & WB, 3rd-order explicit Finite Volume scheme with HKR. \\
        FV-HKR-O1-Imp  & WB, 1st-order implicit Finite Volume scheme with HKR. \\
        FV-HKR-O2-Imp  & WB, 2nd-order implicit Finite Volume scheme with HKR. \\
        SL-HKR-O1      & WB, 1st-order Semi-Lagrangian scheme with HKR.        \\
        \bottomrule
        \multicolumn{2}{l}{\footnotesize \textit{*WB: Well-Balanced; HKR: Hyperbolic Kinetic Relaxation.}}
    \end{tabular}
    \caption{Summary of the proposed numerical schemes and their characteristics.}
    \label{tab:schemes_summary}
\end{table}

The main objective of these experiments is threefold: first, to analyze the ability of the schemes to accurately capture complex physical dynamics, such as the formation of shock and rarefaction waves; second, to empirically verify the fulfillment of the exactly well-balanced property against various steady states; and third, to perform a convergence study confirming the theoretical order of accuracy of each method. For simulations lacking a known analytical solution, reference solutions have been computed using a \Error{third-order} explicit well-balanced finite volume scheme with the local Lax-Friedrichs numerical flux on extremely fine meshes. Furthermore, since most of the test cases require free-flow boundary conditions that allow waves to leave the computational domain without generating spurious reflections, we have implemented the sponge layer technique at the open boundaries\REVISORONE{, described in \cref{sec:explicit_WB_FV}; the well-balanced treatment of the ghost cells is described in \cref{sec:HKR-FV-WB,sec:HKR-SL-WB}}.

Below, we detail the three physical models of hyperbolic balance laws that will be used throughout the numerical experiments.
They are all written under the form~\eqref{balance law}.

\medskip
\noindent\textbf{Burgers' Equation with a Quadratic Source Term.}
The first model under consideration is the classical scalar Burgers' equation, augmented with a nonlinear source term. It is a benchmark for studying the formation of discontinuities, and validating numerical convergence rates in smooth regimes. The model is given by:
\begin{equation*}
    \partial_t u + \partial_x \left(\frac{u^2}{2}\right) = \alpha u^2,
    \label{eq:model_burgers}
\end{equation*}
where $u(x,t)$ represents the conserved variable (typically interpreted as a velocity profile) and $\alpha \in \mathbb{R}$ is a constant parameter that modulates the intensity of the source term.
The stationary solutions that our schemes are designed to preserve for this equation correspond to exponential branches of the form $u^e(x) = C e^{\alpha x}$, with $C \in \mathbb{R}$.

\medskip
\noindent\textbf{Shallow Water Equations (SWE) with Bathymetry.}
The second model corresponds to the one-dimensional SWE with variable bottom topography. The system can be written in vectorial form by taking $u = (h, q)^T$, resulting in:
\begin{equation*}
    \begin{dcases}
        \partial_t h + \partial_x q = 0, \\
        \partial_t q + \partial_x \left( \frac{q^2}{h} + \frac{1}{2}gh^2 \right) = g h \partial_x H,
    \end{dcases}
    \label{eq:model_swe}
\end{equation*}
where the physical variables involved are the water column height $h(x,t)$ and the discharge $q(x,t) = h(x,t)v(x,t)$, with $v(x,t)$ being the depth-averaged fluid velocity. The parameter $g$ denotes the acceleration due to gravity, and the known smooth function $H(x)$ represents the bathymetry, or bottom topography, of the channel.
For this system, the preserved steady states are characterized by a constant discharge, $q^e(x,t) = q_0$. Consequently, the water height $h(x)$ must satisfy the conservation of the Bernoulli invariant, given by
\begin{equation*}
    \frac{q_0^2}{2g(h^e)^2} + h^e - H = E_0,
\end{equation*}
where $E_0$ is the constant specific energy. Multiplying by $\Error{(h^e)^2}$, this yields the cubic polynomial equation:
\begin{equation*}
    (h^e)^3 + (-H(x) - E_0)(h^e)^2 + \frac{q_0^2}{2g} = 0.
\end{equation*}
% Depending on the local Froude number, the roots of this polynomial dictate whether the moving-water steady state represent subcritical or supercritical flow regimes.
In the simplest scenario where the fluid is at rest ($q_0 = 0$), the polynomial trivially reduces to $h^e(x) - H(x) = E_0$, naturally recovering the well known lake at rest equilibrium where the free surface elevation $\eta = h - H$ remains strictly constant across the domain.
\REVISORTWO{Otherwise, when $q_0 \neq 0$, the cubic polynomial has at most three real roots: a negative one (which is discarded), a supercritical one (where $0<h<h_c$), and a subcritical one (where $h>h_c$). In the previous definitions, $h_c=(\nicefrac{q_0^2}{g})^{\nicefrac13}$ is the critical height. In \cref{alg:WB_reconstruction}, we select the positive root whose flow regime, subcritical ($\mathrm{Fr}^2 = \nicefrac{q^2}{gh^3} < 1$) or supercritical ($\mathrm{Fr}^2>1$), coincides with that of the cell average, in the spirit of~\cite{XinShuNoe2011}. This rule is consistent since the two physical branches are separated by $h_c$. It is ambiguous only at a sonic point, where both roots are equal. This situation is addressed in \cref{subsec:test10_swe}.}

\medskip
\noindent\textbf{Euler Equations with a Gravitational Potential.}
The third and final model addresses the dynamics of compressible gases under the influence of an external gravitational field. Taking the vector of conserved variables $u = (\rho, q, E)^T$, the system is expressed as:
\begin{equation*}
    \begin{dcases}
        \partial_t \rho + \partial_x q = 0,                                               \\
        \partial_t q + \partial_x \left(\frac{q^2}{\rho} + p\right) = -\rho \partial_x H, \\
        \partial_t E + \partial_x \left(\frac{q}{\rho}(E + p)\right) = -q \partial_x H.
    \end{dcases}
    \label{eq:model_euler}
\end{equation*}
Here, $\rho(x,t)$ is the gas density, $q(x,t) = \rho(x,t) v(x,t)$ represents the momentum density (where $v(x,t)$ is the fluid velocity), $p(x,t)$ is the pressure, and $E(x,t)$ is the total energy per unit volume. The pressure is related to the other variables through the equation of state for an ideal gas, i.e.
\begin{equation*}
    p=(\gamma -1)\left(E - \frac{q^2}{2\rho} \right),
\end{equation*}
where $\gamma$ is the adiabatic index (the ratio of specific heats). The function $H(x)$ represents the given and smooth external gravitational potential.
In this context, multiple families of hydrostatic equilibria (with $q = 0$) arise.
We elect to make our numerical schemes preserve a specific two-parameter family of hydrostatic equilibria. These steady states are characterized by a fluid at rest ($q^e = 0$) and a balance between the internal pressure gradient and the external gravitational force, i.e., the relation $\partial_x p = -\rho \partial_x H$. Choosing a specific form for the density and integrating this momentum balance yields the two-parameter family of stationary solutions given by:
\begin{equation*}
    \rho^e(x) = C_1 e^{-H(x)}, \quad
    q^e(x)    = 0, \quad
    p^e(x)    = \rho^e(x) + C_2, \quad
    E^e(x)    = \frac{p^e(x)}{\gamma - 1},
\end{equation*}
where $C_1 > 0$ and $C_2 \in \mathbb{R}$ are the two independent constants that parameterize this family of equilibria.
We could have chosen to preserve a different family of hydrostatic equilibria, or even non-static isentropic equilibria like in~\cite{BerMicTho2026,MicTho2025}, but we opted for this one as it is the most classical and widely used in the literature, and it allows us to easily generate a wide variety of test cases by simply tuning the two parameters.

\medskip\subsection{Test 1: Transient Dynamics and Shock Capturing in Burgers' Equation} \label{subsec:test1_burgers}
This test evaluates the ability of the explicit FV kinetic schemes to capture transient shock waves and verifies the reduction of numerical diffusion at higher orders. The computational domain is $[a, b] = [-7.5, 7.5]$, simulated up to $T = 2.5$. We set $\text{CFL} = 0.9$ and $\alpha = 0.15$. Free-flow boundaries are used. The initial condition consists of two Gaussian pulses:
\begin{equation*}
    u(x,0) = 1.2 \exp\left(-(x + 3)^2\right) - 1.2 \exp\left(-(x - 3)^2\right).
\end{equation*}

\begin{SCfigure}
    \centering
    \includegraphics[width=0.6\textwidth]{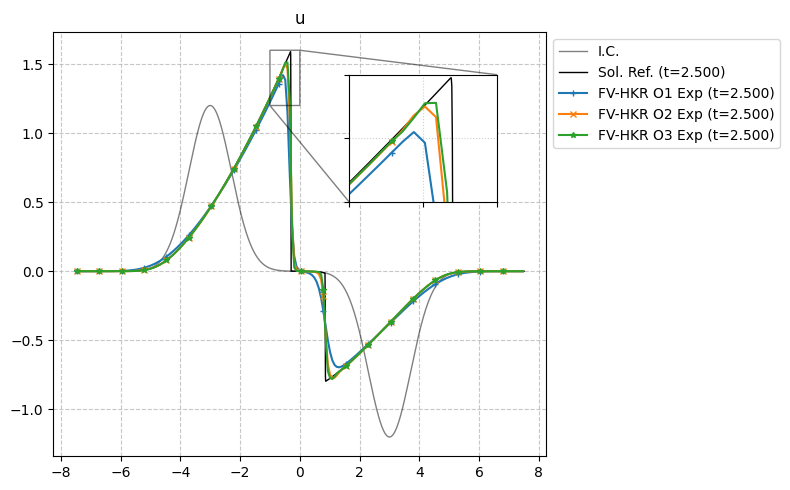}
    \caption{Numerical solutions for Test 1 (Burgers' equation) at $T=2.5$ using explicit FV-HKR schemes of orders 1, 2, and 3 with $N_x = 200$, $\Delta x=0.075$ and $\text{CFL}=0.9$.}
    \label{fig:test1_burgers}
\end{SCfigure}

The numerical results at $T = 2.5$ are shown in \cref{fig:test1_burgers}. The nonlinear flux steepens the initial pulses into sharp discontinuities moving towards one another. The first-order scheme exhibits significant numerical diffusion, smearing the shock fronts and underestimating peak amplitudes. The second-order and third-order scheme reduce this smearing, and achieve good agreement with the fine-mesh reference solution, without spurious oscillations.

\medskip
\subsection{Test 2: Blow-up, Implicit and Semi-Lagrangian Schemes for Burgers' Equation} \label{subsec:test2_burgers}
This test evaluates the performance of the implicit finite volume and Semi-Lagrangian kinetic schemes when handling shock formation and solutions approaching a blow-up regime. Driven by the quadratic source term, the initial profile is strongly amplified, which pushes the solution toward a finite-time blow-up. The computational domain is $[a, b] = [0, 5]$, simulated up to $T = 1.9$, just before the critical blow-up time. We set $\text{CFL} = 1$ and $\alpha = 0.5$. The initial condition consists of a clipped inverted parabola:
\begin{equation*}
    u(x,0) = \max\big(0, 1 - (x - 2)^2\big).
\end{equation*}

\begin{SCfigure}
    \centering
    \includegraphics[width=0.6\textwidth]{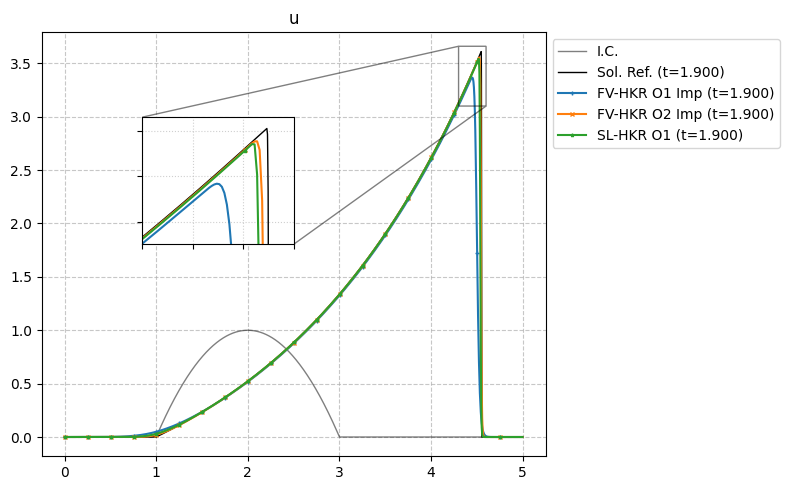}
    \caption{Numerical solutions for Test 2 (Burgers' equation) at $T=1.9$ using implicit FV-HKR (orders 1 and 2) and SL-HKR (order 1) schemes with $N_x = 1000$, $\Delta x=0.005$, and $\text{CFL}=1.0$.}
    \label{fig:test2_burgers}
\end{SCfigure}

The numerical results at $T = 1.9$ are shown in \cref{fig:test2_burgers}. The initial parabolic profile is advected and amplified by the source term, developing a shock wave on its right side as it approaches the blow-up. The first-order implicit scheme introduces the most numerical diffusion, smearing the shock front and underestimating the peak. The first-order Semi-Lagrangian scheme performs visibly better, reducing the diffusion compared to its finite volume counterpart. Finally, the second-order implicit scheme provides the sharpest profile.

\medskip
\subsection{Test 3: Impact of Large CFL Numbers on the First Order Implicit Finite Volume Scheme} \label{subsec:test3_burgers}
This test investigates the effect of using large time steps in the first-order implicit Finite Volume scheme. The computational domain is $[a, b] = [-1, 4]$, simulated up to $T = 1.5$. The spatial domain is discretized with $N_x = 4000$ cells. We test a range of CFL numbers from $1$ to $10$ and set the source term parameter to $\alpha = -0.5$. The initial condition is a rectangular pulse:
\begin{equation*}
    u(x,0) =
    \begin{cases}
        1,   & \text{if } 0 \leq x \leq 1, \\
        0.1, & \text{otherwise}.
    \end{cases}
\end{equation*}

\begin{SCfigure}
    \centering
    \includegraphics[width=0.6\textwidth]{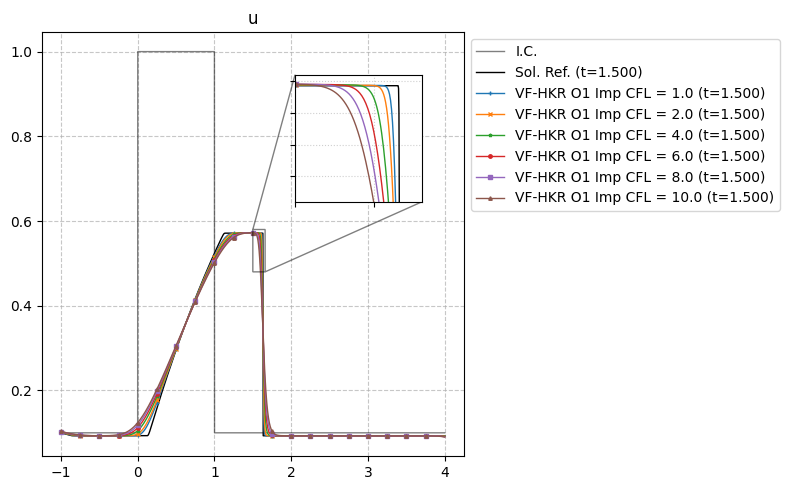}
    \caption{Numerical solutions for Test 3 (Burgers' equation) at $T=1.5$ using the implicit FV-HKR scheme of order 1 across different CFL numbers, with $N_x = 4000$ and $\Delta x=0.00125$.}
    \label{fig:test3_burgers}
\end{SCfigure}

The numerical results at $T = 1.5$ are shown in \cref{fig:test3_burgers}. The initial discontinuity evolves into a rarefaction wave on the left and a traveling shock wave on the right, while the negative source term dampens the overall amplitude of the solution over time. The results demonstrate that while the implicit scheme allows the use of CFL numbers well above $1$, increasing the time step introduces progressively larger amounts of numerical diffusion. This effect is clearly observed at the shock front, which becomes significantly more smeared as the CFL number increases from $1$ to $10$.

\medskip

\subsection{Test 4: Well-Balanced Property and Small Perturbations in Burgers' Equation} \label{subsec:test4_burgers}
This test verifies the well-balanced property of the proposed schemes and evaluates their accuracy when capturing small perturbations over a stationary background. To this end, we analyze two distinct scenarios: one focused on the exact preservation and long-term recovery of steady states, and another dedicated to the transient evolution of a traveling perturbation.

\smallskip

\textbf{Scenario A: Exact Preservation and Small Perturbation.}
In this scenario, we perform two separate numerical experiments. In the first experiment, we test the exact preservation of the stationary solution. The computational domain is $[a, b] = [-0.5, 0.5]$, simulated up to $T = 1$ with $N_x = 200$ cells, $\text{CFL} = 0.9$, and $\alpha = 1$. The initial condition is set exactly to the exponential equilibrium profile, proposed in \cite{castro2008well}:
\begin{equation*}
    u(x,0) = 0.1\exp(x).
\end{equation*}
As shown in \cref{tab:test4_exact}, the well-balanced schemes maintain the steady state to machine precision, confirming that the spatial reconstructions and the source term treatments are perfectly balanced.

In the second experiment, we introduce a localized transient perturbation over a different background steady state to verify empirical convergence rates. The domain is extended to $[a, b] = [-0.5, 1]$ and the simulation runs until $T = 2$, maintaining $\text{CFL} = 0.9$ and $\alpha = 1$. The initial condition includes a steep Gaussian bump:
\begin{equation*}
    u(x,0) = \exp(x) + 0.25 \exp({-1000\ (x - 0.8)^2}).
\end{equation*}
\cref{tab:test4_perturbed} displays the errors for this perturbed scenario, confirming that the perturbation successfully leaves the domain, and that the schemes are able to preserve the underlying, unperturbed steady solution up to machine accuracy.

\begin{table}[!ht]
    \centering
    % 1. Reduce el ancho de las minipages (ej. 0.4 o menos)
    \begin{minipage}[t]{0.4\textwidth}
        \centering
        % 2. El resizebox ahora escalará la tabla al nuevo ancho de la minipage
        \begin{tabular}{lr}
\toprule
Scheme & Err $L^1$ ($u$) \\
\midrule
FV-HKR O1 Exp & 2.85e-16 \\
FV-HKR O2 Exp & 2.83e-16 \\
FV-HKR O3 Exp & 8.95e-16 \\
FV-HKR O1 Imp & 2.42e-16 \\
FV-HKR O2 Imp & 2.80e-16 \\
SL-HKR O1 & 2.04e-16 \\
\bottomrule
\end{tabular}

        \caption{Test 4: $L^1$ errors for the unperturbed steady solution.}
        \label{tab:test4_exact}
    \end{minipage}
    % 3. CAMBIO CLAVE: Usa un espacio fijo en lugar de \hfill
    \hspace{2cm}
    \begin{minipage}[t]{0.4\textwidth}
        \centering
        \begin{tabular}{lr}
\toprule
Scheme & Err $L^1$ ($u$) \\
\midrule
FV-HKR O1 Exp & 8.41e-15 \\
FV-HKR O2 Exp & 1.73e-15 \\
FV-HKR O3 Exp & 3.73e-14 \\
FV-HKR O1 Imp & 5.81e-15 \\
FV-HKR O2 Imp & 4.51e-15 \\
SL-HKR O1 & 5.19e-15 \\
\bottomrule
\end{tabular}

        \caption{Test 4: $L^1$ errors for the perturbed steady solution, \\ after a long time.}
        \label{tab:test4_perturbed}
    \end{minipage}
\end{table}

\smallskip

\textbf{Scenario B: Transient Evolution of a Small Perturbation.}
In this second scenario, we compare the performance of the numerical models by tracking the movement of a perturbation traveling over a stationary background. The computational domain and parameters ($N_x = 200$, $\text{CFL} = 0.9$, and $\alpha = 1$) remain identical to those in the perturbed case of Scenario A. The background steady state is denoted by $u^e(x) = \exp(x)$. The initial condition $u(x,0)$ incorporates a similar localized perturbation, now centered at $x=0$:
\begin{equation*}
    u(x,0) = u^e(x) + 0.25 \exp({-1000 x^2}).
\end{equation*}

To evaluate the numerical schemes during the transient phase, the simulation is run up to $T = 0.3$, before the wave exits the domain. In order to effectively visualize the small moving perturbation without it being overshadowed by the background values, we plot the difference between the computed numerical solution and the exact stationary state, $u - u^e$, as depicted in \cref{fig:test4_burgers_perturbation}.

The results highlight the varying degrees of numerical diffusion introduced by the different schemes. The higher-order explicit methods, particularly the third-order scheme (FV-HKR O3 Exp), sharply capture the traveling shock and closely match the reference solution with minimal numerical dissipation. Conversely, the implicit scheme (FV-HKR O1 Imp) introduces significant smearing to the perturbation profile. The first-order explicit and semi-Lagrangian schemes display comparable, moderate levels of numerical diffusion.

\begin{SCfigure}
    \centering
    \includegraphics[width=0.6\textwidth]{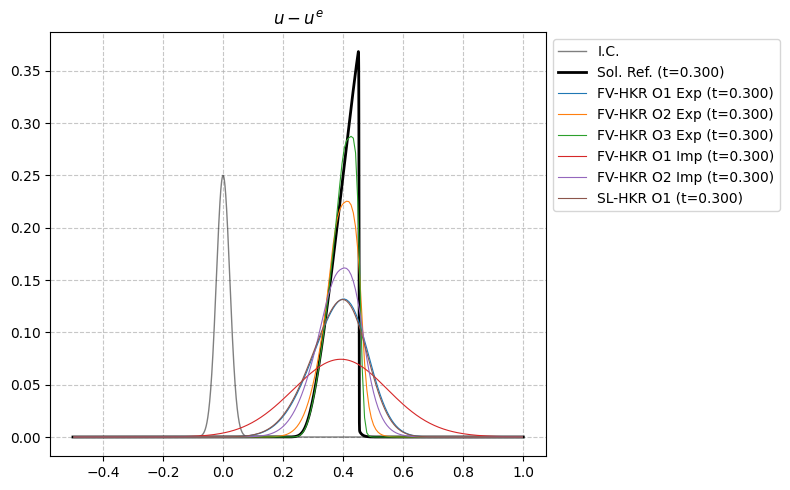}
    \caption{Numerical solutions for Test 4 (Burgers' equation) showing the transient perturbation ($u - u^e$) at $T=0.3$ with $N_x = 200$ and $\text{CFL}=0.9$.}
    \label{fig:test4_burgers_perturbation}
\end{SCfigure}

\medskip
\subsection{Test 5: Wave Scattering over a Submerged Bar for SWE} \label{subsec:test5_swe}
This test evaluates the ability of the explicit Finite Volume schemes to resolve smooth transmitted and reflected waves interacting with non-flat bottom topography. The computational domain is $[a, b] = [0, 50]$, simulated up to $T = 10$ using $N_x = 1000$ cells and $\text{CFL} = 0.8$. Free-flow boundaries are applied. The bathymetry consists of a submerged bar defined by:
\begin{equation*}
    H(x) =
    \begin{cases}
        -0.4 \cos^2\left(\frac{\pi (x - 32)}{4}\right), & \text{if } 30 \leq x \leq 34, \\
        0,                                              & \text{otherwise}.
    \end{cases}
\end{equation*}
The initial condition represents a Gaussian pulse on the free surface elevation $\eta = h - H$ over a fluid at rest, replicating the experimental setup from \cite{BejiBattjes1993}:
\begin{equation*}
    \eta(x,0) = 1 + 0.2 \exp({-(x - 5)^2}), \qquad q(x,0) = 0.
\end{equation*}

\begin{SCfigure}
    \centering
    \includegraphics[width=0.7\textwidth]{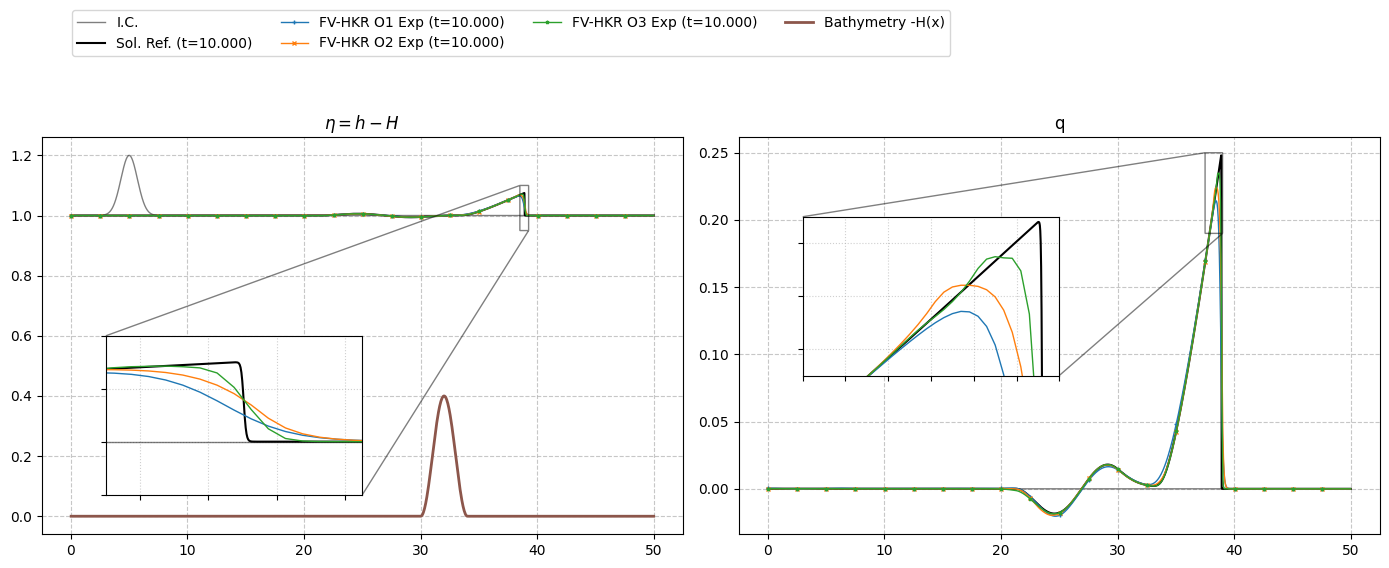}
    \caption{Numerical solutions for Test 5 (SWE) at $T=10.0$ showing the free surface elevation $\eta$ (left) and discharge $q$ (right). Computed using explicit FV-HKR schemes of orders 1, 2, and 3 with $N_x = 1000$, $\Delta x=0.05$, and $\text{CFL}=0.8$.}
    \label{fig:test5_swe}
\end{SCfigure}

The numerical results at $T = 10$ are shown in \cref{fig:test5_swe}. The left-travelling wave escapes from the computational domain whereas the initial right-traveling wave interacts with the submerged bar, splitting into a steepening transmitted wave and a smaller reflected wave moving leftwards. The first-order scheme exhibits numerical dissipation, underestimating the peak of the transmitted wave. The second-order and third-order schemes mitigate this dissipation, aligning more closely with the reference solution.

\medskip
\subsection{Test 6: Dam-Break over a Bottom Bump for SWE} \label{subsec:test6_swe}
This test examines the interaction of strong discontinuous waves with variable bottom topography. It assesses the ability of the explicit Finite Volume schemes to resolve complex flow structures without generating spurious oscillations. The computational domain is $[a, b] = [0, 20]$, simulated up to $T = 1$ using a relatively coarse mesh of $N_x = 200$ cells and $\text{CFL} = 0.5$, with free-flow boundary conditions. The bathymetry consists of a Gaussian bump centered at $x=12$, defined by:
\begin{equation*}
    H(x) = -\frac 1 2 \exp\left(-\frac 1 2(x - 12)^2\right).
\end{equation*}
The initial condition simulates a dam-break scenario over a fluid at rest, a standard benchmark in \cite{GoutalMaurel1997}, with a discontinuity in the free surface elevation $\eta = h - H$ at $x=10$:
\begin{equation*}
    \eta(x,0) =
    \begin{cases}
        2,   & \text{if } x \leq 10, \\
        0.6, & \text{otherwise},
    \end{cases}
    \qquad q(x,0) = 0.
\end{equation*}

\begin{SCfigure}
    \centering
    \includegraphics[width=0.7\textwidth]{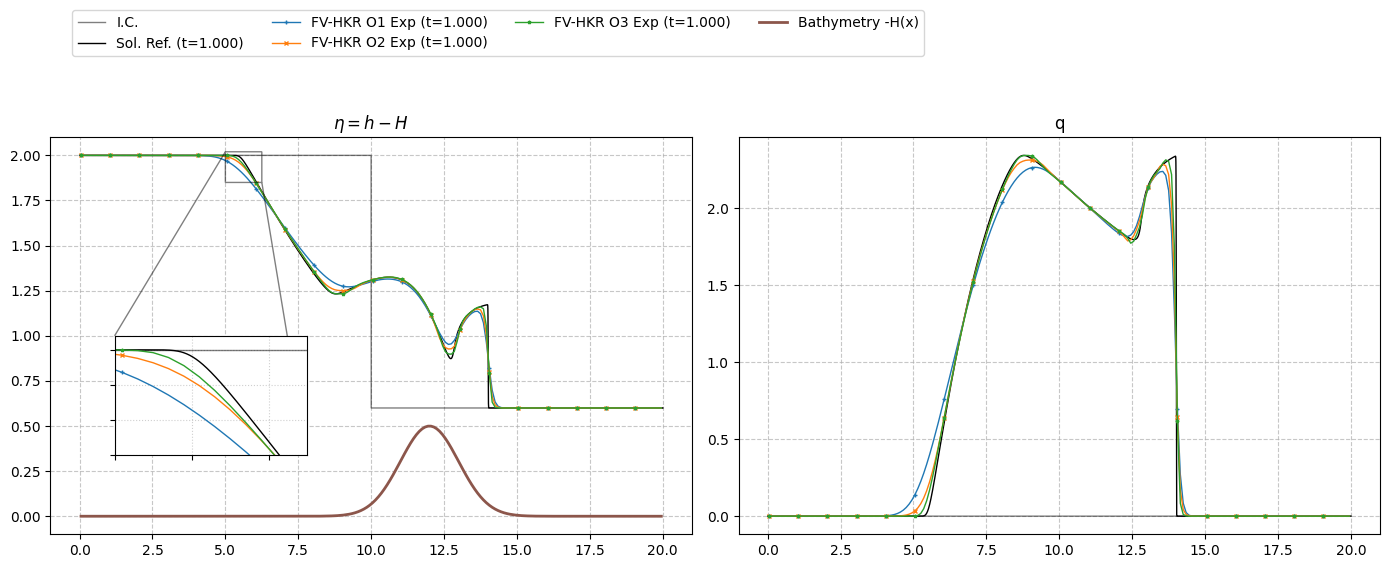}
    \caption{Numerical solutions for Test 6 (SWE) at $T=1$ showing the free surface elevation $\eta$ (left) and discharge $q$ (right). Computed using explicit FV-HKR schemes of orders 1, 2, and 3 with $N_x = 200$, $\Delta x=0.1$, and $\text{CFL}=0.5$.}
    \label{fig:test6_swe}
\end{SCfigure}

The numerical results at $T = 1$ are presented in \cref{fig:test6_swe}. This Riemann problem generates a right-traveling shock wave and a left-traveling rarefaction wave. As the shock propagates over the bump, it creates variations in the discharge. The first-order scheme suffers from excessive numerical diffusion, smoothing the head of the rarefaction wave and underpredicting the discharge peaks. In contrast, the second and third-order schemes again mitigate these shortcomings.

\medskip
\subsection{Test 7: Impact of Large CFL Numbers on the Semi-Lagrangian Scheme in SWE} \label{subsec:test7_swe}
This test investigates the effect of using large time steps on the numerical diffusion of the first-order Semi-Lagrangian scheme. The computational domain is $[a, b] = [0, 20]$, simulated up to $T = 2$. The spatial domain is discretized with $N_x = 2000$ cells. We test CFL numbers ranging from $1$ to $10$. The bathymetry represents a localized drop, defined by:
\begin{equation*}
    H(x) = -1 + 0.8 \exp\left(-(x - 10)^2\right).
\end{equation*}
The initial condition consists of a fluid at rest with a localized depression in the free surface elevation:
\begin{equation*}
    \eta(x,0) = 2 - 0.5 \exp\left(-2(x - 10)^2\right), \qquad q(x,0) = 0.
\end{equation*}

\begin{SCfigure}
    \centering
    \includegraphics[width=0.7\textwidth]{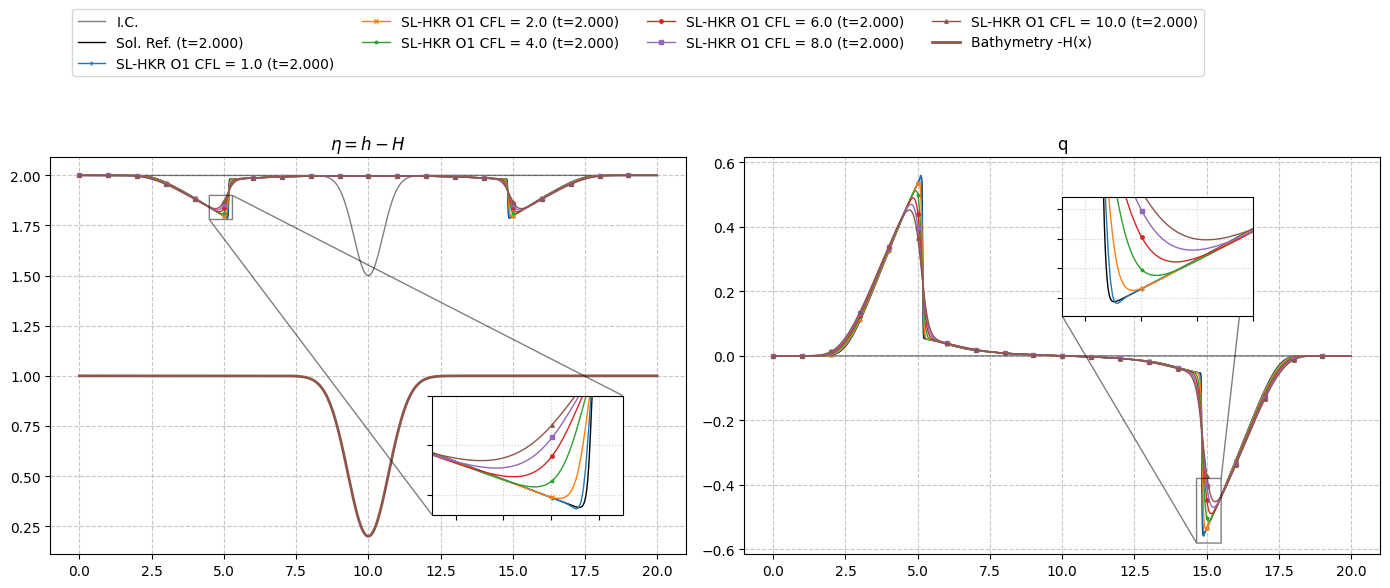}
    \caption{Numerical solutions for Test 7 (SWE) at $T=2$ showing the free surface elevation $\eta$ (left) and discharge $q$ (right). Computed using the SL-HKR scheme of order 1 across different CFL numbers, with $N_x = 2000$ and $\Delta x=0.01$.}
    \label{fig:test7_swe}
\end{SCfigure}

The numerical results at $T = 2$ are presented in \cref{fig:test7_swe}. The initial free surface depression collapses, generating two steep, outward-traveling waves. The interaction with the non-flat bathymetry produces similarly steep discharge profiles. The Semi-Lagrangian scheme demonstrates robust unconditional stability, generating no spurious reflections, even at a CFL of $10$. However, increasing the CFL number increases the numerical diffusion, as expected. This effect is clearly visible at the steep wave fronts, which become progressively more smeared as larger time steps are employed.

\medskip
\subsection{Test 8: Well-Balanced Property (Lake at Rest) over a Complex Bathymetry for the SWE} \label{subsec:test8_swe}
This test verifies the exact well-balanced property of the proposed schemes for the lake at rest equilibrium over a smooth but highly oscillatory bottom topography. To challenge the spatial reconstruction operators, the bathymetry is generated via a cubic spline interpolation of $40$ randomly distributed control points. The computational domain is $[a, b] = [-5, 5]$, discretized with $N_x = 200$ cells, and simulated up to $T = 1$ using $\text{CFL} = 0.9$. The initial condition represents water at rest, with a constant free surface:
\begin{equation*}
    \eta(x,0) = 1, \qquad q(x,0) = 0.
\end{equation*}
\cref{fig:test8_swe,tab:test8_swe} present the numerical solutions and the $L^1$ errors, respectively. As expected, despite the highly variable topography, the schemes are exactly well-balanced.

% {\color{red} Why to use CFL=0.9 for the implicit solvers?}

% Dejamos un poco de espacio antes
\vspace{1em}
\noindent
% Usamos [t] (top) y \linewidth en lugar de \textwidth
\begin{minipage}[t]{0.55\linewidth}
    \vspace{0pt} % TRUCO: Fuerza el anclaje en la parte superior exacta
    \centering
    \includegraphics[width=\linewidth]{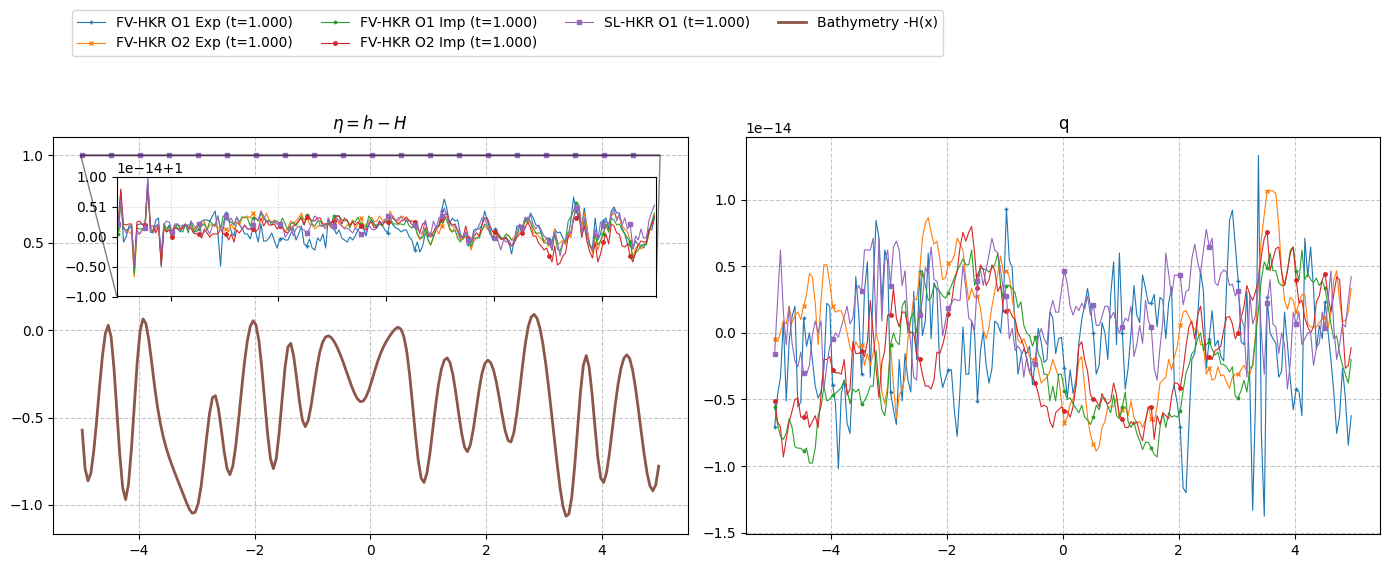}

    \captionof{figure}{Numerical solutions for Test 8 at $T=1$ displaying deviations strictly bounded to machine precision noise around the exact lake at rest steady state.}
    \label{fig:test8_swe}
\end{minipage}% <-- El símbolo % aquí es vital para que no haya un espacio extra oculto
\hfill
\begin{minipage}[t]{0.4\linewidth}
    \vspace{0pt} % TRUCO: Iguala el anclaje con la imagen de al lado
    \centering
    \begin{tabular}{lrr}
\toprule
Scheme & Err $L_1$ ($h$) & Err $L_1$ ($q$) \\
\midrule
FV-HKR O1 Exp & 1.74e-15 & 3.39e-15 \\
FV-HKR O2 Exp & 1.89e-15 & 3.45e-15 \\
FV-HKR O3 Exp & 6.01e-15 & 3.30e-14 \\
FV-HKR O1 Imp & 1.94e-15 & 4.30e-15 \\
FV-HKR O2 Imp & 1.86e-15 & 3.68e-15 \\
SL-HKR O1 & 1.94e-15 & 2.56e-15 \\
\bottomrule
\end{tabular}

    \captionof{table}{Test 8: $L^1$ errors for the exact stationary state in SWE ($T=1$).}
    \label{tab:test8_swe}
\end{minipage}
\vspace{1em} % Dejamos un poco de espacio después

\medskip
\subsection{Test 9: Small Perturbations over stationary solutions for the SWE} \label{subsec:test9_swe}
This set of tests investigates the capacity of the proposed schemes to accurately resolve small transient waves traveling over different steady-state backgrounds. Standard non-well-balanced schemes typically fail this test because their truncation errors, which scale with the mesh size, are much greater than the small physical perturbations. We analyze two distinct scenarios: a perturbation over a lake at rest state, and a perturbation over a moving-water subcritical flow. In both cases, the computational domain is $[a, b] = [-5, 5]$ and the simulation runs up to $T = 100$ to thoroughly test long term stability. The spatial domain is discretized using $N_x = 200$ cells, with $\text{CFL} = 0.9$ and free-flow boundary conditions.
In these tests, the bottom topography is defined by the smooth function $H(x) = 1 - 0.5 \ \exp({-2x^2})$.

\smallskip

\textbf{Scenario A: Perturbation over Lake at Rest.}
The background steady state is water at rest, with $\eta^e(x) = 1$ and $q^e(x) = 0$. We introduce a small Gaussian perturbation exclusively on the free surface, following the classical benchmark from \cite{LeVeque1998}:
\begin{equation*}
    \eta(x,0) = 1 + 0.05 \ \exp({-x^2}), \qquad q(x,0) = 0.
\end{equation*}
The initial perturbation splits into two small waves propagating in opposite directions out of the domain, leaving the original equilibrium state behind. As shown in \cref{tab:test9_lake}, all numerical schemes successfully recover the lake at rest state after the perturbations leave the domain, leaving residual errors on the order of machine precision.

\smallskip

\textbf{Scenario B: Perturbation over Subcritical Flow.}
The background steady state is now a subcritical flow characterized by a constant, nonzero discharge $q^e(x) = 1$ and an energy head parameter $E_0 = 0.5$. The initial condition is generated by perturbing the stationary water height $h^e(x)$ with a small Gaussian bump:
\begin{equation*}
    h(x,0) = h^e(x) + 0.05 \ \exp\Big{(}-\frac{(x + 2)^2}{2 \ (0.1)^2}\Big{)}, \qquad q(x,0) = 1.
\end{equation*}
This perturbation travels downstream.
\REVISORONE{Let us precisely state how the local stationary solutions $u_i^e$ of \cref{alg:WB_reconstruction} are actually computed for moving equilibria. In each cell, the discharge of the stationary solution is the cell average $q_i$, and its specific energy $E_i$ is determined by imposing that the quadrature average of $u_i^e$ coincides with $(h_i,q_i)$. With the midpoint rule used by the first- and second-order schemes, $E_i$ is given in closed form by the Bernoulli relation evaluated at $(h_i,q_i,H(x_i))$; with the two-point Gauss rule used by the third-order scheme, $E_i$ is obtained by a scalar Newton iteration, with tolerance $10^{-15}$. Then, $u_i^e$ is evaluated at the interfaces and quadrature nodes of the stencil by solving the Bernoulli cubic in closed form using Cardano formulas, selecting the positive root whose flow regime (subcritical or supercritical) coincides with that of the cell average. As a consequence, and as explained in \cref{rem:exactness}, the moving equilibria are preserved up to round-off errors, independently of the order of the reconstruction operator.}
\cref{tab:test9_subcritical} displays the $L^1$ errors for $h$ and $q$ measured at $T=100$, after the perturbation has left the domain. The results confirm that our exactly well-balanced schemes maintain the subcritical moving-water equilibrium to machine precision\REVISORONE{, for the first-, second- and third-order schemes alike}.

\vspace{1em}
\noindent
\begin{minipage}[t]{0.46\linewidth}
    \vspace{0pt} % TRUCO: Fuerza el anclaje en la parte superior exacta
    \centering
    \begin{tabular}{lrr}
\toprule
Scheme & Err $L_1$ ($h$) & Err $L_1$ ($q$) \\
\midrule
FV-HKR O1 Exp & 1.64e-15 & 8.69e-15 \\
FV-HKR O2 Exp & 2.69e-15 & 6.10e-15 \\
FV-HKR O3 Exp & 1.17e-14 & 2.67e-14 \\
FV-HKR O1 Imp & 2.73e-15 & 4.57e-15 \\
FV-HKR O2 Imp & 4.57e-15 & 8.07e-15 \\
SL-HKR O1 & 4.71e-15 & 9.92e-15 \\
\bottomrule
\end{tabular}

    \captionof{table}{Test 9: $L^1$ errors after a perturbation over the lake at rest state ($T=100$).}
    \label{tab:test9_lake}
\end{minipage}% <-- El símbolo % vital para no empujar la segunda caja
\hfill
\begin{minipage}[t]{0.46\linewidth}
    \vspace{0pt} % TRUCO: Iguala el anclaje con la tabla de al lado
    \centering
    \begin{tabular}{lrr}
\toprule
Scheme & Err $L_1$ ($h$) & Err $L_1$ ($q$) \\
\midrule
FV-HKR O1 Exp & 4.48e-16 & 1.76e-15 \\
FV-HKR O2 Exp & 3.75e-15 & 5.93e-15 \\
FV-HKR O3 Exp & 5.11e-14 & 2.28e-13 \\
FV-HKR O1 Imp & 6.68e-15 & 9.36e-15 \\
FV-HKR O2 Imp & 1.31e-14 & 9.98e-15 \\
SL-HKR O1 & 2.50e-15 & 9.61e-15 \\
\bottomrule
\end{tabular}

    \captionof{table}{Test 9: $L^1$ errors after a perturbation over the subcritical flow state ($T=100$).}
    \label{tab:test9_subcritical}
\end{minipage}
\vspace{1em}

\smallskip

\textbf{\REVISORONE{Scenario C: Independence with respect to the relaxation parameter.}}
\REVISORONE{To illustrate that the well-balanced property is uniform with respect to the relaxation time (see \cref{rem:AP}), we now run the unperturbed subcritical steady state of Scenario B up to $T=1$, with the same mesh and CFL number, for the four values $\omega \in \{1, 1.5, 1.9, 2-0.1\Delta t\}$ of the relaxation parameter, which correspond to relaxation times ranging from $\tau = \nicefrac{\Delta t}{2}$ to $\tau = \mathcal{O}(\Delta t^2)$. \cref{tab:test9_omega} reports the $L^1$ errors at the final time. For all the schemes and all the values of $\omega$, the moving equilibrium is preserved to machine precision, in agreement with the fact that the proofs of \cref{teorema WB FV-HKR,teorema WB FV-HKR imp,teorema WB SL-HKR} do not depend on $\omega$.}

\vspace{1em}
\noindent
\begin{minipage}{\linewidth}
    \centering
    \resizebox{\linewidth}{!}{
        \begin{tabular}{lcccccccc}
\toprule
\multirow{2}{*}{Scheme} & \multicolumn{2}{c}{$\omega=1$} & \multicolumn{2}{c}{$\omega=1.5$} & \multicolumn{2}{c}{$\omega=1.9$} & \multicolumn{2}{c}{$\omega=2-0.1\Delta t$} \\
\cmidrule(lr){2-3} \cmidrule(lr){4-5} \cmidrule(lr){6-7} \cmidrule(lr){8-9}
 & Err $L_1$ ($h$) & Err $L_1$ ($q$) & Err $L_1$ ($h$) & Err $L_1$ ($q$) & Err $L_1$ ($h$) & Err $L_1$ ($q$) & Err $L_1$ ($h$) & Err $L_1$ ($q$) \\
\midrule
FV-HKR O1 Exp & 3.20e-16 & 1.71e-15 & 1.01e-15 & 1.97e-15 & 3.85e-15 & 4.30e-15 & 1.22e-15 & 3.27e-15 \\
FV-HKR O2 Exp & 1.21e-15 & 1.03e-15 & 1.05e-15 & 1.55e-15 & 6.87e-16 & 2.80e-15 & 5.58e-15 & 1.83e-14 \\
FV-HKR O3 Exp & 1.27e-15 & 3.60e-15 & 6.94e-16 & 3.64e-15 & 4.93e-15 & 1.29e-14 & 2.70e-14 & 9.14e-14 \\
FV-HKR O1 Imp & 2.83e-15 & 4.69e-15 & 7.32e-16 & 1.60e-15 & 1.10e-15 & 2.61e-15 & 9.56e-15 & 2.67e-14 \\
FV-HKR O2 Imp & 1.10e-15 & 3.77e-15 & 9.17e-16 & 4.66e-15 & 3.23e-15 & 5.36e-15 & 1.08e-14 & 3.03e-14 \\
SL-HKR O1 & 1.80e-15 & 2.65e-15 & 6.93e-16 & 1.98e-15 & 2.85e-15 & 4.64e-15 & 4.11e-15 & 1.71e-14 \\
\bottomrule
\end{tabular}
    }
    \captionof{table}{\REVISORONE{Test 9, Scenario C: $L^1$ errors for the unperturbed subcritical steady state ($T=1$) for several values of the relaxation parameter $\omega$.}}
    \label{tab:test9_omega}
\end{minipage}
\vspace{1em}

\medskip
\subsection{Test 10: Perturbation over a Transcritical Flow for the SWE} \label{subsec:test10_swe}

This second perturbation test evaluates the robustness and exactly well-balanced nature of the kinetic schemes when dealing with transcritical moving-water equilibria. The flow regime transitions smoothly from subcritical to supercritical over the crest of the topography, passing exactly through a critical point. Resolving a transient perturbation that travels across this sonic transition without generating non-physical reflections or numerical instability is a notoriously difficult task for well-balanced schemes, see e.g.~\cite{MicBerClaFou2016}.

The computational domain is $[a, b] = [-5, 5]$, discretized with an odd number of cells $N_x = 201$ to ensure a cell center coincides exactly with the critical point at $x_c = 0$. The simulation runs up to $T = 60$ using $\text{CFL} = 0.9$ and free-flow boundary conditions. The bathymetry is the same as in \cref{subsec:test9_swe}.

The background steady state has a nonzero constant discharge $\Error{q^e = 1}$. The specific energy is calculated as $\Error{E_0} = -H(x_c) + \frac{3}{2} h_{crit}$, where $h_{crit} = (\Error{(q^e)^2} / g)^{1/3}$, ensuring that the flow is strictly transcritical. The exact water height $h^e(x)$ is computed by selecting the subcritical root of the Bernoulli cubic equation for $x \leq x_c$ and the supercritical root for $x > x_c$.
The initial condition introduces a Gaussian perturbation over the water height, located upstream in the subcritical region at $x_{pert} = -2$:
\begin{equation*}
    h(x,0) = h^e(x) + 0.05 \ \exp({-50(x + 2)^2}), \qquad q(x,0) = 1.
\end{equation*}

During the simulation, the perturbation splits. One wave travels upstream and exits the left boundary. The second wave travels downstream, accelerating as it crosses the critical point $x_c=0$ into the supercritical region, and eventually exits the right boundary. To prevent the root-finding algorithm of the exactly well-balanced reconstruction from failing when the transient perturbation crosses the sonic point, we locally enforce the unperturbed exact stationary solution $u^e(x)$ as the steady state at the problematic cell. This localized treatment is applied at the central cell for the 1st and 2nd-order schemes, and across the three central cells for the 3rd-order scheme.

% El parámetro [1.0] indica que el caption tendrá el mismo ancho que la tabla.
% Puedes subirlo a [1.5] o [2.0] si quieres que el texto ocupe más espacio horizontal.
\begin{SCtable}
    \centering
    \begin{tabular}{lrr}
\toprule
Scheme & Err $L_1$ ($h$) & Err $L_1$ ($q$) \\
\midrule
FV-HKR O1 Exp & 1.31e-15 & 2.99e-15 \\
FV-HKR O2 Exp & 1.44e-15 & 4.76e-15 \\
FV-HKR O3 Exp & 7.30e-15 & 3.54e-14 \\
FV-HKR O1 Imp & 1.65e-15 & 6.65e-15 \\
FV-HKR O2 Imp & 1.33e-15 & 6.91e-15 \\
SL-HKR O1 & 9.99e-16 & 4.03e-15 \\
\bottomrule
\end{tabular}

    \caption{Test 10: $L^1$ errors after a perturbation over the transcritical flow state ($T=60$).}
    \label{tab:test10_transcritical}
\end{SCtable}

\cref{tab:test10_transcritical} presents the $L^1$ errors for $h$ and $q$ measured at the end of the simulation, once the transient waves have entirely left the computational domain. The results confirm that, aided by the local treatment at the sonic point, all proposed kinetic schemes accurately preserve the transcritical equilibrium state up to machine precision, effectively resolving the complex nonlinear wave interactions.

\medskip
\subsection{Test 11: Empirical Order Convergence Analysis for the SWE} \label{subsec:test11_swe}
This last test on the SWE assesses the empirical order of accuracy for all six proposed kinetic schemes (explicit, implicit, and semi-Lagrangian). To properly measure the convergence rates, the simulation must remain in a smooth transient regime before any shock waves develop. Therefore, the final simulation time is restricted to $T = 0.3$.
The computational domain is set to $[a, b] = [-5, 5]$ with a CFL number of $0.9$. The bathymetry is defined as in \cref{subsec:test9_swe}.
The initial condition is smooth, with the fluid initially at rest:
\begin{equation*}
    h(x,0) = 1 + e^{-x^2}, \qquad q(x,0) = 0.
\end{equation*}

\begin{table}[!ht]
    \centering
    % ROW 1: Explicit O1 and Explicit O2
    \begin{minipage}[t]{0.45\textwidth}
        \centering
        \resizebox{\textwidth}{!}{
            \begin{tabular}{rrlrl}
\toprule
$N_x$ & Err $L_1$ ($h$) & Ord $L_1$ ($h$) & Err $L_1$ ($q$) & Ord $L_1$ ($q$) \\
\midrule
50 & 0.107325 & - & 0.504941 & - \\
100 & 0.050750 & 1.080504 & 0.253125 & 0.996262 \\
200 & 0.026070 & 0.960993 & 0.128392 & 0.979296 \\
400 & 0.013038 & 0.999633 & 0.063069 & 1.025557 \\
800 & 0.006345 & 1.039185 & 0.029789 & 1.082129 \\
\bottomrule
\end{tabular}

        }
        \caption{Test 11: order for the FV-HKR-O1-Exp scheme}
        \label{tab:test11_order_FV-HKR-O1-Exp}
    \end{minipage}
    \hfill
    \begin{minipage}[t]{0.45\textwidth}
        \centering
        \resizebox{\textwidth}{!}{
            \begin{tabular}{rrlrl}
\toprule
$N_x$ & Err $L_1$ ($h$) & Ord $L_1$ ($h$) & Err $L_1$ ($q$) & Ord $L_1$ ($q$) \\
\midrule
50 & 0.145845 & - & 0.656847 & - \\
100 & 0.050665 & 1.525360 & 0.236157 & 1.475808 \\
200 & 0.015605 & 1.699021 & 0.073160 & 1.690630 \\
400 & 0.004695 & 1.732911 & 0.020429 & 1.840412 \\
800 & 0.001163 & 2.013209 & 0.004426 & 2.206628 \\
\bottomrule
\end{tabular}

        }
        \caption{Test 11: order for the FV-HKR-O2-Exp scheme}
        \label{tab:test11_order_FV-HKR-O2-Exp}
    \end{minipage}

    \vspace{0.2cm} % Spacing between rows

    % ROW 2: Explicit O3 and Implicit O1
    \begin{minipage}[t]{0.45\textwidth}
        \centering
        \resizebox{\textwidth}{!}{
            \begin{tabular}{rrlrl}
\toprule
$N_x$ & Err $L_1$ ($h$) & Ord $L_1$ ($h$) & Err $L_1$ ($q$) & Ord $L_1$ ($q$) \\
\midrule
50 & 0.072543 & - & 0.322786 & - \\
100 & 0.012993 & 2.481057 & 0.060916 & 2.405678 \\
200 & 0.002628 & 2.305665 & 0.012362 & 2.300895 \\
400 & 0.000403 & 2.704671 & 0.001893 & 2.706970 \\
800 & 0.000056 & 2.860462 & 0.000262 & 2.855699 \\
\bottomrule
\end{tabular}

        }
        \caption{Test 11: order for the FV-HKR-O3-Exp scheme}
        \label{tab:test11_order_FV-HKR-O3-Exp}
    \end{minipage}
    \hfill
    \begin{minipage}[t]{0.45\textwidth}
        \centering
        \resizebox{\textwidth}{!}{
            \begin{tabular}{rrlrl}
\toprule
$N_x$ & Err $L_1$ ($h$) & Ord $L_1$ ($h$) & Err $L_1$ ($q$) & Ord $L_1$ ($q$) \\
\midrule
50 & 0.396549 & - & 2.806795 & - \\
100 & 0.260167 & 0.608060 & 1.859321 & 0.594148 \\
200 & 0.154762 & 0.749390 & 1.108791 & 0.745788 \\
400 & 0.085888 & 0.849513 & 0.613520 & 0.853805 \\
800 & 0.045494 & 0.916794 & 0.323568 & 0.923041 \\
\bottomrule
\end{tabular}

        }
        \caption{Test 11: order for the FV-HKR-O1-Imp scheme}
        \label{tab:test11_order_FV-HKR-O1-Imp}
    \end{minipage}

    \vspace{0.2cm} % Spacing between rows

    % ROW 3: Implicit O2 and Semi-Lagrangian O1
    \begin{minipage}[t]{0.45\textwidth}
        \centering
        \resizebox{\textwidth}{!}{
            \begin{tabular}{rrlrl}
\toprule
$N_x$ & Err $L_1$ ($h$) & Ord $L_1$ ($h$) & Err $L_1$ ($q$) & Ord $L_1$ ($q$) \\
\midrule
50 & 0.239188 & - & 1.035790 & - \\
100 & 0.090327 & 1.404925 & 0.404688 & 1.355851 \\
200 & 0.032422 & 1.478191 & 0.145498 & 1.475813 \\
400 & 0.009275 & 1.805594 & 0.040594 & 1.841667 \\
800 & 0.002114 & 2.133040 & 0.010269 & 1.982901 \\
\bottomrule
\end{tabular}

        }
        \caption{Test 11: order for the FV-HKR-O2-Imp scheme}
        \label{tab:test11_order_FV-HKR-O2-Imp}
    \end{minipage}
    \hfill
    \begin{minipage}[t]{0.45\textwidth}
        \centering
        \resizebox{\textwidth}{!}{
            \begin{tabular}{rrlrl}
\toprule
$N_x$ & Err $L_1$ ($h$) & Ord $L_1$ ($h$) & Err $L_1$ ($q$) & Ord $L_1$ ($q$) \\
\midrule
50 & 0.148774 & - & 1.214817 & - \\
100 & 0.099366 & 0.582303 & 0.704989 & 0.785066 \\
200 & 0.057801 & 0.781647 & 0.377313 & 0.901838 \\
400 & 0.031678 & 0.867603 & 0.197131 & 0.936607 \\
800 & 0.016447 & 0.945666 & 0.099355 & 0.988492 \\
\bottomrule
\end{tabular}

        }
        \caption{Test 11: order for the SL-HKR-O1 scheme}
        \label{tab:test11_order_SL-HKR-O1}
    \end{minipage}

\end{table}

We display the $L^1$ errors and the corresponding empirical convergence orders for the water height $h$ and the discharge $q$. The results are organized across all spatial resolutions ($N_x$ from $50$ to $800$). As expected, the empirical orders successfully match their theoretical counterparts: the first-order schemes converge with $\mathcal{O}(\Delta x)$, the second-order schemes reach $\mathcal{O}(\Delta x^2)$, and the third-order explicit scheme achieves its designated $\mathcal{O}(\Delta x^3)$ accuracy. This validates the high-order accurate space and time discretizations employed in the kinetic relaxation framework. The corresponding tables are shown in \cref{tab:test11_order_FV-HKR-O1-Exp,tab:test11_order_FV-HKR-O2-Exp,tab:test11_order_FV-HKR-O3-Exp,tab:test11_order_FV-HKR-O1-Imp,tab:test11_order_FV-HKR-O2-Imp,tab:test11_order_SL-HKR-O1}.
\REVISORONE{These tables are complemented by the log-log plots in \cref{fig:convergence_slopes}, which illustrate the $L^1$ error decay for the water height $h$ and the discharge $q$ as a function of the mesh resolution $N_x$ in a log-log scale.}

\begin{SCfigure}
    \centering
    \includegraphics[width=0.7\textwidth]{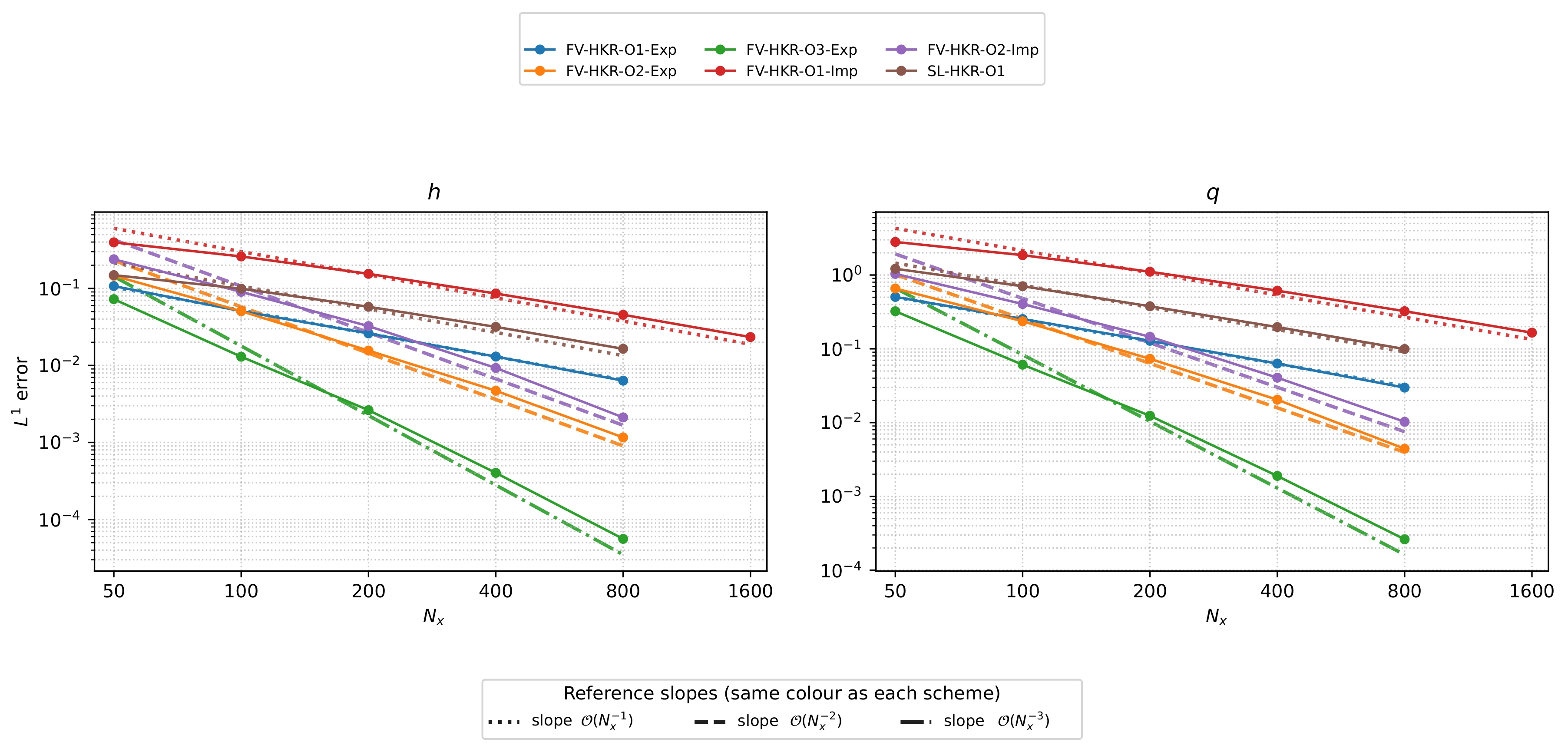}
    \caption{Empirical convergence rates for Test 11: $L^1$ errors for $h$ (left) and $q$ (right) compared with the theoretical reference slopes of convergence.}
    \label{fig:convergence_slopes}
\end{SCfigure}

\medskip
\subsection{Test 12: Euler equations with gravity: Shu-Osher Problem with a Gravitational Potential} \label{subsec:test12_euler}

This test, performed on the Euler equations with a gravitational potential, evaluates the ability of the explicit Finite Volume schemes to resolve the complex interaction between a strong shock wave, high-frequency density fluctuations, and an external gravitational field. The computational domain is $[a, b] = [-5, 5]$, simulated up to $T = 1.8$ with $N_x = 500$ cells, $\text{CFL} = 0.7$, and an adiabatic index $\gamma = 1.4$. Free-flow boundary conditions are applied. The gravitational potential features a smooth bump in the center of the domain:
\begin{equation*}
    H(x) =
    \begin{cases}
        0.2 - 0.05 x^2, & \text{if } {-2} < x < 2, \\
        0,              & \text{otherwise}.
    \end{cases}
\end{equation*}
The initial condition consists of a Mach 3 shock wave propagating to the right into a sinusoidal density field where the fluid is initially at rest. In this context, the Mach number 3 signifies that the shock front travels at three times the local speed of sound, characterizing a high-intensity compression regime. The initial state modifies the classical problem from \cite{ShuOsher1989} by adding a gravitational potential as in \cite{XinShu2012}:
\begin{equation*}
    (\rho, v, p)(x,0) =
    \begin{cases}
        (3.857143,\ 2.629369,\ 10.33333), & \text{if } x < -4, \\
        (1 + 0.2 \sin(5x),\ 0,\ 1),       & \text{otherwise}.
    \end{cases}
\end{equation*}
Note that the conserved variables $q$ and $E$ are computed from these primitive variables.

\vspace{1em}
\noindent
\begin{minipage}[c]{0.65\linewidth}
    \centering
    \includegraphics[width=\linewidth]{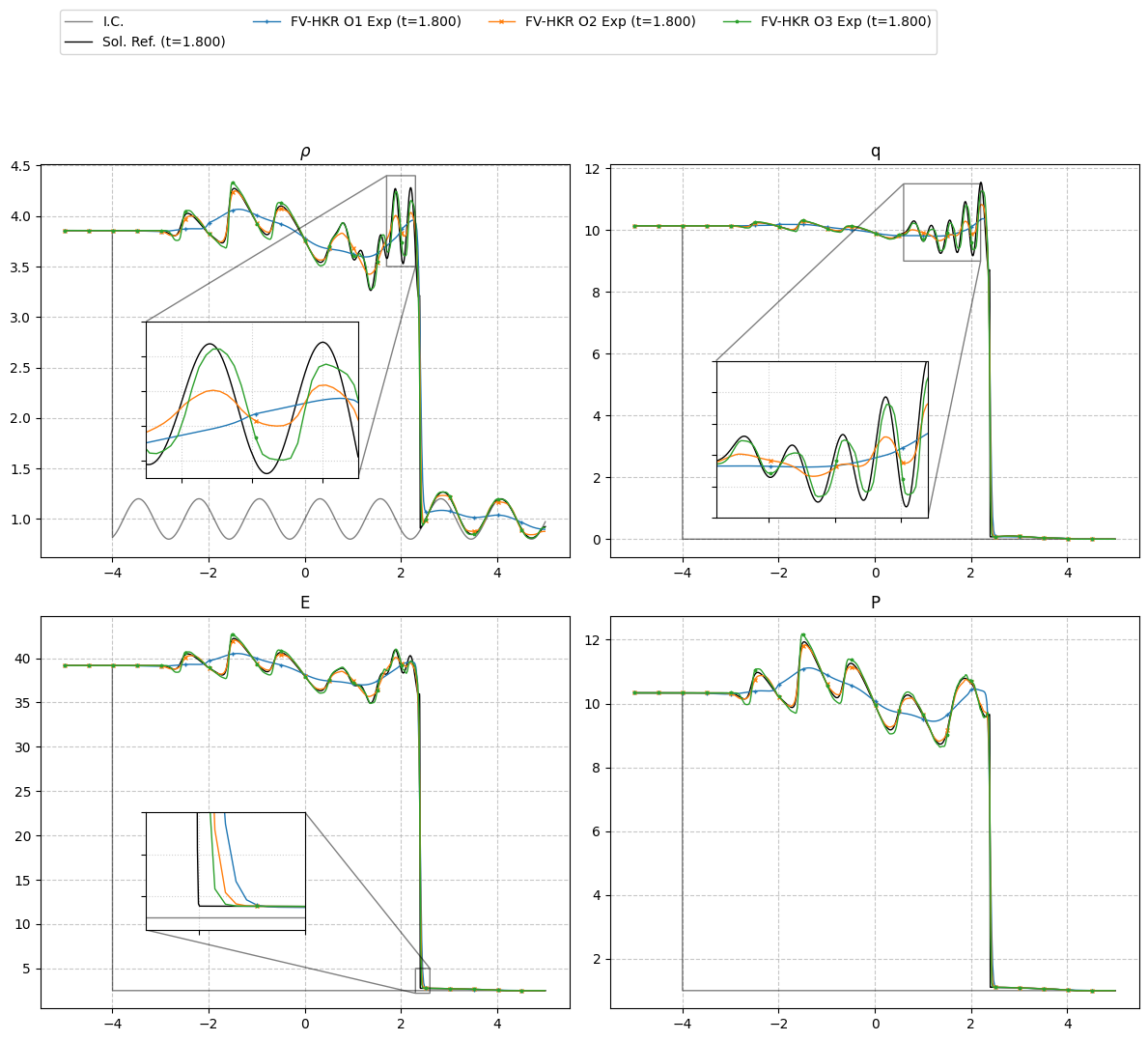}
\end{minipage}%
\hfill
\begin{minipage}[c]{0.30\linewidth}
    \captionof{figure}{Numerical solutions for Test 12 (Euler equations) at $T=1.8$ showing density $\rho$, momentum $q$, energy $E$, and pressure $p$. Computed using explicit FV-HKR schemes of orders 1, 2, and 3 with $N_x = 500$, $\Delta x=0.02$, and $\text{CFL}=0.7$.}
    \label{fig:test12_euler}
\end{minipage}
\vspace{1em}

The numerical results at $T = 1.8$ are presented in \cref{fig:test12_euler}. As the shock propagates, it compresses the density waves and simultaneously interacts with the gravitational potential bump. This interaction generates highly oscillatory physical structures behind the shock front, coupled with macroscopic flow variations driven by the source term. The first-order scheme exhibits excessive numerical diffusion, almost completely smearing out the high-frequency physical oscillations. The second-order scheme significantly improves wave capture but still visibly dampens the local extrema. The third-order scheme proves to be highly robust and accurate; it captures both the amplitude and phase of the entropy waves as well as the steep shock front, despite the challenging nonlinear source term.

\medskip
\subsection{Test 13: Point Blast in a Stratified Atmosphere} \label{subsec:test13_euler}
This test evaluates the capability of the unconditionally stable kinetic schemes (implicit Finite Volume and Semi-Lagrangian) to resolve strong shock waves and contact discontinuities generated by a localized energy explosion. The test is performed in a stratified atmosphere subject to a constant gravitational field, modeled by the linear potential $H(x) = x$. The computational domain is $[a, b] = [0, 1]$, simulated up to $T = 0.1$ with a high-resolution mesh of $N_x = 2000$ cells, $\text{CFL} = 1$, and an adiabatic index $\gamma = 1.4$. Free-flow boundary conditions are applied.

The initial condition consists of a fluid at rest in an isothermal hydrostatic state where $\rho = \exp(-x)$ and $q=0$, perturbed by a localized pressure spike at the center of the domain ($x_c = 0.5$) with a radius of $r_c = 0.05$ and an overpressure of $\Delta p = 10$:
\begin{equation*}
    (\rho, v, p)(x,0) =
    \begin{cases}
        (\exp({-x}),\ 0,\ \exp({-x}) + 10), & \text{if } |x - 0.5| \leq 0.05, \\
        (\exp({-x}),\ 0,\ \exp({-x})),      & \text{otherwise}.
    \end{cases}
\end{equation*}
The conserved variables $q$ and $E$ are once again initialized from these primitive variables.

\begin{SCfigure}
    \centering
    \includegraphics[width=0.7\textwidth]{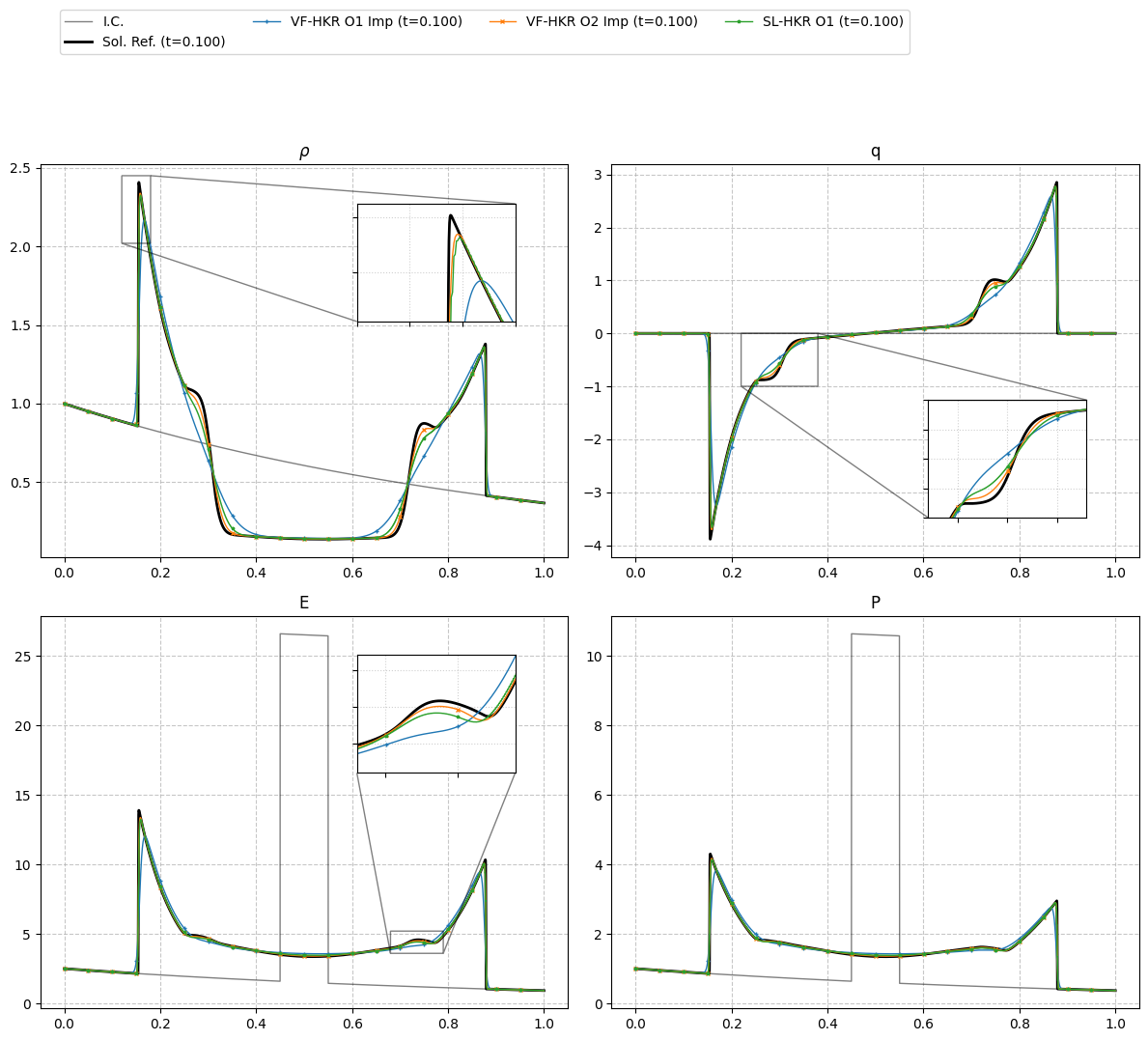}
    \caption{Numerical solutions for Test 13 (Euler equations) at $T=0.1$ showing density $\rho$, momentum $q$, energy $E$, and pressure $p$. Computed using the implicit FV-HKR (orders 1 and 2) and SL-HKR (order 1) schemes with $N_x = 2000$, $\Delta x=0.0005$, and $\text{CFL}=1$.}
    \label{fig:test13_euler}
\end{SCfigure}

The numerical results at $T = 0.1$ are shown in \cref{fig:test13_euler}. The central overpressure drives two strong shock waves propagating outward in both directions, leaving behind a low-density, high-energy core bounded by contact discontinuities. As observed in the zoomed-in regions, the first-order implicit scheme suffers from significant numerical diffusion, heavily smearing the contact discontinuities and rounding off the sharp density peaks. The first-order Semi-Lagrangian scheme performs noticeably better, reducing the diffusion compared to its finite volume counterpart. The second-order implicit scheme provides the most accurate resolution, capturing the shock fronts and the intricate internal wave structures with minimal diffusion, reaching a good accuracy.
We note that the contact waves are visibly more smeared than the shock waves. This is due to our two-velocity kinetic representation, which effectively follows the fastest speeds (i.e., the shock waves, due to the subcharacteristic condition) and thus introduces more numerical diffusion on the slower contact waves. A more sophisticated kinetic representation with additional velocities could potentially mitigate this issue, but it is beyond the scope of this work.

\medskip
\subsection{Test 14: Impact of Large CFL Numbers on the Implicit Scheme of order 2 for the Sod Tube Problem with Gravitational Potential} \label{subsec:test14_euler}
This test investigates the behavior and temporal accuracy of the second-order implicit Finite Volume scheme when operating at large CFL numbers. We employ the classic Sod shock tube problem, but modified to include a constant gravitational field as in \cref{subsec:test13_euler}. The computational domain is $[a, b] = [0, 1]$, simulated up to $T = 0.2$. To properly isolate the temporal truncation errors from the spatial ones, we use a very fine mesh of $N_x = 5000$ cells. The adiabatic index is $\gamma = 1.4$, and free-flow boundaries are applied. The initial condition consists of the classical shock tube \cite{Sod1978} adapted with a gravitational source term \cite{XinShu2012}:
\begin{equation*}
    (\rho, v, p)(x,0) =
    \begin{cases}
        (1,\ 0,\ 1),       & \text{if } x < 0.5, \\
        (0.125,\ 0,\ 0.1), & \text{otherwise}.
    \end{cases}
\end{equation*}

\begin{SCfigure}
    \centering
    \includegraphics[width=0.7\textwidth]{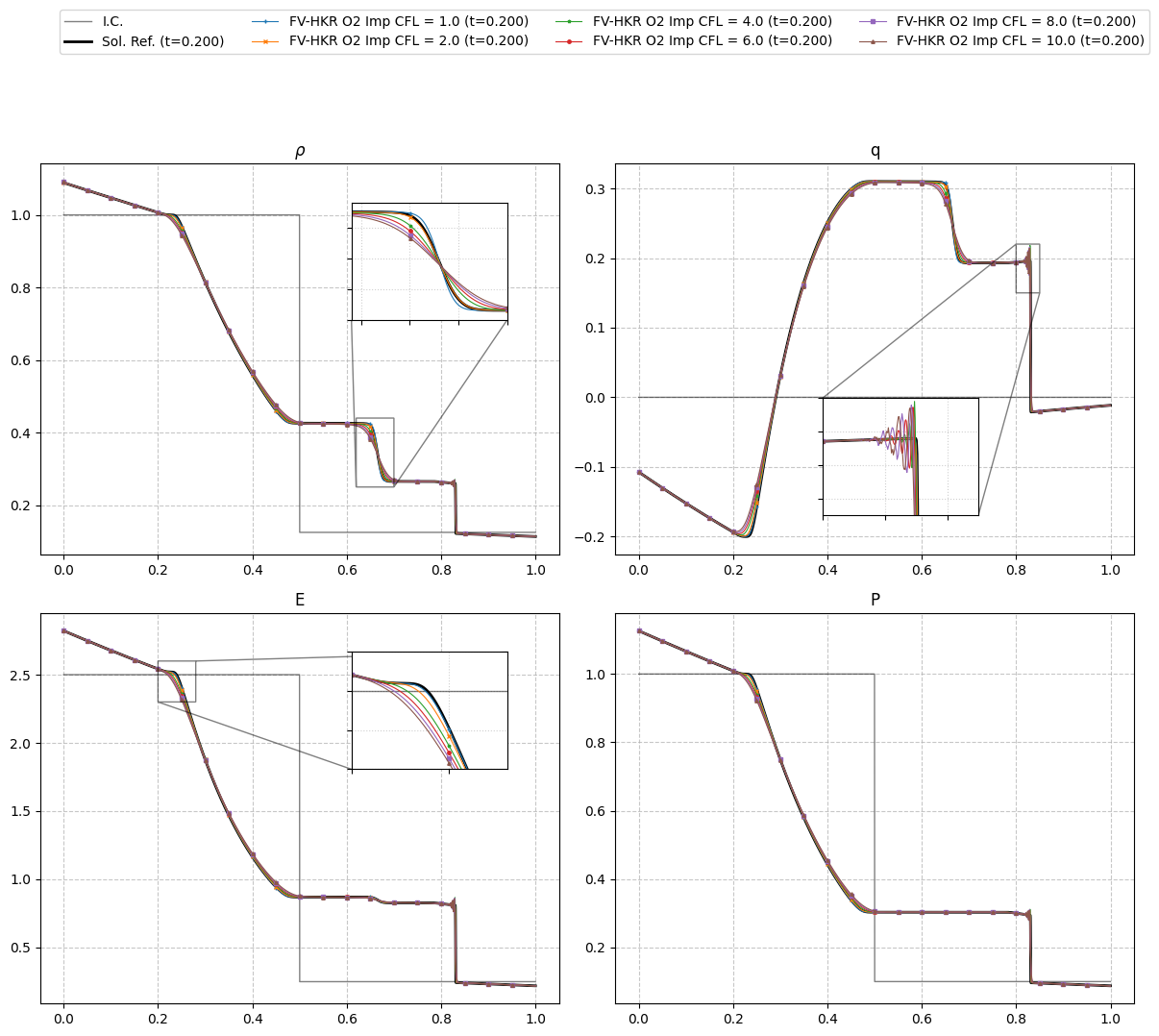}
    \caption{Numerical solutions for Test 14 (Euler equations) at $T=0.2$ under a linear gravitational potential. Computed using the implicit FV-HKR scheme of order 2 across varying CFL numbers, with $N_x = 5000$ and $\Delta x=0.0002$.}
    \label{fig:test14_euler}
\end{SCfigure}

The numerical results at $T = 0.2$ are shown in \cref{fig:test14_euler}. The Riemann problem develops the standard wave pattern: a left traveling rarefaction wave, a right traveling contact discontinuity, and a right traveling shock, which is further modulated by the gravitational source term. The implicit scheme demonstrates robust unconditional stability, successfully handling CFL numbers up to $10$ without blowing up. However, the zoomed in region in $\rho$ reveals that increasing the CFL number introduces temporal diffusion. Furthermore, as expected, minor dispersive oscillations begin to appear near the strong shock front and the contact wave at the CFL values greater than $2$, though the overall scheme remains bounded and stable.

\medskip
\subsection{Test 15: Well-Balanced Property and Stationary Perturbations in the Euler Equations} \label{subsec:test15_euler}
This test verifies the exact well-balanced property of the proposed kinetic schemes for the Euler equations under a constant gravitational field as in \cref{subsec:test13_euler}. We evaluate both the exact preservation of an isothermal hydrostatic equilibrium and the scheme capability to maintain stability when this equilibrium is subjected to a localized density perturbation. Both scenarios use a computational domain $[a, b] = [-1, 1]$ discretized with a coarse mesh of $N_x = 50$ cells, $\text{CFL} = 0.9$, $\gamma = 1.4$, and free-flow boundary conditions.

\textbf{Scenario A: Exact Preservation of the Isothermal Equilibrium.}
In the first case, the initial condition is set exactly to the isothermal hydrostatic steady state \cite{XinShu2012}:
\begin{equation*}
    \rho(x,0) = \exp({-x}), \qquad q(x,0) = 0, \qquad E(x,0) = \frac{\rho(x,0) + 1}{\gamma - 1}.
\end{equation*}
As shown in \cref{tab:test15_exact}, all proposed schemes, explicit, implicit, and Semi-Lagrangian, preserve the steady state to machine precision.

\textbf{Scenario B: Density Perturbation over the Equilibrium.}
In the second case, we introduce a localized Gaussian perturbation exclusively to the density field, while maintaining the initial momentum and energy distributions of the background state:
\begin{equation*}
    \rho(x,0) = \exp({-x}) + 0.4 \exp\Big{(}-200x^2\Big{)}, \qquad q(x,0) = 0, \qquad E(x,0) = \frac{\exp({-x})}{\gamma - 1}.
\end{equation*}
The simulation is run for an extremely long time ($T = 2000$) to allow all transient waves to radiate out of the computational domain through the transmissive boundaries, enabling the system to relax back to its original equilibrium.
\cref{tab:test15_perturbed} displays the $L^1$ errors measured at the end of the simulation relative to the unperturbed background state. The errors remain strictly bounded near machine precision. This result conclusively demonstrates that the well-balanced schemes successfully handle strong deviations without amplifying small perturbations into spurious long term oscillations.

\begin{table}[!ht]
    \centering
    \begin{minipage}[t]{0.48\textwidth}
        \centering
        \resizebox{\textwidth}{!}{
            \begin{tabular}{lrrr}
\toprule
Scheme & Err $L_1$ ($\rho$) & Err $L_1$ ($q$) & Err $L_1$ ($E$) \\
\midrule
FV-HKR O1 Exp & 8.14e-16 & 4.57e-16 & 3.92e-15 \\
FV-HKR O2 Exp & 3.15e-15 & 1.08e-15 & 4.88e-15 \\
FV-HKR O3 Exp & 2.68e-15 & 2.13e-15 & 1.03e-14 \\
FV-HKR O1 Imp & 7.81e-16 & 4.88e-16 & 3.83e-15 \\
FV-HKR O2 Imp & 3.93e-16 & 6.88e-16 & 2.39e-15 \\
SL-HKR O1 & 3.71e-15 & 8.61e-16 & 9.87e-15 \\
\bottomrule
\end{tabular}

        }
        \makeatletter\def\@captype{table}\makeatother
        \caption{Test 15: $L^1$ errors for the exact stationary state ($T=1$).}
        \label{tab:test15_exact}
    \end{minipage}
    \hfill
    \begin{minipage}[t]{0.48\textwidth}
        \centering
        \resizebox{\textwidth}{!}{
            \begin{tabular}{lrrr}
\toprule
Scheme & Err $L_1$ ($\rho$) & Err $L_1$ ($q$) & Err $L_1$ ($E$) \\
\midrule
FV-HKR O1 Exp & 6.30e-16 & 5.15e-16 & 2.05e-15 \\
FV-HKR O2 Exp & 2.05e-13 & 1.16e-14 & 2.82e-13 \\
FV-HKR O3 Exp & 2.03e-13 & 1.19e-14 & 2.81e-13 \\
FV-HKR O1 Imp & 7.23e-15 & 5.72e-15 & 1.61e-14 \\
FV-HKR O2 Imp & 1.43e-14 & 4.34e-15 & 2.33e-14 \\
SL-HKR O1 & 2.01e-14 & 1.10e-14 & 3.35e-14 \\
\bottomrule
\end{tabular}

        }
        \makeatletter\def\@captype{table}\makeatother
        \caption{Test 15: $L^1$ errors after a density perturbation ($T=2000$).}
        \label{tab:test15_perturbed}
    \end{minipage}
\end{table}

\subsection{Test 16: Riemann Problem on Equilibrium Variables in the Euler Equations}\label{subsec:test16_euler}

To further demonstrate the robustness and exact well-balanced property of the proposed kinetic schemes, we simulate a Riemann problem constructed over two distinct isothermal hydrostatic equilibrium states. Unlike standard shock tube tests that use constant primitive variables, which immediately violate the hydrostatic balance and trigger artificial global fluid motion under a gravitational field, this setup ensures that regions unreached by the propagating waves remain perfectly unperturbed.

We consider the computational domain $[0, 1]$ subject to a linear gravitational potential $H(x) = x$. The simulation is performed using $N_x = 500$ cells, a heat capacity ratio $\gamma = 1.4$, and transmissive boundary conditions. The initial condition is defined by joining two separate hydrostatic equilibria at the discontinuity $x = 0.5$:
\begin{equation*}
    \rho(x, 0) = \begin{cases}
        \exp({-x}),       & \text{if } x < 0.5, \\
        0.125 \exp({-x}), & \text{otherwise},
    \end{cases}
    \qquad q(x, 0) = 0,
    \qquad E(x, 0) = \frac{\rho(x, 0)}{\gamma - 1}.
\end{equation*}
The simulation is run until a short final time $T = 0.1$. This time is deliberately chosen so that the acoustic and shock waves originating from the center do not reach the boundaries, leaving the fluid in the lateral intervals $[0, 0.1]$ and $[0.9, 1]$ completely immune to the central perturbation. In these outer regions, the fluid must remain in exact hydrostatic equilibrium with strictly zero velocity.

\cref{fig:test16_euler} and \cref{tab:test16_euler} present the numerical solutions and the error analysis for this test. The figure displays the macroscopic variables, highlighting the zero-velocity regions near the boundaries, while the zooms confirm the expected sharper shock-capturing capabilities of the higher-order methods in the central transition. The accompanying table reports the $L^1$ errors evaluated exclusively within the unperturbed intervals $[0, 0.1]$ and $[0.9, 1]$. The results confirm that all the proposed explicit and implicit schemes successfully prevent the generation of spurious waves, maintaining the background steady state within machine precision in the areas where the physical waves have not yet arrived.

\noindent
\begin{minipage}[c]{0.48\linewidth}
    \vspace{0pt}
    \centering
    % Recuerda cambiar la ruta por la de tu gráfica real
    \includegraphics[width=\linewidth]{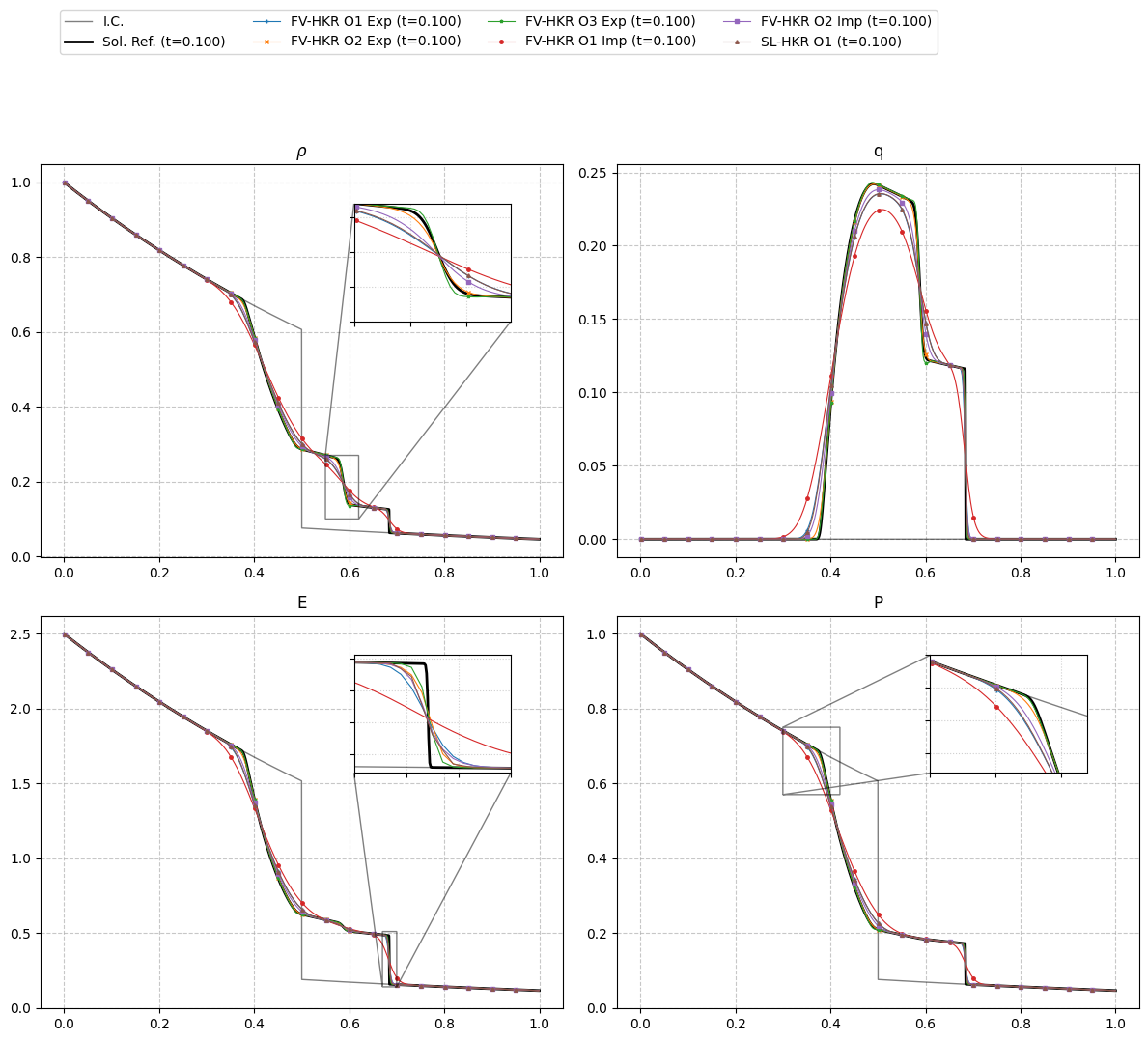}
    \captionof{figure}{Numerical solutions for Test 16 at $T=0.1$ under a linear gravitational potential comparing all schemes. $N_x=500$ and $\Delta x=0.002$.}
    \label{fig:test16_euler}
\end{minipage}%
\hfill
\begin{minipage}[c]{0.48\linewidth}
    \vspace{0pt}
    \centering
    \resizebox{\linewidth}{!}{
        % Recuerda cambiar la ruta por la de tu tabla real
        \begin{tabular}{lrrr}
\toprule
Scheme & Err $L_1$ ($\rho$) & Err $L_1$ ($q$) & Err $L_1$ ($E$) \\
\midrule
FV-HKR O1 Exp & 2.89e-17 & 1.05e-16 & 5.27e-17 \\
FV-HKR O2 Exp & 4.47e-17 & 1.28e-16 & 1.11e-16 \\
FV-HKR O3 Exp & 6.52e-16 & 3.28e-16 & 1.08e-15 \\
FV-HKR O1 Imp & 3.18e-16 & 3.73e-16 & 1.13e-15 \\
FV-HKR O2 Imp & 7.15e-17 & 1.49e-16 & 1.15e-16 \\
SL-HKR O1 & 1.71e-16 & 4.48e-16 & 6.13e-16 \\
\bottomrule
\end{tabular}

    }

    \captionof{table}{Test 16: $L^1$ errors for the exact stationary state measured exclusively in $[0,0.1]\cup[0.9,1]$.}
    \label{tab:test16_euler}
\end{minipage}
\vspace{1em}

\medskip
\subsection{\REVISORONE{Computational efficiency: comparison with standard FV-WB schemes}}
\label{computation_time}

\REVISORONE{
    As a last experiment, we compare the computational performance of the proposed FV-HKR schemes with that of standard exactly well-balanced finite volume discretizations applied directly to the macroscopic SWE system, namely the scheme \eqref{eq:FV_WB} with the same well-balanced reconstruction operators as the FV-HKR schemes, the local Lax-Friedrichs numerical flux, and the same explicit or implicit time integrators; they are labelled FV-O$k$-Exp and FV-O$k$-Imp, $k$ being the order.  The experiment is based on \cref{subsec:test11_swe} and is performed up to $T=0.3$ on meshes with
    $N_x\in\{50,100,200,400,800\}$.}

\REVISORONE{
    All the implicit schemes are run with $\mathrm{CFL}=1$. The errors are measured in the $L^1$ norm against a fine-grid reference solution computed with the third-order explicit FV-WB scheme, the LLF numerical flux, $N_x^{\mathrm{ref}}=12800$, and $\mathrm{CFL}=0.9$. The reference cell averages on each coarse mesh are obtained by averaging the corresponding fine grid values.}

\REVISORONE{
    No computational advantage should generally be expected from the explicit FV-HKR formulations with respect to their standard FV-WB counterparts. Indeed, the kinetic relaxation \MOD{doubles the number of evolved variables}, while the splitting compositions require several transport and relaxation substeps within each time step. This overhead is particularly relevant for higher-order methods, for which Strang or higher-order Suzuki-type compositions involve multiple operator evaluations.
    \MOD{We observed this in practice; we do not report these results and focus instead on the implicit schemes.}
}

\REVISORONE{
    Indeed, the situation is different for implicit discretizations. A fully implicit FV-WB discretization of the original macroscopic system requires the solution of a nonlinear coupled algebraic system at every time step, due to the nonlinear physical flux. In the HKR formulation, by contrast, the nonlinearities are moved to the local relaxation operator, with the kinetic transport equations remaining linear. Consequently, each implicit transport stage only requires the solution of linear systems.}

\REVISORONE{
    The nonlinear systems in the implicit FV-WB schemes are solved with Newton's method, with an exactly computed Jacobian matrix and a stopping criterion of $10^{-10}$ on the norm of the residual.
    The two decoupled tridiagonal linear systems
    arising from the kinetic relaxation are solved
    using a cached LU factorization.
    Both codes are written in NumPy and run on the same environment
    (a $12$-core Intel(R) Core(TM) \MOD{i7-1360P} processor and $32$ GiB of RAM\MOD{).}
    The reported wall-clock times correspond to the median of five complete simulations, after one preliminary warm-up run. Each measurement includes the initialization and the complete time integration from $t=0$ to $T$.
}

\begin{SCfigure}
    \centering
    \includegraphics[width=0.7\textwidth]
    {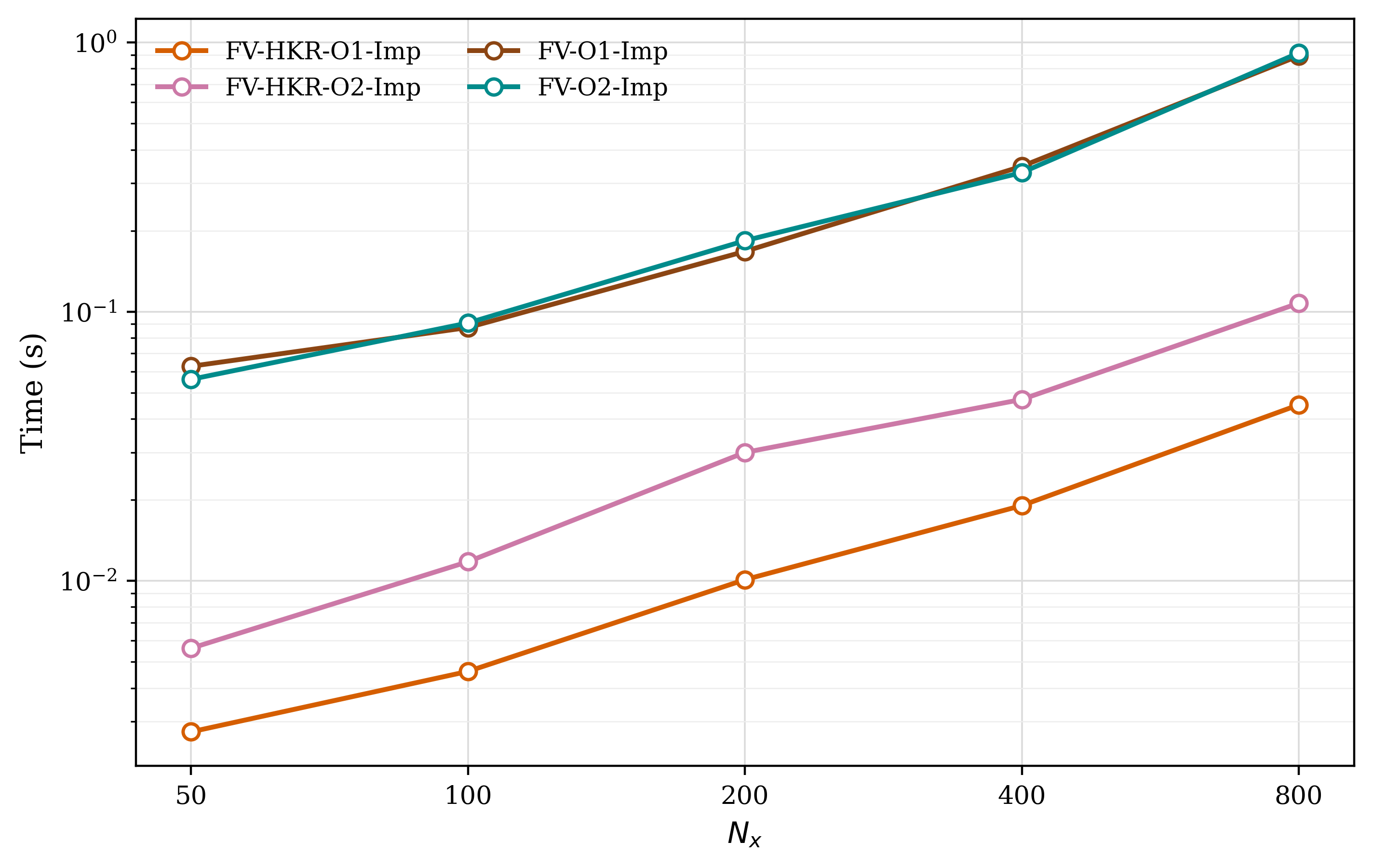}
    \caption{Computational cost for \cref{subsec:test11_swe} depending on the number of cells. The implicit schemes use $\mathrm{CFL}=1$.}
    \label{grafica coste computacional}
\end{SCfigure}

\REVISORONE{
    The computational times in \cref{grafica coste computacional} show a substantial reduction in cost for the kinetic formulation. At a fixed spatial resolution, FV-HKR-O1-Imp is approximately $20$ times faster than FV-O1-Imp, while FV-HKR-O2-Imp is \MOD{between $6$ and $10$ times} faster than FV-O2-Imp. These differences confirm a considerable computational advantage.
}

\begin{SCfigure}
    \centering
    \includegraphics[width=0.7\textwidth]
    {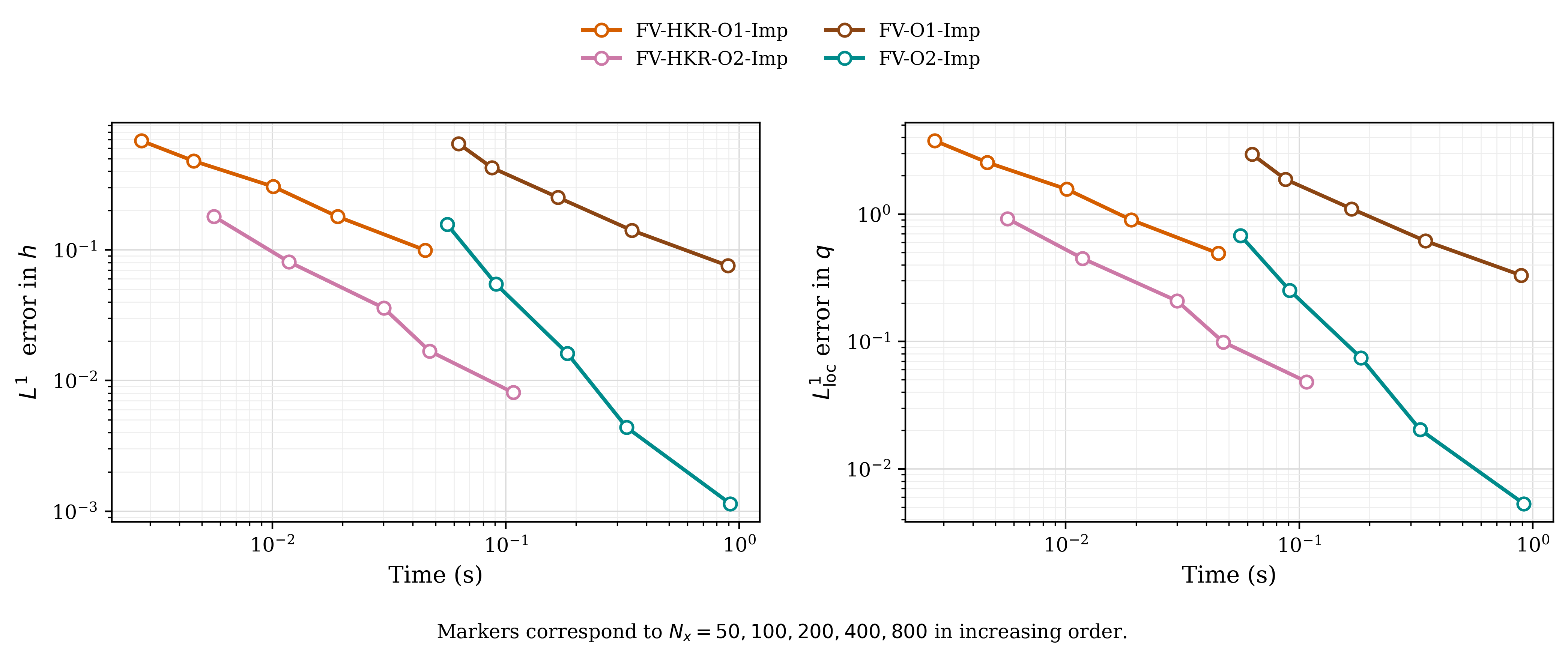}
    \caption{Error--cost efficiency for \cref{subsec:test11_swe}. The $L^1$ errors in the water height $h$ and discharge $q$ are plotted against the median wall-clock time. Along each curve, the markers correspond, from left to right, to $N_x=50,100,200,400$, and $800$.}
    \label{grafica eficiencia}
\end{SCfigure}

\REVISORONE{
    The error-cost diagrams in \cref{grafica eficiencia} show that the reduction in computational time translates into an improvement in efficiency. In particular, the FV-HKR curves are shifted significantly towards lower computational times. Nevertheless, at a fixed value of $N_x$, FV-HKR-O2-Imp produces larger errors than the standard second-order FV-WB scheme. This loss of accuracy can be attributed to the accumulation of numerical diffusion through the additional transport substeps involved in the kinetic splitting. The effect is especially visible for the second-order method, since the smaller truncation error makes the additional transport diffusion comparatively more apparent. Despite this increase in error at fixed resolution, the considerably lower computational cost allows the kinetic method to reach comparable accuracy \MOD{at a lower computational cost}.
}

\medskip
\REVISORONE{
    Finally, we examine the memory requirements of the implicit schemes during one
    complete time step. The memory cost must be separated into primary state
    storage and linear-algebra workspace. In the D1Q2 formulation, FV-HKR stores the
    macroscopic state $u$ together with the two kinetic variables $f_+$ and
    $f_-$. Since the three arrays have the same number of components, its primary
    state storage is exactly three times that of the standard FV-WB formulation,
    which stores only $u$. However, the total working memory also depends on the
    linear algebra operations arising from the implicit discretization. The optimized FV-HKR
    schemes store two LU factorizations of tridiagonal transport matrices of
    dimension $N_x+1$, whereas each Newton iteration of the standard schemes
    constructs and factorizes a coupled Jacobian of dimension $2(N_x+2)$. When
    evaluating the peak memory, successive Newton factorizations are not added,
    since they are not stored simultaneously.}

\begin{table}[!ht]
    \centering
    \begin{tabular}{rcccc}
        \hline
        $N_x$ & FV-HKR-O1-Imp & FV-HKR-O2-Imp & FV-O1-Imp & FV-O2-Imp \\
        \hline
        50 & 0.074 & 0.081 & 0.100 & 0.109 \\
        100 & 0.113 & 0.124 & 0.148 & 0.165 \\
        200 & 0.189 & 0.209 & 0.247 & 0.280 \\
        400 & 0.342 & 0.380 & 0.448 & 0.512 \\
        800 & 0.647 & 0.722 & 0.851 & 0.979 \\
        \hline
    \end{tabular}
    \caption{Peak Python-traced memory allocation (MiB) during one complete implicit time step of \cref{subsec:test11_swe}. Each value is the median of three measurements performed in independent processes.}
    \label{tab:one_step_memory}
\end{table}

\begin{SCfigure}
    \centering
    \includegraphics[width=0.7\textwidth]
    {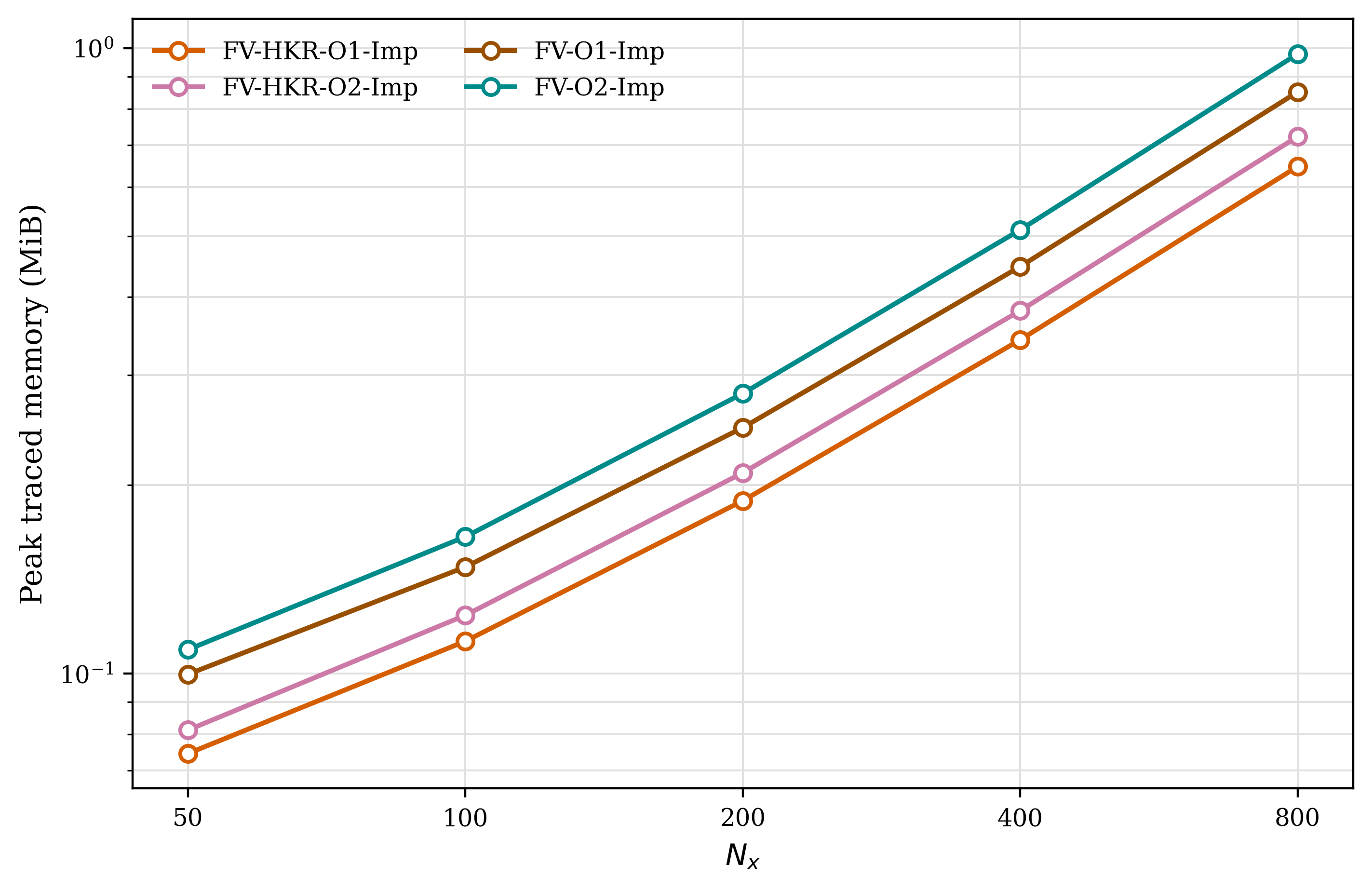}
    \caption{Peak memory allocation, traced with the Python module \texttt{tracemalloc}, during one complete implicit
        time step of \cref{subsec:test11_swe} depending on the number of
        cells $N_x$, using logarithmic scales on both axes.}
    \label{fig:one_step_memory}
\end{SCfigure}

\REVISORONE{
    The results in \cref{tab:one_step_memory,fig:one_step_memory}
    show an essentially linear growth of the memory requirements with $N_x$ for
    all four schemes. Despite the threefold increase in primary state storage, the
    FV-HKR schemes require less total working memory. Relative to their standard FV-WB counterparts, the
    peak allocation is reduced by approximately $25\%$ at first and
    second order. This reduction results from the smaller dimension
    and tridiagonal structure of the kinetic transport matrices, whose two cached
    LU factorizations require less memory than a single factorization of the
    coupled macroscopic Jacobian. The second order schemes exhibit a moderately
    larger peak than their first-order counterparts because of the additional
    reconstruction arrays and the two transport stages of the Strang splitting.
    Thus, the D1Q2 kinetic representation is less efficient when only state storage
    is considered, but its complete one-step working memory cost is lower in this
    experiment.
}

\section{Conclusions}
\label{sec:conclusions}

In this paper, we presented a novel procedure to make vectorial kinetic relaxation schemes well-balanced for hyperbolic systems with source terms. The key idea is to construct a kinetic representation that exactly captures the equilibrium variables, which are the natural variables in which the steady states are expressed. By doing so, we designed both a finite volume and a semi-Lagrangian scheme that are provably well-balanced for any balance law. We conducted numerical experiments on Burgers' equation, the shallow water equations with bathymetry, and the Euler equations with gravitational potential. They confirmed the theoretical properties and showcased the high-order accuracy and robustness of the schemes in various scenarios, while maintaining the well-balanced property to machine precision, even for moving steady solutions in the SWE.

\REVISORONE{Concerning the extension to several space dimensions, the kinetic side is direct: with a D2Q4 representation, each kinetic variable is transported along a single direction, so that no directional splitting is needed. The proofs of \cref{teorema WB FV-HKR,teorema WB FV-HKR imp,teorema WB SL-HKR} also naturally extend, as they are local to each cell, and so they carry over to structured and unstructured meshes as soon as a multidimensional exactly well-balanced reconstruction operator is available~\cite{castro2020well}. This is the case for families of equilibria given in closed form (lake at rest, hydrostatic equilibria). The main difficulty of such an extension concerns moving equilibria, which are characterized in two dimensions by a system of PDEs rather than ODEs, so that no local parametrized family is available in general. Hence, approximate strategies such as~\cite{gomez2021collocation,gomez2021high,FraMicNav2024} would be required. We have restricted this paper to one space dimension in order to isolate the new ingredient, the embedding of the macroscopic equilibrium into the kinetic relaxation. Future work will focus on extending the schemes to larger dimensions, as well as to \MOD{simulations} involving dry areas \MOD{and vacuum}.}

\section*{Acknowledgments}
This work is supported by projects PCI2024-155061-2 and PID2022-137637NB-C21 funded by MCIN/\\AEI/10.13039/501100011033/ and ERDF A way of making Europe, and also by the Grant PPRO-FQM216-G-2023 funded by Consejería de Universidad, Investigación e Innovación and by ERDF Andalusia Program 2021-2027.
L. Ávila León is supported by the grant
PRE2022-000022 funded by MCIN/AEI/10.13039/501100011033 and FSE+.
This work was also funded by the Fog Research Institute under contract no.~FRI-454.  Manuel J. Castro thanks IMUS-Maria de Maeztu grant CEX2024-001517-M - Apoyo a Unidades de Excelencia María de Maeztu for supporting this research, funded by MICIU/\allowbreak AEI/\allowbreak 10.13039/\allowbreak 501100011033. The authors thank the support of the Centre National de Recherche Scientifique, International Research Project hyPerbolIC models, numerical AnalysiS and Scientific cOmputation (CNRS, IRP PICASSO).

\section*{Declaration of generative AI and AI-assisted technologies in the manuscript preparation process}
During the preparation of this work, the authors used Claude (Anthropic) as a writing assistant, in order to accelerate the process. The authors reviewed and edited the output as needed and take full responsibility for the content of the published article.

\bibliographystyle{plain}
\bibliography{references}

\begin{thebibliography}{10}

\bibitem{AbgLiu2024}
R.~Abgrall and Y.~Liu.
\newblock {A New Approach for Designing Well-Balanced Schemes for the Shallow Water Equations: A Combination of Conservative and Primitive Formulations}.
\newblock {\em SIAM J. Sci. Comput.}, 46(6):A3375--A3400, 2024.

\bibitem{AuPe00}
D.~Aregba-Driollet and R.~Natalini.
\newblock {Discrete Kinetic Schemes for Multidimensional Systems of Conservation Laws}.
\newblock {\em SIAM J. Numer. Anal.}, 37(6):1973--2004, 2000.

\bibitem{AudBouBriKlePer2004}
E.~Audusse, F.~Bouchut, M.-O. Bristeau, R.~Klein, and B.~Perthame.
\newblock A fast and stable well-balanced scheme with hydrostatic reconstruction for shallow water flows.
\newblock {\em SIAM J. Sci. Comput.}, 25(6):2050--2065, 2004.

\bibitem{BejiBattjes1993}
S.~Beji and J.~A. Battjes.
\newblock Experimental investigation of wave propagation over a bar.
\newblock {\em Coastal Eng.}, 19(2):151--162, 1993.

\bibitem{Bellotti2023}
Thomas Bellotti.
\newblock Truncation errors and modified equations for the lattice {B}oltzmann method via the corresponding {F}inite {D}ifference schemes.
\newblock {\em ESAIM: Mathematical Modelling and Numerical Analysis}, 57(3):1225--1255, 2023.

\bibitem{BerVaz1994}
A.~Berm{\'u}dez and M.~E. V{\'a}zquez.
\newblock Upwind methods for hyperbolic conservation laws with source terms.
\newblock {\em Comput. Fluids}, 23(8):1049--1071, 1994.

\bibitem{BerBulFouMbaMic2022}
C.~Berthon, S.~Bulteau, F.~Foucher, M.~M{\textquotesingle}Baye, and V.~Michel-Dansac.
\newblock {A Very Easy High-Order Well-Balanced Reconstruction for Hyperbolic Systems with Source Terms}.
\newblock {\em SIAM J. Sci. Comput.}, 44(4):A2506--A2535, 2022.

\bibitem{BerMicTho2026}
C.~Berthon, V.~Michel-Dansac, and A.~Thomann.
\newblock Towards a fully well-balanced and entropy-stable scheme for the {E}uler equations with gravity: {P}reserving isentropic steady solutions.
\newblock {\em Math. Comp.}, 95:1251--1292, 2026.

\bibitem{BerCha2016}
Christophe Berthon and Christophe Chalons.
\newblock A fully well-balanced, positive and entropy-satisfying {G}odunov-type method for the shallow-water equations.
\newblock {\em Mathematics of Computation}, 85(299):1281--1307, 2016.

\bibitem{Bou1999}
F.~Bouchut.
\newblock {Construction of BGK Models with a Family of Kinetic Entropies for a Given System of Conservation Laws}.
\newblock {\em J. Stat. Phys.}, 95(1/2):113--170, 1999.

\bibitem{Bouchut2004}
Fran{\c{c}}ois Bouchut.
\newblock {\em Nonlinear Stability of Finite Volume Methods for Hyperbolic Conservation Laws and Well-Balanced Schemes for Sources}.
\newblock Frontiers in Mathematics. Birkh{\"a}user, Basel, 2004.

\bibitem{BriXin2020}
J.~Britton and Y.~Xing.
\newblock {High Order Still-Water and Moving-Water Equilibria Preserving Discontinuous Galerkin Methods for the Ripa Model}.
\newblock {\em J. Sci. Comput.}, 82(2):30, 2020.

\bibitem{castro2008well}
M.~Castro, J.~M. Gallardo, J.~A. L{\'o}pez-Garc{\'{\i}}a, and C.~Par{\'e}s.
\newblock Well-balanced high order extensions of {G}odunov's method for semilinear balance laws.
\newblock {\em SIAM J. Numer. Anal.}, 46(2):1012--1039, 2008.

\bibitem{castro2020well}
M.~J. Castro and C.~Par{\'{e}}s.
\newblock {Well-Balanced High-Order Finite Volume Methods for Systems of Balance Laws}.
\newblock {\em J. Sci. Comput.}, 82(2), 2020.

\bibitem{CheLevLiu1994}
G.-Q. Chen, C.~D. Levermore, and T.-P. Liu.
\newblock Hyperbolic conservation laws with stiff relaxation terms and entropy.
\newblock {\em Comm. Pure Appl. Math.}, 47(6):787--830, 1994.

\bibitem{coulette2019high}
D.~Coulette, E.~Franck, Ph. Helluy, M.~Mehrenberger, and L.~Navoret.
\newblock High-order implicit palindromic discontinuous {G}alerkin method for kinetic-relaxation approximation.
\newblock {\em Comput. Fluids}, 190:485--502, 2019.

\bibitem{CraPupSemVis2017}
I.~Cravero, G.~Puppo, M.~Semplice, and G.~Visconti.
\newblock {CWENO}: {U}niformly accurate reconstructions for balance laws.
\newblock {\em Math. Comput.}, 87(312):1689--1719, 2017.

\bibitem{Dellar2013}
Paul~J. Dellar.
\newblock An interpretation and derivation of the lattice {B}oltzmann method using {S}trang splitting.
\newblock {\em Computers \& Mathematics with Applications}, 65(2):129--141, 2013.

\bibitem{DesMas2021}
V.~Desveaux and A.~Masset.
\newblock A fully well-balanced scheme for shallow water equations with {C}oriolis force.
\newblock {\em Commun. Math. Sci.}, 20(7):1875--1900, 2022.

\bibitem{DumZanGabPes2024}
M.~Dumbser, O.~Zanotti, E.~Gaburro, and I.~Peshkov.
\newblock {A well-balanced discontinuous Galerkin method for the first-order Z4 formulation of the Einstein-Euler system}.
\newblock {\em J. Comput. Phys.}, 504:112875, 2024.

\bibitem{FraMen2016}
E.~Franck and L.~S. Mendoza.
\newblock {Finite Volume Scheme with Local High Order Discretization of the Hydrostatic Equilibrium for the Euler Equations with External Forces}.
\newblock {\em J. Sci. Comput.}, 69(1):314--354, 2016.

\bibitem{FraMicNav2024}
E.~Franck, V.~Michel-Dansac, and L.~Navoret.
\newblock {Approximately well-balanced Discontinuous Galerkin methods using bases enriched with Physics-Informed Neural Networks}.
\newblock {\em J. Comput. Phys.}, 512:113144, 2024.

\bibitem{GalParCas2007}
J.~M. Gallardo, C.~Par{\'e}s, and M.~Castro.
\newblock On a well-balanced high-order finite volume scheme for shallow water equations with topography and dry areas.
\newblock {\em J. Comput. Phys.}, 227(1):574--601, 2007.

\bibitem{GerHelMic2022}
P.~Gerhard, Ph. Helluy, and V.~Michel-Dansac.
\newblock Unconditionally stable and parallel {D}iscontinuous {G}alerkin solver.
\newblock {\em Comput. Math. Appl.}, 112:116--137, 2022.

\bibitem{GerHelMicWeb2024}
P.~Gerhard, Ph. Helluy, V.~Michel-Dansac, and B.~Weber.
\newblock {Parallel Kinetic Schemes for Conservation Laws, with Large Time Steps}.
\newblock {\em J. Sci. Comput.}, 99(1), 2024.

\bibitem{gomez2021implicit}
I.~G{\'{o}}mez-Bueno, S.~Boscarino, M.~J. Castro, C.~Par{\'{e}}s, and G.~Russo.
\newblock Implicit and semi-implicit well-balanced finite-volume methods for systems of balance laws.
\newblock {\em Appl. Numer. Math.}, 184:18--48, 2023.

\bibitem{gomez2021high}
I.~G{\'{o}}mez-Bueno, M.~J. Castro, and C.~Par{\'{e}}s.
\newblock High-order well-balanced methods for systems of balance laws: a control-based approach.
\newblock {\em Appl. Math. Comput.}, 394:125820, 2021.

\bibitem{gomez2021collocation}
I.~G{\'{o}}mez-Bueno, M.~J.~Castro D{\'{\i}}az, C.~Par{\'{e}}s, and G.~Russo.
\newblock {Collocation Methods for High-Order Well-Balanced Methods for Systems of Balance Laws}.
\newblock {\em Mathematics}, 9(15):1799, 2021.

\bibitem{Gosse2001}
Laurent Gosse.
\newblock A well-balanced scheme using non-conservative products designed for hyperbolic systems of conservation laws with source terms.
\newblock {\em Mathematical Models and Methods in Applied Sciences}, 11(2):339--365, 2001.

\bibitem{GoutalMaurel1997}
N.~Goutal and F.~Maurel.
\newblock Proceedings of the 2nd workshop on dam-break wave simulation.
\newblock In {\em Department Laboratoire National d'Hydraulique}, Paris, 1997. Electricit\'{e} de France.

\bibitem{GreLeR1996}
J.~M. Greenberg and A.-Y. LeRoux.
\newblock A well-balanced scheme for the numerical processing of source terms in hyperbolic equations.
\newblock {\em SIAM J. Numer. Anal.}, 33(1):1--16, 1996.

\bibitem{gomez2021well}
I.~Gómez-Bueno, M.~J. Castro, and C.~Parés.
\newblock {\em {Well-Balanced Reconstruction Operator for Systems of Balance Laws: Numerical Implementation}}, pages 57--77.
\newblock Springer International Publishing, 2021.

\bibitem{holden2010splitting}
H.~Holden, K.~H. Karlsen, K.-A. Lie, and N.~H. Risebro.
\newblock {\em {Splitting Methods for Partial Differential Equations with Rough Solutions: Analysis and MATLAB programs}}.
\newblock EMS Press, 2010.

\bibitem{Jin1999}
Shi Jin.
\newblock Efficient asymptotic-preserving ({AP}) schemes for some multiscale kinetic equations.
\newblock {\em SIAM Journal on Scientific Computing}, 21(2):441--454, 1999.

\bibitem{Jin2012}
Shi Jin.
\newblock Asymptotic preserving ({AP}) schemes for multiscale kinetic and hyperbolic equations: a review.
\newblock {\em Rivista di Matematica della Universit{\`a} di Parma}, 3(2):177--216, 2012.

\bibitem{JinXin1995}
Shi Jin and Zhouping Xin.
\newblock The relaxation schemes for systems of conservation laws in arbitrary space dimensions.
\newblock {\em Communications on Pure and Applied Mathematics}, 48(3):235--276, 1995.

\bibitem{KazParRic2025}
M.~Kazolea, C.~Parés, and M.~Ricchiuto.
\newblock {Approximate well-balanced WENO finite difference schemes using a global-flux quadrature method with multi-step ODE integrator weights}.
\newblock {\em Comput. Fluids}, 296:106646, 2025.

\bibitem{LeVeque1998}
R.~J. LeVeque.
\newblock Balancing source terms and flux gradients in high-resolution {G}odunov methods: {T}he quasi-steady wave-propagation algorithm.
\newblock {\em J. Comput. Phys.}, 146(2):346--365, 1998.

\bibitem{McLQui2002}
R.~I. McLachlan and G.~R.~W. Quispel.
\newblock Splitting methods.
\newblock {\em Acta Numer.}, 11:341--434, 2002.

\bibitem{MicBerClaFou2016}
V.~Michel-Dansac, C.~Berthon, S.~Clain, and F.~Foucher.
\newblock A well-balanced scheme for the shallow-water equations with topography.
\newblock {\em Comput. Math. Appl.}, 72(3):568--593, 2016.

\bibitem{MicBerClaFou2017}
V.~Michel-Dansac, C.~Berthon, S.~Clain, and F.~Foucher.
\newblock A well-balanced scheme for the shallow-water equations with topography or {M}anning friction.
\newblock {\em J. Comput. Phys.}, 335:115--154, 2017.

\bibitem{MicTho2025}
V.~Michel-Dansac and A.~Thomann.
\newblock Towards a fully well-balanced and entropy-stable scheme for the {E}uler equations with gravity: {G}eneral equations of state.
\newblock {\em Comput. Fluids}, 303:106853, 2025.

\bibitem{PerSim2001}
Beno{\^i}t Perthame and Chiara Simeoni.
\newblock A kinetic scheme for the {S}aint-{V}enant system with a source term.
\newblock {\em Calcolo}, 38(4):201--231, 2001.

\bibitem{QiaHumLal1992}
Y.~H. Qian, D.~d'Humi{\`e}res, and P.~Lallemand.
\newblock {Lattice BGK Models for Navier-Stokes Equation}.
\newblock {\em Europhys. Lett.}, 17(6):479--484, 1992.

\bibitem{SemTraPup2021}
M.~Semplice, E.~Travaglia, and G.~Puppo.
\newblock {One- and Multi-dimensional CWENOZ Reconstructions for Implementing Boundary Conditions Without Ghost Cells}.
\newblock {\em Commun. Appl. Math. Comput.}, 5(1):143--169, 2021.

\bibitem{ShuOsher1989}
C.-W. Shu and S.~Osher.
\newblock Efficient implementation of essentially non-oscillatory shock-capturing schemes, {II}.
\newblock {\em J. Comput. Phys.}, 83(1):32--78, 1989.

\bibitem{Sod1978}
G.~A. Sod.
\newblock A survey of several finite difference methods for systems of nonlinear hyperbolic conservation laws.
\newblock {\em J. Comput. Phys.}, 27(1):1--31, 1978.

\bibitem{strang1968construction}
G.~Strang.
\newblock On the construction and comparison of difference schemes.
\newblock {\em SIAM J. Numer. Anal.}, 5:506--517, 1968.

\bibitem{Lee1979}
B.~{van Leer}.
\newblock {T}owards the {U}ltimate {C}onservative {D}ifference {S}cheme, {V}. {A S}econd {O}rder {S}equel to {G}odunov's {M}ethod.
\newblock {\em J. Comput. Phys.}, 32:101--136, 1979.

\bibitem{WangShuYeeSjogreen2009}
W.~Wang, C.-W. Shu, H.~C. Yee, and B.~Sjögreen.
\newblock High-order well-balanced schemes and applications to non-equilibrium flow.
\newblock {\em J. Comput. Phys.}, 228(18):6682--6702, 2009.

\bibitem{XingShu2005}
Y.~Xing and C.-W. Shu.
\newblock High order finite difference {WENO} schemes with the exact conservation property for the shallow water equations.
\newblock {\em J. Comput. Phys.}, 208(1):206--227, 2005.

\bibitem{XinShu2012}
Y.~Xing and C.-W. Shu.
\newblock {High Order Well-Balanced WENO Scheme for the Gas Dynamics Equations Under Gravitational Fields}.
\newblock {\em J. Sci. Comput.}, 54(2--3):645--662, 2012.

\bibitem{XinShuNoe2011}
Y.~Xing, C.-W. Shu, and S.~Noelle.
\newblock On the advantage of well-balanced schemes for moving-water equilibria of the shallow water equations.
\newblock {\em J. Sci. Comput.}, 48(1-3):339--349, 2011.

\bibitem{XinZhaShu2010}
Y.~Xing, X.~Zhang, and C.-W. Shu.
\newblock Positivity-preserving high order well-balanced discontinuous {G}alerkin methods for the shallow water equations.
\newblock {\em Adv. Water Resour.}, 33(12):1476--1493, 2010.

\bibitem{Xu2002}
Kun Xu.
\newblock A well-balanced gas-kinetic scheme for the shallow-water equations with source terms.
\newblock {\em Journal of Computational Physics}, 178(2):533--562, 2002.

\end{thebibliography}

\end{document}